\documentclass[11pt]{amsart}

\usepackage[pdfauthor={Congling Qiiu}, 
        pdftitle={???}%
      dvips,colorlinks=true]{hyperref}
\hypersetup{
  bookmarksnumbered=true,
  citecolor=black,
  pagecolor=black, 
  urlcolor=black,  
}

\usepackage[utf8]{inputenc}

\usepackage{xtab}
\usepackage{amsmath, amsthm, amssymb, cancel,verbatim}
\usepackage[all]{xy}
\usepackage[scr=euler,scrscaled=.96]{mathalfa}
 \usepackage{ esint }
\usepackage{enumerate}
\usepackage{mathtools,booktabs}
\usepackage{amscd}
\usepackage{bbm, stmaryrd}
\usepackage{yfonts}
\usepackage{color}

\usepackage{epstopdf} 
\usepackage{booktabs}

\newtheorem{teo}{Theorem}[subsection]
\newtheorem{thm}[teo]{Theorem}

\newtheorem{eg}[teo]{Example}
\newtheorem{lem}[teo]{Lemma}
\newtheorem{cor}[teo]{Corollary}
\newtheorem{defn}[teo]{Definition}

\newtheorem{rmk}[teo]{Remark}

\newtheorem{asmp}[teo]{Assumption}

  \newtheorem{prop}[teo]{Proposition}
    \newtheorem {conj}[teo]{Conjecture}

\newtheorem{thmdefn}[teo]{Theorem/Definition}

\numberwithin{equation}{section}

  \newcommand{\BA}{{\mathbb {A}}} \newcommand{\BB}{{\mathbb {B}}}
    \newcommand{\BC}{{\mathbb {C}}} 
    \newcommand{\BE}{{\mathbb {E}}} \newcommand{\BF}{{\mathbb {F}}}

     \newcommand{\BL}{{\mathbb {L}}}
    \newcommand{\BM}{{\mathbb {M}}} 
     \newcommand{\BP}{{\mathbb {P}}}
    \newcommand{\BQ}{{\mathbb {Q}}} \newcommand{\BR}{{\mathbb {R}}}
    \newcommand{\BS}{{\mathbb {S}}}

     \newcommand{\BZ}{{\mathbb {Z}}}

    \newcommand{\CA}{{\mathcal {A}}} 
    \newcommand{\CC}{{\mathcal {C}}} \newcommand{\cD}{{\mathcal {D}}}
    \newcommand{\CE}{{\mathcal {E}}} \newcommand{\CF}{{\mathcal {F}}}
    \newcommand{\CG}{{\mathcal {G}}} \newcommand{\CH}{{\mathcal {H}}}
     
     \newcommand{\CL}{{\mathcal {L}}}
    \newcommand{\CM}{{\mathcal {M}}} \newcommand{\CN}{{\mathcal {N}}}
    \newcommand{\CO}{{\mathcal {O}}} \newcommand{\CP}{{\mathcal {P}}}
     
    \newcommand{\CS}{{\mathcal {S}}} \newcommand{\CT}{{\mathcal {T}}}
     \newcommand{\CV}{{\mathcal {V}}}
     
     \newcommand{\CZ}{{\mathcal {Z}}}

    \newcommand{\RM}{{\mathrm {M}}}

     \newcommand{\fp}{{\mathfrak{p}}}

    \newcommand{\pair}[1]{\langle {#1} \rangle}

    \newcommand{\incl}{\hookrightarrow}

    \newcommand{\bsl}{\backslash}

 \newcommand{\vep}{\varepsilon} \newcommand{\ep}{\epsilon}
  \newcommand{\vpl}{\varprojlim}
             
 \newcommand{\vil}{\varinjlim}  
 
\newcommand{\Neron}{N\'{e}ron}

    \newcommand{\etale}{\'{e}tale~}

  \newcommand{\bc}{{\mathrm{bc}}}
 
 \newcommand{\ab}{{\mathrm{ab}}}
    \newcommand{\ad}{{\mathrm{ad}}}\newcommand{\an}{{\mathrm{an}}}
    
    \newcommand{\Aut}{{\mathrm{Aut}}}
    
    \newcommand{\Art}{{\mathrm{Art}}}

    \newcommand{\Cl}{{\mathrm{Cl}}}

    \newcommand{\coker}{{\mathrm{coker}}}\newcommand{\cusp}{{\mathfrak{cusp}}}

    \newcommand{\Div}{{\mathrm{Div}}} 
    \newcommand{\End}{{\mathrm{End}}} \newcommand{\Eis}{{\mathrm{Eis}}} 
  \newcommand{\Det}{{\mathrm{Det}}}
    \newcommand{\Ell}{{\mathrm{Ell}}}
    \newcommand{\Fr}{{\mathrm{Fr}}}
    \newcommand{\Frob}{{\mathrm{Frob}}}

    \newcommand{\Gal}{{\mathrm{Gal}}} \newcommand{\GL}{{\mathrm{GL}}}

    \newcommand{\Hom}{{\mathrm{Hom}}}
    
    \newcommand{\id}{{\mathrm{id}}}
    
   \newcommand{\Ram}{{\mathrm{Ram}}}

    \newcommand{\inv}{{\mathrm{inv}}}

    \newcommand{\NT}{{\mathrm{NT}}} 
    
    \newcommand{\ord}{{\mathrm{ord}}} \newcommand{\rank}{{\mathrm{rank}}}
    \newcommand{\PGL}{{\mathrm{PGL}}} \newcommand{\Pic}{\mathrm{Pic}}
    \newcommand{\pr}{{\mathrm{pr}}}    
            
     \newcommand{\proj}{{\mathrm{proj}}}      

    \renewcommand{\mod}{\mathrm{mod}\ }\renewcommand{\Re}{{\mathrm{Re}}}

    \newcommand{\red}{{\mathrm{red}}}
    \newcommand{\reg}{{\mathrm{reg}}}\newcommand{\Res}{{\mathrm{Res}}}

    \newcommand{\sep}{{\mathrm{sep}}}     \newcommand{\LC}{{\mathrm{LC}}}
    \newcommand{\supp}{{\mathrm{supp}}}

    \newcommand{\sign}{{\mathrm{sign}}}
    \newcommand{\SL}{{\mathrm{SL}}}
    \newcommand{\Spec}{{\mathrm{Spec{\ } }}} \newcommand{\Spf}{{\mathrm{Spf}}}  
 
    \newcommand{\Sp}{{\mathrm{Sp}}}

    \newcommand{\tr}{{\mathrm{tr}}}
    
    \newcommand{\ur}{{\mathrm{ur}}}
    \newcommand{\Vol}{{\mathrm{Vol}}}

  \newcommand{\ssp}{{\mathrm{sp}}}

      \newcommand{\zero}{{\mathfrak{zero}}}

    \newcommand{\Nm}{{\mathrm{Nm}}}

 \newcommand{\sing}{{\mathrm{sing}}}

  \newcommand{\ave}{{\mathrm{ave}}}

\newcommand\supervisor[1]{\def\@supervisor{#1}}

\newcounter{elno}

\renewcommand{\cong}{\simeq}   

\begin{document}
\title{The Gross-Zagier-Zhang formula over function fields}
\author{Congling Qiu}
\maketitle
 \begin{abstract}

We prove the Gross-Zagier-Zhang formula over  global function fields of arbitrary characteristics. 
It is an explicit formula which relates the   \Neron-Tate  heights  of  CM points on  abelian varieties and 
   central derivatives of      associated quadratic base change $L$-functions.  
   Our proof is based on an  arithmetic variant of a relative trace identity of Jacquet.
 This   approach is proposed by W. Zhang. We  apply our results to the Birch   and Swinnerton-Dyer conjecture for abelian varieties of $\GL_2$-type. In particular, we prove the conjecture for elliptic curves of  analytic rank 1.     \end{abstract}
\tableofcontents

  \section{Introduction}\label{Introduction}

  \subsection{Motivation}\label{Mot}
For an elliptic curve $A$
defined over a global field $F$, there are two associated objects  of fundamental importance. First, its set of rational points
$ A(F)$, which form  a finitely generated abelian
group. Second, its  $L$-function $L(s,A)$. 
The conjecture of Birch and Swinnerton-Dyer  asserts that 
$$\rank_\BZ A(F)=\ord _{s=1}L(s,A).$$
 Then it is natural to ask the following question:
if    $  \ord _{s=1}L(s,A)=1 $,  how  to  find a non-torsion point?

In the case that   $F=\BQ$, Gross and Zagier \cite{GZ} established a formula which relates  the   \Neron-Tate height of  a Heegner point on   $A$  associated to an imaginary quadratic field $E$, and    the derivative 
   $L'(1,A_E)$.
      In particular, if $\ord _{s=1}L(1,A_E)=1$, then the Heegner point is non-torsion.
 Over global  function fields of odd characteristics, an analog of the Gross-Zagier formula  was established by R\"{u}ck and Tipp \cite{RT}.

 The work of  Gross and Zagier  was  first generalized by S. Zhang \cite{Zha1} \cite{Zha01} to Shimura curves over totally real number field.
 Later the  general form of the Gross-Zagier-Zhang  formula was obtained  by Yuan, W. Zhang and S. Zhang  \cite{YZZ} in the automorphic framework.  
This general   formula  applies to arbitrary  cuspidal automorphic representations of $\GL_2$ over totally real fields holomorphic of weight 2. 

    In this paper,    over global function fields of arbitrary characteristics,   we  fully generalize the analog of the Gross-Zagier  formula 
in  the format of  \cite{YZZ}.       Our result applies to    arbitrary cuspidal automorphic representations  of $\GL_{2}$.
Besides, we prove the   Waldspurger formula   \cite{Wal}. 

As an application of  
our  Gross-Zagier-Zhang formula, we prove  the   Birch and Swinnerton-Dyer  conjecture for elliptic curves  of analytic rank 1   in arbitrary positive characteristics. Indeed, this is a special case  of one of our results  for  abelian varieties of $\GL_2$-type.  The  Birch and Swinnerton-Dyer  conjecture in the analytic rank 1 case in  characteristic $>3$ was considered by Ulmer \cite{Ulm} \cite{Ulm1}.

To prove our Gross-Zagier-Zhang formula and Waldspurger formula,   we employ  the relative trace formulas of Jacquet  \cite{Jac86}\ \cite{Jac87}\cite{JN} and 
 their  arithmetic analogs following   W. Zhang \cite{Zha12}.    The main advantage of  the relative trace formula approach is  its applicability in arbitrary characteristics.    Over global function fields of odd characteristics,  it is also possible to  approach our results using the theta lifting machinery.  For example, Chuang and   Wei  \cite{CW} proved  the   Waldspurger formula    in odd characteristics following Waldspurger's original approach closely. A key ingredient is the  Siegel-Weil formula   \cite{WFT}. Similarly,   it is also possible to prove the Gross-Zagier-Zhang formula   in odd characteristics
 following   \cite{YZZ}.

  Below in this introduction, we first state our main results. Then we describe the relative trace formula strategy and  
  the structure of this paper. 
 Finally, we  discuss  higher dimensional  generalizations of the Gross-Zagier formula via the relative trace formula approach.  

\subsection{Main results}

\subsubsection{Modular curves}
Let $X$ be a smooth proper connected curve over a finite field  $\BF_q$ of  characteristic $p>0$.
Let $F$ be the function field of $X$.
Let  $|X|$ be the set of closed points of $X$, equivalently the set of places of $F$.
 Let $\BA_F$   be the ring of adeles of $F$.
  Let $\BB$ be an incoherent quaternion algebra over $\BA_F$, i.e. 
   the set $\Ram\subset |X|$    of  ramified places of $\BB$  is finite and of odd cardinality. 
Distinguish a place $\infty\in \Ram$, which plays the role of the infinite place as in the number field case. 
Fix a ``level structure" at $\infty$, which is given by  an  arbitrary  open normal subgroup $U_\infty\subset \BB_\infty^\times$ containing a uniformizer of $F_\infty$.
In  \ref{Moduli spaces of},   we define  a  smooth proper modular curve
  $M_I$  over $F$ \cite{DriEll1} \cite{DriEll2} \cite{LLS} for  every finite and nonempty closed subscheme  $I\subset |X|-\{\infty\}$.
 
Let $M$ be the procurve $\vpl_I M_I$  where the transition maps are natural projections. Then $M$  is endowed with the right action of 
$\BB ^\times$. Let $J_I$ be the Jacobian of $M_I$, and let $J(F^\sep)_\BQ=\vpl _I J_I(F^\sep)_\BQ$, where $F^\sep$ is a separable closure of $F$.
Using  the normalized Hodge divisor class (see \ref{Hodge classes})  on  each $ M_I$,
we have a  map  \begin{equation}M(F^\sep)\to J(F^\sep)_\BQ  .\label{MJ}\end{equation} 

\subsubsection{Abelian varieties and $L$-functions}\label{Abelian varieties parametrized by modular curvesint}

Let $A$ be a simple abelian variety over $F$. Assume that $A$ is modular (w.r.t. $\BB/U_\infty$)  in the sense that  $$\pi=\pi_A:=\vil \Hom(J_I,A)_\BQ$$ is nontrivial.
Then $\pi$  is  an irreducible $\BQ$-coefficient representation of $\BB^\times$. 
Let $K:=\End(A)_\BQ$ which is a number field and acts on $\pi$. 
Then  $L(s,\pi,\ad)$ a polynomial on $q^{-s}$ with coefficients in $K$.  
 Let $E$ be a quadratic field extension of $F$. Let $\Omega$ be  a continuous character of $\Gal(E^\ab/E)$ valued in   a finite field extension $K'$  of $K$.  Also regard $\Omega$ as a Hecke character of $E^\times$ via the reciprocity map
 $\BA_E^\times/E^\times\to \Gal(E^\ab/E).$ Then the $L$-function  $$L(1/2,\pi,\Omega):=L(1/2,\pi_{E}\otimes \Omega).$$
 is a polynomial on $q^{-s}$ with coefficients in $K'$.
 The  twisted   $L$-function  $L(s,A_E,\Omega)$ of $A_E$   satisfies
  $$L(s,A_E,\Omega)=L(s-1/2,\pi , \Omega ).$$

 \subsubsection{Global and local periods} 
  Let $E/F$ be nonsplit at $\infty$.
 Fix an embedding \begin{equation}e_0:E  \incl \BB  \label{e0}\end{equation} of $F$-algebras. Let $P_0\in M^{E^\times}( F^\sep)$ be a CM point.
 Then $P_0$ is  defined over $E^{\ab}$,  the maximal abelian extension of $E$ in $F^\sep$ (see \ref{CM theory}). 
Regard $P_0$ as a point in $J(E^\ab)_\BQ$ via \eqref{MJ}. 
For $\phi\in \pi_A$, we have a CM point $\phi(P_0)\in A(E^\ab)_\BQ.$
Define 
$$P_\Omega(\phi):=\int_{\Gal(E^\ab/E)}\phi(P_0)^\tau\otimes_K\Omega(\tau)d\tau\in A(E^\ab)_\BQ\otimes_KK'$$
 where the Haar measure on $\Gal(E^\ab/E)$ is of  total volume 1.  This  integral  is essentially a finite sum.
For  $\varphi\in \pi_{A^\vee}$, our global   height period of $\phi$ and $\varphi$ is 
  $$ \pair{P_\Omega(\phi),P_{\Omega^{-1}}(\varphi)}_\NT^{K'}.$$ 
 Here   $\pair{\cdot,\cdot}_\NT^{K'}$ is the natural  extension by scalar of 
 $K$-bilinear \Neron-Tate  height pairing   
  $$\pair{\cdot,\cdot}_\NT^{K}:A(F^\sep)_\BQ\otimes_K A^\vee(F^\sep)_\BQ \to K\otimes \BC$$ such that 
  $\tr_{K\otimes_\BQ \BC/\BC}\pair{\cdot,\cdot}_\NT^K$ is the usual 
\Neron-Tate  height pairing.

 Now we define local periods. 
 Regard $E^\times/F^\times$ as an algebraic group over $F$. 
We always use the Tamagawa measure on $(E^\times/F^\times) (\BA_F)$, and fix a local-global decomposition of the measure.
 Let $(\cdot,\cdot)_v$ be the 
natural  pairing on   $\pi_v\otimes \tilde \pi_v$ which takes values in $K$.
For $  \phi_v\in   \pi_v$ and $   \varphi_v\in  \tilde \pi_v$,  
 define  the local period to be $$\alpha_{\pi_v} (\phi_v,\varphi_v) := \int_{E_v^\times/F_v^\times} (\pi_v(t)\phi_v,\varphi_v)_v\Omega_v(t)d t\in K' .$$ 
   This  integral  is essentially a finite sum.
 Let $\eta$ be the quadratic Hecke character of $F^\times$ associated to the quadratic extension $E/F$.
 Define the normalized local period to be 
 \begin{equation*} \alpha_{\pi_v}^\sharp:=\frac{L(1,\eta_v)L(1,\pi_v,\ad)}{L(2,1_{F_v})L(1/2,\pi_v,\Omega_v)} \alpha_{\pi_v}.
 \end{equation*}
 Then $\alpha_{\pi_v}^\sharp$ takes values in $K'$. Let \begin{equation}\alpha_{\pi_A}:=\bigotimes_{v\in |X|}\alpha_{\pi_v}^\sharp.\label{aglobal} \end{equation}

    \begin{thm}[The Gross-Zagier-Zhang formula]\label{GZ} 
 Let $\phi \in \pi_A,\varphi\in \pi_{A^\vee}$, then
 \begin{equation}\pair{P_\Omega(\phi),P_{\Omega^{-1}}(\varphi)}_\NT^{L'}=\frac{L(2,1_F)L'(1/2,\pi_{A  },\Omega)}{4L(1,\eta)^2 L(1,\pi_A,\ad)} \alpha_{\pi_A} (\phi ,\varphi ) \label{GZeq}
 \end{equation}
as an identity in $K'\otimes_\BQ\BC$.  
\end{thm}

Here, we use the following   pairing \eqref{ duality pairing} to identify $\pi_{A^\vee}$ with the contragradient  $\tilde\pi$ of $\pi $.
For $\phi\in \Hom(J_I,A^\vee)$, let $\phi^\vee \in \Hom(A,J_I^\vee)$ be the dual morphism. Regard $\phi^\vee$ as in   $\Hom(A,J_I)$ by canonically identifying 
$J_I$ and $ J_I^\vee$.  We have a perfect  $\BB^\times$-invariant pairing
\begin{equation}(\cdot,\cdot) : \pi_A
 \times\pi_{A^\vee}\to \End(A)_\BQ = K\label{ duality pairing}
 \end{equation}  defined by     $$(\phi_1,\phi_2):=\Vol(M_I)^{-1} \phi_{1,I}\circ \phi_{2,I}^\vee.$$
 Here $\Vol(M_I) $ is   the degree of the Hodge class on $M_I$.

\subsubsection{Generality of Theorem \ref{GZ}}\label{choices} 
  
    Let $A$ be an  abelian variety     over $F$  of strict $\GL_2$-type   which does not have potential good reduction \footnote{ Potential good reduction includes good reduction.} at  $\infty$. The latter condition ensures that the corresponding automorphic representation   is a discrete series at $\infty$ (so comes from $\BB^\times_\infty$ via the Jacquet-Langlands correspondence). Since $U_\infty $ contains a uniformizer of $F_\infty$,    
     for $A$ to be modular (w.r.t.  $\BB/U_\infty$),    the value of the corresponding automorphic representation on this uniformizer  should be identity. We twist   $A$  by a suitable character   so that it  is modular. 
Then
Theorem \ref{GZ} can be applied to   $A$ thanks to the freedom on $\Omega$ in  Theorem \ref{GZ}.

From the representation-theoretical perspective, our  result applies to    arbitrary cuspidal automorphic representations  of $\GL_{2}$, see \ref{choices'}.

\subsubsection{The Birch and Swinnerton-Dyer  Conjecture}

  A direct consequence  of Theorem \ref{GZ}  is   a twisted Birch and Swinnerton-Dyer  Conjecture   for $A_E$ in the analytic rank 1 case, see Theorem \ref{BSDTE}.
  
With some more effort, we can remove the base change.
\begin{thm}
  The Birch and Swinnerton-Dyer conjecture holds for   elliptic curves   over $F$ in the analytic rank 1 case.
\end{thm} 
Indeed,   this theorem is a special case of  Theorem \ref{BSDT}, which applies to $A$ as in \eqref{choices}.
 The condition that the  automorphic representation corresponding to $A$  is a discrete series at $\infty$ is also necessary for
 the existence of a quadratic  base change  $A_E$  which  has analytic rank 1 (see Remark \ref{notgod}), so that we can apply Theorem \ref{GZ}.  
 
 \subsubsection{Waldspurger formula over function fields}
\label{The global relative trace formulaintro}
Let $B$ be a   quaternion algebra over $F$, and let $E$ 
be a quadratic extension of $F$ embedded in $B$ (not necessary nonsplit over $\infty$).
Let $\pi$ be a  cuspidal automorphic representation of $B^\times $ and  $\Omega$    a Hecke character of $E^\times$ (both are $\BC$-valued).
    Let $\alpha_\pi$ be defined as in \eqref
{aglobal}.
For $\phi\in \pi$, define the toric period of $\phi$ by  
\begin{equation}\label{toric period}P_\Omega(\phi):=\int_{
E^\times\bsl \BA_E^\times /\BA_F^\times}\phi(t)\Omega(t)dt.
\end{equation}

\begin{thm}[The Waldspurger formula]\label{The Waldspurger formula over function fieldsintro}   
Let $\phi \in \pi ,\varphi\in \tilde \pi $, then
 \begin{equation} P_\Omega(\phi) P_{\Omega^{-1}}(\varphi) =\frac{L(2,1_F)L(1/2,\pi,\Omega)}{ 2 L(1,\pi ,\ad)}  \alpha_{\pi } (\phi ,\varphi ). \label{Waldeq}\end{equation}

   \end{thm}
    
    \subsection{Relative trace formula strategy}    
   
    Instead of studying periods directly,   relative trace formulas treat   distributions on adelic groups related to the periods. We will compare 
     relative trace formulas for unit groups of quaternion algebras and $\GL_{2,E}$.
\subsubsection{Distribution version of  Theorem \ref{The Waldspurger formula over function fieldsintro}}  

 The Hecke action  of $f\in  C_c^\infty(B^\times(\BA_F))$ on $\pi$ 
is \begin{equation*}\pi(f)\phi:=\int_{  B ^\times(\BA_F)} f (g)\pi(g)\phi dg. \end{equation*}
Define  a distribution on $ B^\times(\BA_F)$  by assigning to $f\in C_c^\infty(B^\times(\BA_F))$ the value
 \begin{equation*}\CO_\pi(f):=\sum_{\phi }P_\Omega(\pi(f)\phi)P_{\Omega^{-1}}( \tilde \phi), 
 \end{equation*}
where the sum is over an orthonormal  basis $\{\phi\}$ of $\pi$,  and $\{\tilde \phi \}$ is the dual basis of $\tilde \pi$. 
For $f_v\in C_c^\infty(B_v^\times)$, define $$\rho_{\pi_v}(f_v):=\sum_{\phi }\pi_v(f_v)\phi \otimes \tilde \phi,$$  where 
  the sum is over an orthonormal   basis $\{\phi\}$ of $\pi_v$, $\{\tilde \phi \}$ is the dual basis of $\tilde \pi_v$,
  and 
  $$\pi_v(f_v)\phi :=\int_{  B ^\times(F_v) } f_v (g)\pi_v(g)\phi   dg.$$ 
 Abusing notation, we use $\alpha^\sharp_{\pi_v}$ to denote 
the    distribution on $ B_v^\times$  which assigns to   $f_v $ the value  \begin{equation}\alpha^\sharp_{\pi_v}( f_v) :=\alpha^\sharp _{\pi_v}(\rho_{\pi_v}(f_v)).\label{alpha}\end{equation}

 Let $\omega$ be the restriction of $\Omega$ to  $\BA_F^\times/F^\times$. 
 Assume that the  central character of $\pi$ is $\omega^{-1}$, otherwise both sides of \eqref{Waldeq} are 0.
Then for every $v\in |X|$, $ \vep(1/2,\pi_{v,E_v}\otimes \Omega_v)=\pm1$, and it is 1 for all but finitely many places  (see \cite{Tun}). 
Define   \begin{equation}\Sigma(\pi,\Omega):=\{v\in |X|:\vep(1/2,\pi_{v,E_v}\otimes \Omega_v)\neq \Omega_v(-1) \}.\label{sigset} \end{equation}
 Let $\Ram(B)$ be the ramification set of $B$.  
   Theorem \ref{The Waldspurger formula over function fieldsintro}  implies the following theorem.\begin{thm}\label{The Waldspurger formula over function fieldsintrof}\label{zhl}

 Assume that the  central character of $\pi$ is $\omega^{-1}$ 
 and  $\Ram(B)=\Sigma(\pi,\Omega)$.
 There exists $f=\bigotimes_{v\in |X| }f_v\in C_c^\infty(B^\times(\BA_F))$ such that 
 \begin{equation*}
2\CO_\pi (f )  =\frac{L(2,1_F)L(1/2,\pi, \Omega)}{ L(1,\pi,\ad)} \prod_{v\in |X| } \alpha_{\pi_v}^\sharp (f_v) 
\end{equation*} and $ \alpha_{\pi_v}^\sharp (f_v)\neq 0$
for  every  $v\in |X|$. 

  \end{thm}

 Let $\CP_{\Omega }(\pi ):=\Hom_{ \BA_E^\times }(\pi\otimes \Omega    ,\BC).$  Then  both  
$P_\Omega \otimes P_{\Omega^{-1}}  $ and $  \alpha_{\pi } $   are
  elements in $\CP_\Omega(\pi)\otimes  \CP_{\Omega^{-1}}(\tilde\pi)$. 
The theorem of Tunnell-Saito  \cite{Tun}\cite{Sai} says that   $\CP_{\Omega}(\pi)\neq \{0\}$  if and only if $\Ram(B)=\Sigma(\pi,\Omega)$ and  in this case,  $\CP_{\Omega}(\pi )$  is of 
dimension 1. 
 It follows that  Theorem \ref{zhl} implies
   Theorem \ref{The Waldspurger formula over function fieldsintro}.  

\subsubsection{Distribution version of  Theorem \ref{GZ}}\label{1.3.2} 
   Let  $\CH_{\BC} $ be the Hecke algebra of locally constant $\BC$-valued bi-$U_\infty$-invariant functions on $\BB ^\times$ with compact support modulo $U_\infty\times U_\infty$. 
Define a distribution on  $  \CH_{\BC}$ by  assigning to $f\in \CH_{\BC}$   the    height pairing  $H  (f)$   of the   images  of the CM cycle given by $e_0:E  \incl \BB  $ and its Hecke translation by $f$ in $J$,   twisted by $\Omega$ and $\Omega^{-1}$ respectively   (see Definition \ref{CM height}).
 Let $\pi$ be  an irreducible admissible representation  of $\BB^\times/U_\infty$ whose  Jacquet-Langlands correspondence to $\GL_{2,F} $ is cuspidal. Let 
$ H_\pi^\sharp (f )$ be the normalized $\pi$-component of $H  (f)$ (see \eqref{HHpi0}).
Let    $\alpha^\sharp_{\pi_v}( f_v)$ be defined as in  \eqref{alpha}, suitably modified when $v=\infty$ (see \ref{The CM height distribution}). 
 \begin{thm} [Theorem \ref{GZdis'}]\label{GZdis}
Assume that the  central character of $\pi$ is $\omega^{-1}$ 
 and  $\Ram=\Sigma(\pi,\Omega).$
 There exists $f=\bigotimes_{v\in |X| }f_v\in \CH_\BC$  such that  \begin{equation*}  H_\pi^\sharp(f )  =\frac{L(2,1_F)L'(1/2,\pi,\Omega)}{ L(1,\pi,\ad)} \prod_{v\in |X|} \alpha_{\pi_v}^\sharp (f_v),\end{equation*}
and $ \alpha_{\pi_v}^\sharp (f_v)\neq 0$
for every $v\in |X|$.

 \end{thm} 

\subsubsection{Distributions on $\GL_{2}(\BA_E)$}   
  Let $A$ be the diagonal torus  of $G=\GL_{2,E}$,  and let $Z$ be the center of   $G$. 
 Let $H\subset G$ be the similitude unitary group with respect to the Hermitian matrix
$\begin{bmatrix}0&1\\
 1&0\end{bmatrix}$
 and  $\kappa$   the associated similitude  character of $H$.
Let $\omega_E$ be the base change of $\Omega$ to $E$.

 Let   $f'\in  C_c^\infty(G(\BA_E))$. Consider the Hecke action of $f'$ on the space of automorphic forms on $ G(\BA_E)$ which transform
 by  $\omega_E^{-1}$ under the action of the center. The action of $f'$ is given by a  kernel  function  $K_{\omega_E,f'}$ on $$G(E)\bsl G(\BA_E)\times G(E)\bsl G(\BA_E).$$   
 Let $\sigma$ be an automorphic representation of $G$.
 Define  $\CO (s, f')$  (resp. $\CO_\sigma(s, f')$) to be  the   integral of $K_{\omega_E,f'}$  (resp. the  $\sigma$-component of $K_{\omega_E,f'}$)  
 on $$Z(\BA_E)A(E)\bsl A(\BA_E)\times Z(\BA_E)H(F)\bsl   H(\BA_F)$$ 
against the character $\Omega|\cdot |_E^s \boxtimes(\eta\cdot(\omega^{-1}\circ\kappa))  $.  
Let  $\CO (f')=\CO (0,f')$  (resp. $\CO_\sigma( f')=\CO_\sigma(0, f')$)

\subsubsection{Theorem \ref{The Waldspurger formula over function fieldsintrof}}\label{skw} 
 Let $\sigma=\pi_E$.    For  functions $f\in C_c^\infty(B^\times(\BA_F))$ and $f'\in C_c^\infty(G(\BA_E))$ with matching local orbital integrals,   $\CO(f') $ equals its counterpart $\CO(f)$ for $B^\times$.
 (see \eqref{RTF}).
The spectral decomposition gives  an equation relating $\CO_\pi(f)$ and $\CO_\sigma(f') $.   
 
 In the number field case   \cite{JN}, by  the smooth matching  for $f'$,   the function $f'$    can be arbitrary (but not  $f$).
If $ L(1/2,\pi,\Omega)\neq 0$, one can choose $f'$ such that $ \CO_\sigma(f') $ is  nonzero.
Moreover, under the condition $ L(1/2,\pi,\Omega)\neq 0$, the local factors of   $\CO_\sigma$ and $ \alpha_{\pi_v}^\sharp$'s can be compared, which is called the spherical character  identity.
 Thus  Theorem \ref{The Waldspurger formula over function fieldsintrof} follows in this case.

In our proof of Theorem \ref {The Waldspurger formula over function fieldsintrof},  we take explicit pairs of matching functions $f$ and $f'$ such that $ \alpha_{\pi_v}^\sharp (f_v)\neq 0$, and prove  the  spherical character identity for   $f_v$ and $f_v'$ (see \ref{the proof is more important for us}, \ref{Explicit computations   for smooth matching}).   Thus  Theorem \ref{The Waldspurger formula over function fieldsintrof}  holds  without   the condition $ L(1/2,\pi,\Omega)\neq 0$.

\subsubsection{Theorem \ref{GZdis}}
   Let $|X|_s$ be  the set of places of $F$ split in $E$, $\Xi_\infty:=F_\infty^\times\cap U_\infty$ and $S$  a  finite set of ``bad places".  We have an arithmetic relative trace   identity
(see Theorem \ref{jacrtf'}):  \begin{equation}2 H(f) =\CO' (0,f' )+\sum_{v\in S-|X|_s } \text{distributions on }B(v)^\times(\BA_F)
  \label{bbb}\end{equation} for    good matching functions $f$ and $f'$.
 Here $B(v)$ is a quaternion algebra over $F$.
An easy but essential observation is that
under the    condition $\Ram=\Sigma(\pi,\Omega),$ the $\pi$-component of second term on the right hand side is 0 by the theorem of Tunnell-Saito.
 Then Theorem \ref{GZdis}  follows from \eqref{bbb}  by taking explicit pairs of     matching functions $f$ and $f'$, such that $ \alpha_{\pi_v}^\sharp (f_v)\neq 0$
 and the spherical character identity holds for   $f$ and $f'$. 

 Now we sketch the  proof of    \eqref{bbb}. To compute $H(f)$, we apply  the theory of admissible pairing  \cite{Zha01} to an integral model $\CN$ of a modular curve over   a certain field extension of $E$.   A vanishing condition on the average of a  local component of $f$  makes the contribution in $H(f)$ from the Hodge class vanish.  Then   $H(f)$ equals the intersections of admissible extensions of CM points.
Decompose $H(f)$ into a sum of local intersection numbers $H(f)_v$, $v\in |X|$.
We call  the intersection number of horizontal divisors in $H(f)_v$ the $i$-part,
and the rest  the $j$-part.

For $v\in |X|_s$, let $f^v$ be regularly  supported.  Then the $i$-part in $H(f)_v $ vanishes.
The $j$-part in $H(f)_v $   comes from  intersections on $\CN$   of horizontal divisors with components in the special fiber with moduli interpretations.
The  integral Hecke actions on   these components are easy to understand. Let the average  of $f^v$  vanish. Then
the $j$-part in $H(f)_v $ vanishes.

 For $v\in |X|-|X|_s$, let $f^v$ be regularly  supported. We compute local intersections  on   the Lubin-Tate or Drinfeld uniformization spaces on which $B(v)^\times$ acts.
 Then  $ H(f)_v$ is decomposed into a sum   over regular $E^\times\times E^\times $-orbits in $B(v)^\times$.
 For each  regular orbit, we compare the local component  at $v$ of the corresponding summand with the local orbital integral of $f_v'$ at a matching orbit of $G$. 
 Outside a large enough finite set $S\subset |X|$, we prove  the  arithmetic fundamental lemma for the full
 spherical Hecke algebra. 
 For $v\in S-|X|_s$,   we  prove the   arithmetic smooth matching for both the $i$-part and the $j$-part. Then  \eqref{bbb} follows.
  \subsection{Structure of the paper}

 In Part 1, we define global objects which will be studied  locally in Part 2.
 We first define the modular curves and CM points in Section \ref{MCM}. In Section \ref{Height distributions, Rational Representations, and Abelian varieties}, we define the height distribution on $\BB^\times$, study its spectral decomposition, and    reduce 
the Gross-Zagier-Zhang formula to   the distribution version.
Finally in Section \ref{The automorphic distributions}, we define   automorphic distributions on $\GL_{2}$ and 
unit groups of quaternion algebras. (The automorphic distributions on  the quaternion groups
will be used to prove the Waldspurger  formula.)
We study their  orbital  and spectral decompositions. The orbital terms are decomposed into   local orbital integrals. 

   The most important section of Part 2  is  
Section \ref{Local intersection multiplicity}, where we  study the orbital decomposition of   the height distribution  into local intersection numbers. Besides, Section  \ref{review} and Section \ref{amafl} also concern the orbital side.
In Section  \ref{review}, we give 
 explicit  functions on  $\GL_2$ and the quaternion groups over local fields
 with matching orbital integrals. 
In Section \ref{amafl}, we compare the  derivatives  of local orbital integrals and the  local intersection numbers associated to the matching functions.  
For the  spectral sides, 
   we introduce  local   distributions for  $\GL_2$ and the quaternion groups in Section  \ref{local relative trace formula}, 
  compare  their values   at the matching functions in Section  \ref{review},
  and relate the local and  global distributions (as well as $L$-functions)  for  $\GL_2$ in Section \ref{Global and local periods}  (then the Waldspurger  formula  follows). 
  Finally in
    Section \ref{Proof of Theorem}, we   use the above ingredients to  prove the Gross-Zagier-Zhang formula.
  
  In part 3, we apply the Gross-Zagier-Zhang formula  to the Birch and Swinnerton-Dyer conjecture.
  \subsection{Related works in higher dimensions}\label{Remarks on related works}

\subsubsection{}  Using the geometrized  relative trace formula, Yun and W. Zhang   \cite{YZ} \cite{YZ2} proved an amazing formula relating higher
 derivatives of $L$-functions  and intersection numbers on the moduli stacks of shtukas  in odd characteristics. Their formula applies to representations with trivial central character and   Iwahori level-structures. Our modular curve is closely related to
 a special case of the moduli stack of shtukas when  $\BB$ is only ramified at one place.

We hope to  generalize this higher derivative formula to  more general representations, using the arithmetic relative trace formula as in this work. 
Indeed, one obstacle in   Yun and W. Zhang's geometric  approach is that they need to construct precise  matching functions. However, our choices of  matching functions  are more flexible, by allowing an extra  term (i.e., the last term in \eqref{bbb}).


 \subsubsection{} 
 For  groups  over number fields other than $\GL_2$,  
a    generalization of the Gross-Zagier formula was conjectured by Gan-Gross-Prasad \cite{GGP}  and refined by S. Zhang   \cite{Zha10}. 
 W. Zhang \cite{Zha12} proposed    the  arithmetic relative trace formula approach  to   attack this conjecture.
During the revision of this paper, W. Zhang \cite{Zha19} proved 
 the arithmetic fundamental lemma over $\BQ$ for large enough primes.
Cases of the arithmetic smooth matching     were proved by Rapoport, Smithling,  Terstiege and W. Zhang  
   \cite{RSZ15} \cite{RSZ16} \cite{RSZ17}.
 The  spherical character identity   over  $p$-adic fields, conjectured in \cite{Zha14b},  was  proved by 
 Beuzart-Plessis \cite{BP}\cite{BP1}.

\subsection{Notations}\label{measures}


The Haar measures on $\BA_F$, $\BA_E$, $\BA_F^\times$, $\BA_E^\times $,  $\GL_2(\BA_F)$, $G(\BA_E)$, $H(\BA_F)$  and  $\BB^\times$  takes values in $\BQ$ on open compact subgroups, and will be specified in Section \ref{notations and measures }.
 
 For  $S\subset |X|$  and a decomposable adelic object $Z$ over $\BA_F$, we use $Z_S$ (resp. $Z^S$) to denote the $S$-component (resp. component away from $S$) of $Z$.     
  Let  $\BA_{F,\mathrm{f}}  =\BA_F^{{\infty}}$   and  $\BB_{\mathrm{f}}=\BB^{{\infty}}$. 
  
We have defined  $\Xi_\infty =U_\infty \cap F_\infty^\times$ where $ F_\infty^\times$ is regarded as the center of $\BB_\infty^\times$.   For an open compact subgroup $U$ of $\BB_{\mathrm{f}}^\times$, let $\Xi_U:=U \cap \BA_{F,\mathrm{f}}^\times$ where $ \BA_{F,\mathrm{f}}^\times$ is regarded as the center of $\BB_{\mathrm{f}}^\times$.   Let $\Xi=\Xi_U\Xi_\infty\subset\BA_F^\times.$
 Let  $\tilde U:=UU_\infty\subset \BB^\times .$
 
  Let $l\neq p$ be a prime number.
Let  $\CA_{U_\infty} (\BB^\times,\bar \BQ_l)$ (resp. $\CA_{U_\infty} (\BB^\times,\BC)$) be the set of  isomorphism classes of $\bar \BQ_l$ (resp. $\BC$)-coefficient irreducible admissible representations   of $\BB^\times/U_\infty$, whose Jacquet-Langlands correspondence to $\GL_{2,F}$ is cuspidal.

 For a field extension $K/\BQ$, let 
$\CH_K$  be the Hecke algebra of 
$K$-valued locally constant   bi-$U_\infty$-invariant   functions on $\BB^\times  $ with compact  support modulo $U_\infty\times U_\infty$. 
For an open compact  subgroup  $U$ of $\BB^\times_{\mathrm{f}}$, let $\CH_{U ,K} \subset \CH_K$ be the subalgebra of bi-$\tilde U $-invariant functions.

    \section*{Acknowledgements}
       The author heartily thanks    Shouwu Zhang  for suggesting this problem, helpful discussions and  constant encouragement.
           The author would also like to thank Nick Katz, Yifeng Liu, Xinyi Yuan and  Wei Zhang for their useful suggestions, and    the referees for pointing out several inaccuracies   in the earlier versions of the manuscript.                       The author is grateful to the   Morningside Center of Mathematics at Chinese Academy of Sciences     and  the Institutes for Advanced Studies at Tsinghua University 
       for their hospitality and support during the  preparation of  part of this work. 
    The result of this paper was reported in  ``Workshop on arithmetic geometry, Tokyo-Princeton at Komaba" in March 2019.
 The author would like to thank the organizers of this workshop.
  The final revision of the paper was performed while the author was supported by  the NSF grant DMS-2000533.
\part{Global Theory}

\section{Modular curves and CM points}\label{MCM} We at  define  modular curves and review their properties, following Drinfeld \cite{DriEll1} \cite{DriEll2} \cite{DriCar},  Laumon, Rapoport, and Stuhler \cite{LLS}, et al.. 
Then we define CM points, and show their algebraicity.
 \subsection{Modular curves}\label{Moduli spaces of} 
 
 For a   sheaf  $\CF$ of $\CO_X$-modules and  $x\in |X|$, let $\CF_x$ be the completion of the stalk of $\CF$ at $x$. 
  An order of $D$ is  a locally free coherent  sheaf $\cD$ of $\CO_X$-algebras on $X$ whose stalk at the generic point is isomorphic to $D$. The set of local orders $\{\cD_x\subset D_x \}_{x\in |X|}$ satisfies the following property: there exists  an $F$-basis $R$ of $D$ such that $\cD_x=\CO_{F_x}R$ for almost all $x$.
Let $Ord$ be the set of  sets of local orders   satisfying this property.

  \begin{lem}[{\cite[Section 1] {LLS}}]\label{order} 
 The map $\cD\mapsto\{\cD_x\}_{x\in |X|} $
  is a bijection between the set orders of $D$ and the set $Ord$.
  \end{lem}
  
 An order  $\cD$ of $D$ is called  a maximal order of $D$  if
 $\cD_x$   is a maximal order of $D_x$ for every $x\in |X|$. 
Let $\cD$ be a  maximal order of $D$.  
 \subsubsection{Moduli spaces without level structures at $\infty$}

Let $S$ be an $\BF_q$-scheme. For a sheaf $\CE$ on $X\times S$, let ${}^\tau\CE:=(1\times \Frob_S)^*\CE$.
 \begin{defn}[{\cite[(2.2)]{LLS}}] \label{DES}
 A $\cD$-elliptic sheaf, on $X$  with respect to $\infty$, over $S$ is the following data:  a morphism $\zero_\BE:S\to X$ with image away from $\Ram$ and 
 a sequence  of commutative diagrams 
$$
    \xymatrix{
 {}^\tau \CE_{i-1} \ar[r]^{{}^\tau j_{i-1}}\ar[d]_{ t_{i-1}}  &{}^\tau \CE_{i} \ar[d]^{t_{i} } \\
  \CE_i \ar [r]_{j_i}   &    \CE_{i+1}   }
$$
indexed by $i\in\BZ$, 
 where each $\CE_i$ is a locally free $\CO_{X\times S}$-module  of rank 4 equipped with a \textit{right} action of $\cD$
compatible with the $\CO_{X}$ action such that $$\CE_{i+2\cdot\deg(\infty)}=\CE_i(\infty),$$ and $j_i,t_i$  are injections compatible with $\cD$-actions and   satisfy  the following conditions :
\begin{itemize} 
\item[(1)]    the composition $j_{i+2\cdot\deg(\infty)-1}\circ\cdot\cdot\cdot\circ j_{i+1}\circ j_i$ is the canonical inclusion $\CE_i\incl\CE_i(\infty)$;
\item[(2)] let $\pr_S:X\times S\to S$ be  the projection, then $\pr_{S,*}(\coker j_i)$ is a locally free $\CO_S$-module of rank $2$;
\item[(3)] $\coker t_i$ is the direct image of a locally free $\CO_S$-module of rank 2 by the graph morphism $(\zero_\BE,\id): S\to X\times S$.
 \end{itemize}  A morphism between two $\cD$-elliptic sheaves $\BE,\BF$ is a number $n\in \BZ$ and  a sequence of  morphisms $\phi_i:\CE_{i }\to\CF_{i+n}$ of right $\cD$-modules satisfying the  obvious compatibility with the other data. 

 \end{defn}
 
 For     a   sheaf  $\CF$ of $\CO_X$-modules and a  finite closed subscheme $I$ of $X-\{\infty\}$, 
 let $\CF|_{I}$ be the restriction of $\CF$ to $I$. 
   Let  $\hat\CO_F :=\prod_ {v\in |X|-\{\infty\}}\CO_{F_v}$.  Let $$\CF  \otimes\hat\CO_F:=\vpl_I \CF|_{I}=\prod_{v\in |X|-\{\infty\}}\CF_v$$ where the inverse limit is over all finite closed subschemes of $ X-\{\infty\} $.
 For a $\cD$-elliptic sheaf  $\BE$ such   $\zero_\BE(S)\cap I=\emptyset$, the restrictions of $\CE_i$ and $t_i$ to $I\times S$ are independent of $i$. Let  $\BE|_{I}$   and $t|_{I}$ be these restrictions.
Define a level-$I$-structure on $\BE$ 
to be  an isomorphism $$\kappa:\cD|_{I}\boxtimes\CO_S\cong \BE|_{I}$$ of right $\cD|_{I}\boxtimes\CO_S$-modules such that the following diagram is commutative:
$$ \xymatrix{ 
& \cD|_{I}\boxtimes\CO_S \ar[dr]^{\kappa} \ar[dl]_{{}^\tau\kappa}\\
 {}^\tau \BE|_{I}  \ar[rr]^{t|_{I}} & & \BE|_{I}.}$$

    
 Let $\Ell_I$ be the  set-valued functor $$S\mapsto \{\cD\mbox{-elliptic sheaves over }S \mbox{ with   level-}I \mbox{ structures}\}/\cong$$
on the category of  $\BF_q$-schemes. 
  Note that there is a morphism of functors $$ \Ell_I\to X-\Ram-I $$ 
by mapping a $\cD$-elliptic sheaf over $S$ to  $\zero_\BE$ (which is an $S$-point of $X-\Ram-I $).

 
  \begin{thmdefn}[{\cite{DriEll1}\cite[(4.1,5.1,6.1)]{LLS}}]\label{LLSsmooth}  
  Assume that $I$ is nonempty.
  
    (1) The functor $\Ell_I$ is represented by a smooth   $\BF_q$-scheme, which we denote  by  $\CM_I$. 
  
(2) The morphism $ \Ell_I\to X-\Ram-I $    is represented by a smooth morphism $ \CM_I\to X-\Ram$ of relative dimension 1 which factors through $X-\Ram-I$.  Moreover, if $D$ is a division algebra, the morphism $ \CM_I\to X-\Ram-I$ is proper.  

 \end{thmdefn}

 Define the modular curve $M_I$ to be the  smooth compactification of the  generic fiber of $\CM_I$.  
  The    smooth compactification is only  needed when $D$ is a matrix algebra.  The points in $M_I$ added by  the  smooth compactification  are called cusps of $M_I$.   
 There is a  right action of $(\cD\otimes \hat\CO_F)^\times $ on $M_I$ by acting on  level structures (extended to the compactification).
   For $J\supset I$, let $\pi_{J,I}:M_J\to M_I$ be the natural finite morphism which is \etale outside cusps.
   Define  an $F$-procurve   \begin{equation*}M:=\vpl_I M_I\end{equation*}  where the transition maps are $\pi_{J,I}$'s.   
If $D$ is a matrix algebra, points in $M$ whose images in $M_I$ are cusps are called cusps of $M$.

From now on, we only consider $\BE$ such that $\zero_\BE$ factors through the generic point of $X$. Let
 $\BE \otimes\hat\CO_F:=\vpl_I\CE_i |_{I},$  where the inverse limit is over all finite closed subschemes $I\subset X-\{\infty\} $.
This definition  is independent of $i$.  We have the induced morphism $$t  \otimes\hat\CO_F:=t _i \otimes\hat\CO_F:{}^\tau\BE \otimes\hat\CO_F\to \BE \otimes\hat\CO_F$$
which is independent of $i$. Define an infinite level structure on  $\BE$  to be an isomorphism
$$\kappa:(\cD \otimes \hat\CO_F)\boxtimes\CO_S\cong \BE \otimes\hat\CO_F  $$
of right $(\cD \otimes \hat\CO_F)\boxtimes\CO_S$-modules
such that the following diagram is commutative:
$$ \xymatrix{ 
& (\cD \otimes \hat\CO_F)\boxtimes\CO_S\ar[dr]^{\kappa} \ar[dl]_{{}^\tau\kappa}\\
 {}^\tau  \BE \otimes\hat\CO_F\ar[rr]^{t  \otimes\hat\CO_F} & &  \BE \otimes\hat\CO_F.}$$
Distinguish   the notion``infinite level structures" here and ``level structures at $\infty$" in \ref{LSI}.

The following lemma is easy to prove. \begin{lem}\label{infmod}   
The procurve $M$, excluding cusps if $D$ is a matrix algebra,  is the moduli space of $\cD$-elliptic sheaves over  $F$-schemes with  infinite level structures.
\end{lem}
In particular, there is a natural action of  $(\cD\otimes \hat\CO_F)^\times $ on $M$. 
    We summarize \cite[Section 5, D)]{DriEll1} \cite[Proposition 9.3]{DriEll1}  and \cite[(7.1)-(7.4)]{LLS} as follows.
\begin{prop}\label{BBaction}
 The action   of $(\cD\otimes \hat\CO_F)^\times $ on $M$   extends to a right action of $D^\times(\BA_{\mathrm{f}})$.
  \end{prop} Let us describe this construction since it will be referred below.
  
  \begin{proof} 
   Let $\BE=\{\CE_i : i\in \BZ\}$ be a $\cD$-elliptic sheaf with infinite level structure  $\kappa$. 
The construction is divided into two parts.  
First, let $g\in \BA_{\mathrm{f}} ^\times$,  which  corresponds to  a line bundle  $\CL$   on $X$ with ``infinite level structure". 
  The collection   $\{\CE_i\otimes \CL: i\in \BZ\}$ is naturally a $\cD$-elliptic sheaf with infinite level structure. 
This gives the action of $g$. Second, let   $g\in \BB_{\mathrm{f}}^\times  \cap \cD \otimes \hat\CO_F$.  Combined with $\kappa$, $g$ gives an endomorphism $[g]$ on $ \BE \otimes\hat\CO_F $. This endomorphism $[g]$ produces 
  another  $\cD$-elliptic sheaf $\BE'$ as follows. Define $\CE_i'$ by  the following cartesian diagram:
   $$
    \xymatrix{
\CE_i' \ar[r]^{ }\ar[d]_{ }  &\CE_i \otimes\hat\CO_F \ar[d]^{ [g] } \\
 \CE_i \ar [r]_{ }   &  \CE_i \otimes\hat\CO_F };
$$
the definitions of $t_i',j_i'$ are obvious.
The top morphism induces an isomorphism $$\alpha  :\BE' \otimes\hat\CO_F\cong\BE  \otimes\hat\CO_F .$$
 The level structure $\kappa'$  on $\BE'$  is defined to be the composition
of $\kappa$ and $\alpha^{-1}.$
This gives the action of $g$.  \end{proof}

   \subsubsection{Level structures at $\infty$}\label{LSI}
 
When $D=\RM_2$, Drinfeld \cite{DriEll2} \cite{DriCar} introduced   level structures of  elliptic sheaves   at $\infty$.  When $D$ is a division algebra, the definition of level structures of $\cD$-elliptic sheaves at $\infty$ is given in   \cite[Section 8]{LLS}, and depends on the choice of a uniformizer $\varpi_\infty$ of $F_\infty$, which we fix  from now on.
 We do not recall the definitions here, but only note that the   level structures at $\infty$ there  should be considered as  ``infinite level structures at $\infty$". 
 Let $\tilde\Ell_I$ be the  set-valued functor on the category of  $\BF_q$-schemes: $$S\mapsto \{\cD\mbox{-elliptic sheaves over }S \mbox{ with   level-}I \mbox{ structures and level structures at } \infty\}/\cong.$$

  \begin{thmdefn}[{\cite{DriEll2}\cite{DriCar}\cite[(8.10)]{LLS}}]\label{LLSsmooth'}    (1) The functor $\tilde\Ell_{I}$ is represented by an   $\BF_q$-scheme which  we denote   by  $\tilde\CM_I$. 
  
(2) The natural morphism $ \tilde\CM_I\to \CM_I $   is pro-finite pro-Galois      with Galois group $\BB_\infty^\times/\varpi_\infty^\BZ$.

 \end{thmdefn}  
 
 Let $U(I)$ be the principal congruence subgroup of level  $I$  in $(\cD\otimes \hat\CO_F)^\times$.
 For  $U_\infty\subset \BB^\times_\infty$   open   subgroup containing $\varpi_\infty^\BZ$, let  $ \CM_{U(I)U_\infty}$  be  the quotient of
 $\tilde \CM_I$ by $U_{\infty}$. 
 Then the morphism  $\CM_{U(I)U_\infty}\to \CM_I$  is finite \'etale. 
 In particular,  the generic fiber of $\CM_{U(I)U_\infty}$ is smooth.
Define the modular curve $ M_{U(I)U_\infty}$  to be the  smooth compactification of the generic fiber of $\CM_{U(I)U_\infty}$, and call the points   added by  the  smooth compactification  cusps. 
Define an $F$-procurve  $$  M_{U_\infty}:=\vpl_{I}  M_{U(I)U_\infty}.$$ 
Fix an isomorphism $ D^\times(\BA_{\mathrm{f}})\cong \BB_{\mathrm{f}}^\times  $.
Then $M_{U_\infty }$ is endowed with a right action of $ \BB^\times /U_\infty$,
lifting the one in Proposition \ref{BBaction} (see \cite[(8.7)]{LLS}).
  Define $$T_g:  M_{U_\infty}\to M_{ U_\infty}$$
  to  be the right action of $g\in \BB ^\times$.  
   For an open compact subgroup $U$ of $ (\cD\otimes \hat\CO_F)^\times\cong  \BB_{\mathrm{f}}^\times $, we always assume that $U$ is contained in the conjugation of a certain $U(I)$ with $I$ nonempty.    Let $$ M_{UU_\infty}=  M_{U_\infty}/U,$$ which is a smooth projective curve over $F$.    Points in $M_{UU_\infty}$ from cusps of $M_{U_\infty}$ are called cusps of $M_{UU_\infty}$. Let $\cusp$ be the closed subscheme of cusps of $M_{UU_\infty}$


    \subsubsection{Equivalence} \label{Equivalence} Let $\cD'$ be another maximal order of $D$. Then $\cD$ and $\cD'$ are locally isomorphic.
     By \cite[Proposition 3.1, Remarks 4.3 (g), Proposition 5.10]{Spi}, the moduli stacks of $\cD$-elliptic sheaves and $\cD'$-elliptic sheaves  are isomorphic. It is easy to verify that the isomorphism extends to moduli spaces with level structures (as  level structures are local properties), and are compatible with $\BB^\times$-actions.   
       In particular, 
     if $D=\RM_{2,F}$ is the matrix algebra, we take $\cD=\RM_{2}(\CO_X)$.  Then
      the moduli spaces  are isomorphic to  the ones considered by Drinfeld \cite{DriEll1} \cite{DriEll2} \cite{DriCar}. 
   
  \subsubsection{Decomposition of  cohomology}\label{LLSmain}
 
 Let 
 $$H^1(M_{U_\infty,F^\sep},\bar\BQ_l)=\vil_U  H^1(M_{UU_\infty, F^\sep},\bar\BQ_l) $$
where the limit is over open compact subgroups of $\BB_{\mathrm{f}}^\times$.
 For $g\in \BB ^\times$, let  $T_g^*
 $  be the pullback by  $T_g$.
 For  $f\in \CH_{\bar\BQ_l}$, let $f$ act on $H^1(M_{U_\infty, F^\sep},\bar\BQ_l)$ by
 $$T(f):=\int_{\BB^\times/\Xi_\infty } f(g) T_g^*.$$  
  If $f\in \CH_{U  ,\bar\BQ_l}$,  then
  $$T(f)(H^1(M_{ U_\infty,F^\sep},\bar\BQ_l))\subset  H^1(M_{ UU_\infty, F^\sep},\bar\BQ_l) .$$ 

    For $\pi\in \CA _{U_\infty}(\BB^\times,\bar\BQ_l)$, let $\LC(\pi)$ be the 
    unique irreducible   representation of $\Gal(F^\sep/F)$ over $\bar\BQ_l$ of dimension 2  such that $L(s,\pi)=L(s+1/2,\LC(\pi)).$
    It is  the (suitably normalized)
      Langlands correspondence    of the Jacquet-Langlands correspondence of $\pi$ to $\GL_{2,F}$.  
    Recall that the notion ``Langlands correspondence"   only means the local compatibility  of automorphic and Galois representations
    for almost all places and the existence for $\GL_{2,F}$ was  established by Drinfeld \cite{DriEll2}\cite{Dri3}.   The compatibility over all places under ``Langlands correspondence" is a theorem of L.Larfforgue  \cite[Corollaire VII.5]{LL}.    
    
  \begin{thm}[Drinfeld, Laumon-Rapoport-Stuhler]\label{semisimple}  Let 
 $U_\infty$ be an open normal subgroup of $\BB_\infty^\times$ containing $\varpi^\BZ$.
  There is an isomorphism  of $\BB^\times\times \Gal(F ^\sep/F )$-representations:
 $$H^1(M_{U_\infty, F^\sep},\bar\BQ_l) \cong 
 \bigoplus_{\pi\in \CA_{U_\infty}(\BB^\times,\bar\BQ_l)} \pi \boxtimes \LC(\pi).$$

 \end{thm}
 \begin{proof}If $D$ is the matrix algebra, the theorem follows from  \cite{DriEll2} by \ref{Equivalence}. If $D$
 is a division algebra, we use  \cite[(13.8)]{LLS}.
  The  functions satisfying the second condition in \cite[(13.8)]{LLS} are already used in   \cite[p. 166]{DriEll2}. 
 The first condition in \cite[(13.8)]{LLS} follows   from the second    by the Weyl integration formula \cite[(7.2.2)]{JL} (see \cite[Lemma 2]{Fli} for the argument). 
 \end{proof} 
 
  \subsubsection{Rigid analytic uniformization at $\infty$}\label{Rigid analytic uniformization}

 Let $ \Omega_\infty  $ be  Drinfeld's rigid analytic upper half plane  over $F_\infty$. 
 For an integer $n\geq 0$, let $\Sigma_n  $ be Drinfeld's $n$-th covering of   the base change of  $\Omega_\infty$ to 
 the (separable) unramified quadratic extension $F'_\infty$ of $F_\infty$
 (see \cite{Gen}).  
  Suitably choose the deformation  and level structure  data  defining  $\Omega_\infty$ and $\Sigma_n  $ such that
  they are equipped with  (necessary compatible) left actions of $\GL_2(F_\infty) $ and   right actions of $\BB_\infty^\times. $
    Moreover, we require that the  left action of $\GL_2(F_\infty) $ on $\Omega_\infty$ is the
 action  by fractional linear transformations.  
   Let  $D^\times$ act on $\Omega_\infty$ via an isomorphism $D_\infty \cong \RM_2(F_\infty)$ and  act on $\BB_{\mathrm{f}}^\times$ via the isomorphism $D^\times(\BA_{\mathrm{f}})\cong \BB_{\mathrm{f}} $

\begin{prop}  
  \label{riguni}

Let  $U_\infty\subset \BB^\times_\infty$ be generated by   $\varpi^\BZ_\infty$ and   the principal congruence subgroup of   level $n\geq 0$. For every open compact subgroup $U\subset \BB_{\mathrm{f}}^\times$, there are isomorphisms of 
rigid analytic space over $F_\infty$: 
 $$M_{U\BB^\times_\infty}^\an-\{\cusp\}\cong D^\times \bsl\Omega_\infty \times   \BB_{\mathrm{f}} ^\times  /U,
 $$  
 and
   $$M_{UU_\infty}^\an-\{\cusp\}\cong D^\times \bsl\Sigma_n \times   \BB_{\mathrm{f}} ^\times  /U
  $$   
   such that 
   the actions of $  \BB^\times$ on  the inverse systems  $( M_{U\BB^\times_\infty}^\an)_U$ and $( M_{UU_\infty}^\an)_U$ are compatible with
 the natural actions of $  \BB^\times$ on  the inverse systems  $(D^\times \bsl\Omega_\infty \times   \BB_{\mathrm{f}} ^\times  /U)_U$ and  $(D^\times \bsl\Sigma_n \times   \BB_{\mathrm{f}} ^\times  /U)_U.$   
 
  \end{prop} 

\begin{proof} When $\Ram=\{\infty\}$,  the proposition
 is proved in \cite {DriEll1}\cite{DriCar}. 
Let  $\Ram\neq \{\infty\}$. For $M_{U\BB_\infty^\times}$,  the isomorphism is given in   \cite[Theorem 4.4.11]{BS} and the compatibility holds.
Let $U_\infty\subset \BB^\times_\infty$ be generated by   $\varpi^\BZ_\infty$ and   the principal congruence subgroup of   level $n\geq 0$. The isomorphism for $M_{UU_\infty }$ is obtained as follows. 
 If there  exists $v\in \Ram-\{\infty\}$ such that $U_v$ is maximal,   apply  \cite[Proposition 4.28]{Spi}    to \cite[Theorem 8.3]{Hau} to get the isomorphism for $M_{UU_\infty }$. 
In general, let $v\in \Ram-\{\infty\}$, and let $U'=U^v\cD_v^\times$.  Then  
$$M_{UU_\infty}= M_{U'U_\infty}\times_{M_{U'\BB_\infty^\times}}M_{U\BB_\infty^\times}.$$
Then the isomorphism  for $M_{UU_\infty}$ is obtained from  isomorphisms for all three modular curves on the right hand side.   
 \end{proof}

\subsubsection{Redefine notations}\label{useU} 
 If not specified, let $U_\infty\subset \BB^\times_\infty$ be $ \BB^\times_\infty$ or be generated by   $\varpi^\BZ_\infty$ and   the principal   congruence subgroup of $\BB_\infty^\times$ of level $n>0$.
   We use the symbol $M$ for $M_{U_\infty}$ and $M_U$
 for $M_{UU_\infty}$. 
Let $\BC_\infty$ be the completion of the algebraic closure of $F_\infty$.
We use the symbol $[z,h]$ (resp. $[z,h]_U $), where $z\in \Omega_\infty $ or $\Sigma_n $    and $h\in \BB_{ \mathrm{f}}^\times $, to represent a point in $M    $ (resp. $M_U  $) via Proposition  \ref{riguni}. 
     \subsubsection{Jacobians and Height pairings} \label{Jacobians and Height pairings}

    We   define two Jacobians
 $J$ and $J^\vee$ as in \cite[3.1.6]{YZZ}. 
 For  an open compact subgroup ${U }$  of  $\BB ^\times$. 
 Let $J_U$ be the Jacobian variety of $M_U$.
 By   \cite[Proposition 6.9]{MFK}, there is a canonical isomorphism 
 \begin{equation}J_U\cong J_U^\vee.\label{GIT}\end{equation}
 \begin{defn}
  Let  $J $ be the inverse system $(J_U)_U$ where the transition morphisms are  induced by pushforwards of divisors,  $J^\vee $ be the direct system   $(J_U)_U$ where the transition morphisms are  induced by  pullbacks of  line bundles.   For a field extension $F'/F$ and a $\BZ$-algebra $R$, 
 let $$J(F')_R:=\vpl_U J_U(F') \otimes_\BZ R=\vpl_U\Cl^0(M_{U,F'}) \otimes_\BZ R,$$
 $$ J^\vee(F')_R:=\vil _UJ_U^\vee(F')\otimes_\BZ R=\vil_U\Pic^0(M_{U,F'}) \otimes_\BZ R  .$$
 If $R=\BZ$, the subscript $R$ is omitted.
 Let $$\Pic(M\times M)_R:=\vil _U\Pic(M_U\times M_U)\otimes _\BZ R,$$
 where the transition morphisms are     pullbacks of  line bundles. 
   Define  \begin{equation}\Hom(J,J^\vee)_R:=\vil_U \Hom(J_U,J_U^\vee )_R,\label{HJJ}\end{equation}
   where 
   the transition map for $U'\subset U$ is 
 $\phi\mapsto \pi_{U',U}^*\circ \phi\circ\pi_{U',U,*}.$
  \end{defn}
 By \cite[Lemma 3.2]{YZZ},  the pushforward by a correspondence defines a  map
\begin{equation}\Pic(M\times M)_R\to \Hom(J,J^\vee)_R.\label{PICJ}\end{equation}
 
  Define the 
\Neron-Tate height pairing $$\pair{\cdot,\cdot}_\NT :J_U(F^\sep)_\BQ \times J_U^\vee(F^\sep)_\BQ\to \BR$$
as in \cite[7.1]{YZZ}. 
By the projection formula, this pairing extends to the  \Neron-Tate height pairing $$\pair{\cdot,\cdot}_\NT:J  (F^\sep)_\BQ \times J^\vee (F^\sep)_\BQ\to \BR$$
which further induces
$$ \pair{\cdot,\cdot}_\NT:J  (F^\sep)_\BC \times J ^\vee(F^\sep)_\BC\to \BC.$$  
We  swap $J  (F^\sep)_\BQ $ and $ J^\vee (F^\sep)_\BQ $ in the definition of $\pair{\cdot,\cdot}_\NT$ when necessary.
Let $Z \in \Pic(M\times M)_\BQ $, and 
   $x,y\in J  (F^\sep)_\BQ$, then 
 $\pair{Z_* x,y}_\NT$ is well-defined.

  \subsubsection{Hodge classes}\label{Hodge classes}
  For our purpose, we only need to consider the case when $U$ is small enough so that there is   no ``elliptic points"   issue. 
  Let $\omega_{M_U/F}$ be the canonical bundle of $M_U$ over $F$. 
 If $U_\infty=\BB_\infty^\times$,  define the Hodge class of  $M_U$     to be
    \begin{equation*}L_U:=\omega_{M_U/F}(2 \cdot \cusp). \end{equation*}  
   In general, define  the Hodge class $L_U$ to be the  pullback  of  the Hodge class of  $M_{U\BB_\infty^\times}$.
\begin{lem} \label{nmb}For every $U'\subset U$, $L_{U'}=\pi_{U',U}^*L_U$.
\end{lem}
\begin{proof} If $D$ is a division algebra, the lemma follows from the fact that $\pi_{U',U}$ is \'{e}tale. (In fact, we always have $L_U=\omega_{M_U/F}$.)
If $D$ is the matrix algebra, the lemma follows from an explicit computation of the ramifications at cusps \cite[VII, Theorem 5.11]{Gek}. 
 \end{proof}
\begin{defn}Define $ \Vol (M_U):=\deg L_U$.  
\end{defn}

By Lemma \ref{nmb} and the projection formula, we have $\Vol(M_{U'})/\Vol(M_U) =\deg\pi_{U',U} $
 for   $U'\subset U$. Then a direct computation gives the following corollary.
\begin{cor}\label{the constant}  
  The number
$$\frac{\Vol (\tilde U/\Xi )\Vol(M_U)}{ |F^\times\bsl \BA_F^\times/  \Xi|} 
$$ is independent of $  U$.  (The notations are as in \ref{measures}.)
\end{cor}
Indeed, after normalization of measures,   this number is 4 (see Lemma \ref{the constant'}).

  For  $\alpha\in \pi_0(M_{U,F^\sep})$, let $M_{U,\alpha}$ be the corresponding geometrically  connected component and $L_{U,\alpha}:=L_U|_{M_{U,\alpha}}$. Define  normalizations
$\xi_{U,\alpha}:=\frac{1  }{\deg L_{U,\alpha}} L_{U,\alpha}$.  
For $  \alpha=(\alpha_U)\in \pi_0(M_{ F^\sep})$, the sequence     $(\xi_{U,\alpha_U})_U$, indexed by $U$, defines an element 
$$  \xi_\alpha\in\vpl \Cl(M_{U,\alpha_U})_\BQ. $$
For $x\in M(F^\sep)$  in the connected component of $\alpha$. Then $x-\xi_\alpha\in J(F^\sep)_\BQ  .$
This defines a map
\begin{equation}M(F^\sep)\incl J(F^\sep)_\BQ  .\label{xia}\end{equation} 
 

 %

   \subsection{CM points}\label{CM}
 
We   define CM points, 
and prove  the algebraicity of CM points.

 \subsubsection{Endomorphisms of $\cD$-elliptic sheaves}
Let $\BE$ be a $\cD$-elliptic sheaf.
 For $f\in \CO_F:=H^0(X-\{\infty\}, \CO_X) $, multiplication by $f$ on each $\CE_i$ gives an endomorphism of $\BE$.
  In particular, $\End(\BE)$ is a $\CO_F$-algebra. 
  Let   $\BC_\infty$ be  the completion of the algebraic closure of $F_\infty$.
 Suppose that  $\BE$ is defined over $\BC_\infty$ and    $\zero_\BE:\Spec \BC_\infty\to X$  of $\BE$ factors through the generic point of $X$.
 
  \begin{lem}There is an embedding $\End(\BE) \incl \cD\otimes{\CO_F}$, and $\End(\BE)\otimes_{\CO_F}F$  is a field extension  of $F$  which is  not split over $\infty$, and   of degree at most 2. 

     \end{lem} 
     \begin{proof}  By the analytic uniformization    \cite[2.13, 3.6]{Tae2}, there is a rank one $\cD\otimes{\CO_F}$-lattice $\Lambda$  in a $D$-representation on $\BC_\infty^2$  such that  $$\End(\BE)\cong \{ \lambda\in \BC_\infty:\lambda\Lambda\subset \Lambda \}\subset \End_{\cD\otimes{\CO_F}}(\Lambda)\cong\cD\otimes{\CO_F}.$$
Thus $\End(\BE)\otimes_{\CO_F}F$ is a commutative subalgebra of $D$.   Then the lemma follows.
         \end{proof}
 
  \subsubsection{CM $\cD$-elliptic sheave and CM points}
 Let $E/F$ be a quadratic field extension nonsplit over $\infty$, which is fixed from now in this and next section.
 \begin{defn}A $\cD$-elliptic sheaf $\BE$ has   CM by $E$ if $\End(\BE)\otimes_{\CO_F}F\cong E$.
  \end{defn} A point in $M(\BC_\infty)$ or $M_U(\BC_\infty)$   is called a CM point if it
 corresponds to a   $\cD$-elliptic sheaf  with CM by $E$.
 Let $CM$ (resp. $CM_U$) be the set  of all CM points in $M$ (resp. $M_U$).
 We will prove that all CM points are defined over $E^\ab$, the  maximal abelian extension of $E$. Then regard  $CM$ (resp. $CM_U$) as a subset  of $M(E^\ab)$ (resp. $M_U(E^\ab)$).
 
 Let $x \in   M(\BC_\infty)$, and let $\BE$ be the corresponding elliptic sheaf.
Associated to $x$ is an infinite  level structure $\kappa$    on $\BE$  and a  level structure $\kappa_\infty$ at $\infty$.
 The actions of endomorphisms of $\BE$ on $\kappa$ and  $\kappa_\infty$ defines a  group morphism \begin{equation}(\End(\BE)\otimes_{\CO_F}F)^\times \to  \BB^\times/U_\infty  \label{eq241}.\end{equation}  
The following lemmas   are easy to be verified. 
 \begin{lem}   
The    image  of $(\End(\BE)\otimes_{\CO_F}F)^\times$  under \eqref{eq241} fixes $x$.

 \end{lem}
   
  
 In particular, if $x\in CM$,  $x\in M(\BC_\infty)^{e(E^\times)}
   $ for some embedding $e:E \incl   \BB  $ of $F$-algebras.
 
   \begin{lem}  For an embedding $e:E   \incl \BB   $ of $F$-algebras, a point in $M(\BC_\infty)$ fixed by $e(E^\times)$ corresponds to  a CM $\cD$-elliptic sheaf $\BE$  with a infinite level structure over  all finite places and a level structure  at $\infty$ and such that image of   \eqref{eq241} is $e(E^\times)$.

        \end{lem}    
  
To sum up, we have a decomposition of the set CM of points\begin{equation}CM=\bigcup_{e:E \incl   \BB  } M(\BC_\infty)^{e(E^\times)}.\label{allcm}\end{equation}

 \subsubsection{CM points under the rigid analytic uniformization}\label{CMuni0}

 We describe the CM points in $M(\BC_\infty)$ under the rigid analytic unformization  at $\infty$ in \ref{Rigid analytic uniformization}.
We have two  isomorphisms $i_\infty:D_\infty \cong \RM_2(F_\infty)$ and $i_{\mathrm{f}}:D (\BA_{\mathrm{f}})\cong \BB_{\mathrm{f}}$  in the definition of unformization.
 Let $e_\infty:E\incl \BB_\infty $ and $d:E\incl D$ be   embeddings   of $F$-algebras.
Let $e$ be the product of $e_\infty$ and the composition $i_{\mathrm{f}}\circ d$.
  Let $z_0\in \Sigma_n(\BC_\infty)$ be a  fixed point  of $E^\times  $  via $$\left((i_{\infty}\circ d)^{-1}, e_\infty \right): E^\times \incl \GL_2(F_\infty)\times \BB_\infty^\times.$$
  By the Noether-Skolem theorem, there exists  $j
\in \BB^\times $ such that $jgj^{-1}=\bar g$ for every $g\in e(E)$ where $\bar g$ is the Galois  conjugate of $g$. 
Then the normalizer  $H$  of $e(E^\times)$ in $\BB^\times$  is isomorphic to $ \BA_E^\times  \cup \BA_E^\times  j$. 
It is not hard to prove that  
\begin{equation*}M (\BC_\infty)^{e(E^\times)}=\{[z_0h_\infty,h_{\mathrm{f}}]:h\in H\}. \end{equation*} 

For a general embedding $E\incl \BB$, by the Noether-Skolem theorem,  there exists $g\in \BB^\times$ such that the embedding is  $ g^{-1}e g$. 
Then \begin{equation*}M (\BC_\infty)^{e(E^\times)}=\{[z_0h_\infty g_\infty,h_{\mathrm{f}}g_{\mathrm{f}}]:h\in H\}.
\end{equation*}
In particular, from \eqref{allcm} we have \begin{equation*}CM=\{[z_0g_\infty,g_{\mathrm{f}}]:g\in \BB ^\times\}. \end{equation*}

 \subsubsection{Construction of CM $\cD$-elliptic sheaves}\label{CM theory}

 Let $\pi:X'\to X$ be a double cover which is a smooth  projective model of $E/F$. Let $\infty'$ be the unique preimage of $\infty$.
Let  $\CO_E =H^0(X'-\{\infty'\},\CO_{X'})$   the ring of integers of $E$  away from $\infty'$. 
 Let $S$ be an $\BF_q$-scheme.

 \begin{defn}
 An elliptic sheaf $\BL$, on $X'$    with respect to $\infty'$, of rank 1   over $S$   is the following data: a morphism $\zero_\BL:S\to X'$ and 
 a sequence  of commutative diagrams 
$$
    \xymatrix{
 {}^\tau \CL_{i-1} \ar[r]^{{}^\tau j_{i-1}}\ar[d]_{ t_{i-1}}  &{}^\tau \CL_{i} \ar[d]^{t_{i} } \\
  \CL_i \ar [r]_{j_i}   &    \CL_{i+1}   }
$$
indexed by $i\in\BZ$, 
 where each $\CL_i$ is a line bundle  on $X'\times S$ such that  $$\CL_{i+\deg(\infty')}=\CL_i(\infty'),$$
 and $j_i,t_i$  are injections compatible with $\cD$-actions and   satisfy  the following conditions:
\begin{itemize} 
\item[(1)] the composition $j_{i+\deg(\infty')-1}\circ\cdot\cdot\cdot\circ j_{i+1}\circ j_i$ is the canonical inclusion $\CL_i\incl\CL_i(\infty')$;
\item[(2)] let $\pr_S:X\times S\to S$ be  the projection, then $\pr_{S,*}(\coker j_i)$ is a locally free $\CO_S$-module of rank $1$;
\item[(3)] $\coker t_i$ is the direct image of a locally free $\CO_S$-module of rank 1 by the graph morphism $(\zero_\BL,\id): S\to X'\times S$.
 \end{itemize}

 
   \end{defn}
   \begin{lem}\label{cdl}There exists a maximal order $\cD$ which admits an embedding of $\CO_X$-algebras: \begin{equation}\cD\incl \pi_*( \End(\CO_X\oplus \CL)).\label{ebd}\end{equation} where $\CL$ is a line bundle on $X'$.
   \end{lem}
\begin{proof}Embed  $D$ in $ \RM_2(E)$ as the subalgebra of matrices  $ \begin{bmatrix}a&b\ep\\
 \bar b&\bar a\end{bmatrix} $ where $\ep\in F^\times$ and $a,b\in E$ (see \eqref{(5.1)}).  For $x\in |X|$, let 
 $\fp_{E_x}$ be the maximal ideal of $\CO_{E_x}$. For an integer $n$, we have an order     $\CO_n=\End_{\CO_{E_x}}(\CO_{E_x}\oplus \fp_{E_x}^{n_x})$  of  $ \RM_2(E_x)$. There exists an integer $n_x$ such that 
 $D_x\bigcap  \CO_{n_x}$  is a maximal order of $D_x$ and $n_x=0$ outside a finite set $I\subset |X|$.  
 Let $\cD$ correspond to   $\{D_x\bigcap  \CO_{n_x}\}_{x\in |X|}$ via Lemma \ref{order} and 
 $\CL=\otimes_{x\in I} \CO(-n_x x)$.
 \end{proof} 
 
 By   \ref {Equivalence}, we can let  $\cD$ be as in Lemma \ref{cdl}.      \begin{lem} \label{pil} 
  The $\cD$-action on $\pi_*(\CL_i \oplus (\CL_i\otimes \CL))$ by the embedding \eqref{ebd} makes
    $$ ( \pi\circ \zero_\BL, \pi_*(\CL_i \oplus (\CL_i\otimes \CL)),\pi_*(j_i \oplus (j_i\otimes \id_{\CL}
)),\pi_*(t_i \oplus( t_i\otimes \id_{\CL}
)): i\in \BZ)$$
  a $\cD$-elliptic sheaf on $X$ of rank 2    over $S$.    
  
 \end{lem}
  \begin{proof} 
 The condition at $\infty$ follows from the projection formula and that $\pi^*\CO_{X}(\infty)= \CO_{X'}(\infty')$.
 The remaining  verification is trivial.  \end{proof}

  Let  $M^1 $ be the  moduli space over $E$ of rank 1 elliptic sheaves  over $E$-schemes with  all level   structures (defined similarly as in \ref{Moduli spaces of}).
  From Lemma \ref{pil}, we  have an  $E$-scheme morphism $\Pi:M^1\to M_E,$
where $M_E$ is the base change to $M$ to $E$. We describe the map when $U_\infty=\BB_\infty^\times.$
Let $\hat\CL=\CL \otimes\hat\CO_E$.
    By  \cite[(18.7) Theorem]{Rei}, we can fix an isomorphism of right $\cD \otimes \hat \CO_F$-modules:
      \begin{equation}  \hat\CO_E \oplus \hat\CL\cong \cD \otimes \hat \CO_F\label{CMisom},\end{equation}
   where  $\cD \otimes \hat \CO_F$ acts on the left hand side by   \eqref{ebd}.   For  an $E$-scheme $S$ and  a point $(\BL,\kappa)$ in $M^1(S)$ forgetting the  level   structure at $\infty$,
  where $\kappa$ is an infinite level   structure over all finite places,
  let $\BE $ be the  $\cD$-elliptic sheaf on $X$ of rank 2    defined    in  Lemma \ref{pil}.
  Then tautologically    \begin{equation}\BE\otimes \hat \CO_F   \cong  \hat\CO_E \oplus \hat\CL  \label{CMisom0}\end{equation}
as right $\cD\boxtimes \CO_S$-modules, where the $\cD \otimes \hat \CO_F$-module structure on the left hand side is give by    \eqref{ebd} and $\kappa$.
   Thus 
\eqref{CMisom}  and \eqref {CMisom0} give an isomorphism $$ (\cD \otimes \hat\CO_F)\boxtimes\CO_S\cong \BE \otimes\hat\CO_F , $$  i.e. an infinite level structure on $\BE$.  This gives the  morphism  $\Pi:M^1\to M_E.$

 
There is an $\BA_E^\times/E^\times$-action on $M^1$ by acting on level structures. This is related to the $\BB^\times$-action on $M_E$ as follows.
Define an embedding $\BA_{E}\incl \BB $ as follows. 
 By \eqref{CMisom}, we have an isomorphism $$\End_{\cD \otimes \hat \CO_F}( \hat\CO_E \oplus \hat\CL)\cong \End_{\cD \otimes \hat \CO_F}( \cD \otimes \hat \CO_F)\cong \cD \otimes \hat \CO_F.$$
 The  diagonal left   action of 
  $\hat\CO_E$ on $\hat\CO_E \oplus \hat\CL$ gives an embedding
 $\hat\CO_E\incl \cD \otimes \hat \CO_F.$ Then we have an embedding
$\BA_{E,\mathrm{f}}\incl \BB_{\mathrm{f}} $.  The  construction  at  $\infty
$  is defined  in the same way. The embedding gives a  group morphism
\begin{equation}\BA _E^\times\to \BB^\times\to  \BB^\times/U_\infty\label{CMebd'}.\end{equation}  
Let $\BA _E^\times$  act on $ M$ via \eqref{CMebd'}.
By the description of the action of $\BB^\times$ on $M$  \cite[(8.7)]{LLS}  (see the proof of   Proposition \ref{BBaction}    when $U_\infty=\BB_\infty^\times$), 
 we have the follow lemma.  \begin{lem}\label{CMcpt} The morphism
     $ \Pi:  M^1\to  M_E$   is compatible with the actions of $\BA _E^\times$. 
     In particular, $\Pi(M^1)\subset (   M_E)^{E^\times}$.
       \end{lem}

 \subsubsection{Algebraicity of CM points}      Let $E^{\ab,\varpi_\infty}\subset E^\ab$ be the maximal subfield fixed by the image of $\varpi_
  \infty$ under the reciprocity map $\BA_E^\times/E^\times\to \Gal(E^\ab/E).$
 
  \begin{prop}[{\cite[Corollary of Proposition 2.2]{DriEll2}}]\label{CFT} 
  We have  $  M^1 \cong \Spec E^{\ab,\varpi_\infty}$. Moreover, the   $\BA_E^\times /E^\times$-action on $  M^1$ coincides with the  $ \Gal(E^{\ab }/E)$-action via the reciprocity map.

 \end{prop}

      Combining  \ref{CMuni0} with   Lemma \ref{CMcpt} and  Proposition \ref{CFT},   we have the following corollary.
 \begin{cor} \label{TSCM}

 (1)  The scheme $(   M_E)^{E^\times}$  
   consists of  two $\BA_E^\times /E^\times$ orbits and the image of $\Pi $  is one of them.

(2)  All CM points, in particular all points in  $(   M_E)^{E^\times}(\BC_\infty)$,   are defined over $E^\ab$.  Moreover, the  $\BA_E^\times /E^\times$-action on $( M_E)^{E^\times}$  coincides with the  $ \Gal(E^{\ab }/E)$-action via the reciprocity map.

  \end{cor}

   \section{Height distribution}
 \label{Height distributions, Rational Representations, and Abelian varieties}

In this section,   we  first  introduce  the height distribution on $\BB^\times$ and  study its spectral decomposition.
Then we reduce Theorem \ref{GZ} to   its distribution version, namely Theorem \ref{GZdis'}.
 \subsection{Height distribution}   \subsubsection{Hecke correspondences}\label{Hecke correspondences}
 Continue to use the notations in the last section.
For $g\in \BB ^\times,$  let    $T_g:M\to M$ be the  right   action by $g $, $\pi_U:M\to M_U$ be the natural projection.
 Let $Z(g)_U'$ be  the image of  the graph of $T_g$  in $M_U\times M_U$ under $\pi_U\times \pi_U$, as a reduced closed subscheme of $M_U\times M_U$. Define a $\BQ$-coefficient divisor of $M_U\times M_U$:
\begin{equation*}Z(g)_U :=\frac{|\tilde U g \tilde U/ \tilde U|} {|F^\times\bsl F^\times  \tilde U g \tilde U/ \tilde U|}Z(g)_U'. \end{equation*}
Then the pushforward by $Z(g)_U$ on divisors on $M_{\BC_\infty}$ is the usual Hecke correspondence: for example,
under the rigid analytic uniformization at $\infty$,   the pushforward by $Z(g)_U$  is  \begin{equation*} [z,x]_U\mapsto \sum_{y\in  \tilde U g \tilde U/ \tilde U}[zy_\infty,xy_{\mathrm{f}}]_U  .\end{equation*}

  Let $K/\BQ$ be a field extension and $\CH_{U,K}$ the Hecke algebra of $\BB^\times$ (see \ref {measures}). For $f\in \CH_{U,K}$, define $$Z(f)_U:=\sum_{g\in \tilde U\bsl \BB^\times/\tilde U}f(g)Z(g)_U,$$
  and  define a normalization  $$\tilde Z(f)_U:=\Vol(\Xi_U)|F^\times\bsl \BA_F^\times/\Xi|Z(f)_U.$$
  The effect of the normalization factor $\Vol(\Xi_U)|F^\times\bsl \BA_F^\times/\Xi|$ is given by  the following lemma.
  \begin{lem} [{\cite[ Lemma 3.18]{YZZ}}]  \label{normhecke}
The line bundles defined by $ \tilde Z(f)_U   $'s are compatible under pull back. In  particular, the sequence
 $(\tilde Z(f)_U )_U $ defines an element in $  \Pic(M\times M)_K.
 $
 \end{lem}  
    
 Denote this element by $\tilde Z(f)\in  \Pic(M\times M)_K.
 $
 Define   $\tilde Z(f)_*\in \Hom(J ,J^\vee )_{K }$  by \eqref{PICJ}.

 \subsubsection{Cohomological projectors}
 We follow \cite[3.3.1]{YZZ}.  

   Identifying $H^1(M_{U,F^\sep},\bar\BQ_l)$ with $H^1(J_{U,F^\sep},\bar\BQ_l)$, then we have  an injection 
\begin{equation}\Hom(J_U ,J_U )_\BQ\incl \Hom(H^1(M_{U,F^\sep},\bar\BQ_l),H^1(M_{U,F^\sep},\bar\BQ_l)) \label{3121}
\end{equation} which maps a morphism  $\phi\in \Hom(J_U,J_U)$ to its action $\phi^*$ on  $H^1$ via pullback.
Let  $(\phi_U)_U$  be a sequence  of elements in   $\Hom(J_U ,J_U )_\BQ$  indexed by small enough  $U$ such that 
 \begin{equation}  \phi_U(\pi_{U',U,*}(x_{U'}))=\pi_{U',U,*}( \phi_{U'}(x_{U'})).\label{31215}
\end{equation}
 Then the image of $(\phi_U)_U$ via \eqref{3121} gives an element
in $\Hom(H^1(M_{ F^\sep},\bar\BQ_l),H^1(M_{ F^\sep},\bar\BQ_l) )$. 

 \begin{lem}\label{cu}
 Let $ (c_U )_U$ be a sequence of rational numbers such that $c_{U'}/c_U =\deg\pi_{U',U} .$
Let  $\psi=(\psi_U)_U\in\Hom(J,J^\vee)_\BQ$. 
Then  the sequence 
$ (c_{U}^{-1}\psi_U)_U $
satisfies \eqref{31215}.

 \end{lem}
 \begin{proof}    By the definition of $\Hom(J,J^\vee)_\BQ$ (see \eqref{HJJ}), we have $$\pi_{U',U}^*\psi_U(\pi_{U',U,*}(x_{U'}))=\psi_{U'}(x_{U'})$$ for $x_{U'}\in J_{U'}$. 
Apply $\pi_{U',U,*}$ to both sides, we have $$
\deg \pi_{U',U}\cdot \psi_U(\pi_{U',U,*}(x_U')))=\pi_{U',U,*}\psi_{U'}(x_{U'}).$$    This  gives \eqref{31215}. 
 \end{proof}

   The sequence $C:=(c_U)_U$ 
  where  \begin{equation}\label{cu1}c_U:= |F^\times\bsl \BA_F^\times/\Xi|/\Vol(\tilde U/\Xi)\end{equation}
  satisfies the property in Lemma \ref{cu}.
So we have a morphism 
  \begin{equation}
\Hom(J ,J^\vee )_{\bar \BQ_l}\incl \Hom(H^1(M_{ F^\sep},\bar\BQ_l),H^1(M_{ F^\sep},\bar\BQ_l) )\label{inclJ}
\end{equation}
by 
$$\psi=(\psi_U)_U\mapsto (C^{-1}\psi)^*:=(c_U^{-1}\psi_U^*)_U.
$$
 
  \begin{lem}\label{compareHecke}
 Let 
 $ T(f) \in  \Hom(H^1(M_{ F^\sep},\bar\BQ_l),H^1(M_{ F^\sep},\bar\BQ_l) )$ be  defined  as in \ref{LLSmain}, then
$$ (C^{-1} \tilde Z (f)_* ) ^*= T(f).
 $$   
 \end{lem}

 \begin{proof}  
 For an open compact subgroup $U$ of $ (\cD\otimes \hat\CO_F)^\times $ such that $f$ is bi-$U$-invariant,  we have the  restriction of $T(f)$ to ${H^1(M_{U,F^\sep},\bar\BQ_l)}$: $$T(f)|_{H^1(M_{U,F^\sep},\bar\BQ_l)} : H^1(M_{U,F^\sep},\bar\BQ_l)\to H^1(M_{U,F^\sep},\bar\BQ_l) .$$
We show that the pullback action of $\Vol (\tilde U/\Xi_\infty)Z(f)_{U,*}\in \Hom (J_U,J_U)$ on $H^1(J_{U,F^\sep},\bar\BQ_l) $, which is identified with  $H^1(M_{U,F^\sep},\bar\BQ_l)$, equals  $T(f)|_{H^1(M_{U,F^\sep},\bar\BQ_l)}$. 
We may assume $f=1_{\tilde Ug\tilde U}$ for some $g\in \BB^\times$.
 Consider the following diagram  $$
    \xymatrix{
H^1(M_{U,F^\sep},\bar\BQ_l) \ar@{.>}[r] \ar[d]_{\Vol(\tilde U/\Xi_\infty) \bigoplus_{h\in \tilde Ug\tilde  U/\tilde U} T_h^*}  &  \ar[d]  H^1(M_{U,F^\sep},\bar\BQ_l)  \\
 \bigoplus_{h\in \tilde Ug\tilde U/\tilde U}  H^1(M_{hU h^{-1}, F^\sep},\bar\BQ_l)  \ar [r]_{  }    &   H^1(M_{ F^\sep},\bar\BQ_l) },
$$
then $T(f)|_{H^1(M_{U,F^\sep},\bar\BQ_l)}$ is the dotted arrow making the diagram above commute.
So we need to check that $\Vol(\tilde U/\Xi_\infty)Z(f)_{U,*}$ is the dotted arrow making the diagram below commute.
$$
    \xymatrix{
J_U &  \ar@{.>}[l] J_U  \\
 \ar[u]_{  \bigoplus_{h\in \tilde Ug\tilde U/\tilde U} T_{h,*}} \bigoplus_{h\in \tilde Ug\tilde U/\tilde U}  J_{hU h^{-1}}      &   \ar [l]_{  } \ar[u] J }.
$$
 This   follows from the definition of $Z(g)_{U,*}$.    \end{proof}

 Let $\pi\in \CA_{U_\infty} (\BB ^\times,\bar\BQ_l)$.  The Hecke action  of $f\in  \CH _{\bar\BQ_l}$ on $\phi\in\pi$  
is defined by \begin{equation} \pi(f)\phi(x):=\int_{  \BB ^\times/\Xi_\infty} f(g)\phi(xg)dg\label{pixi}.\end{equation}
Let  $\tilde \pi$ be the contragradient representation of $\pi$.  Let $  \rho_{\pi }:\CH _{\bar\BQ_l}\to \pi\otimes \tilde\pi$ be given by   $$\rho_\pi(f)=\sum_{\phi }\pi(f)\phi \otimes \tilde \phi$$
where the sum is over an orthonormal  basis $\{\phi\}$ of $\pi$, and $\{\tilde \phi \}$ is the dual basis of $\tilde \pi$. 
Note that the sum has only finitely many nonzero terms.
It is well known that  
 $$\rho = 
\bigoplus_{\pi\in \CA_{U_\infty}(\BB^\times,\bar \BQ_l)} \rho_{\pi }: \CH _{\bar\BQ_l}\to \bigoplus_{\pi\in \CA_{U_\infty}(\BB^\times,\bar \BQ_l)} \pi\otimes \tilde\pi $$ 
is surjective.    
 \begin{cor}\label{ebdj}  The natural embedding
 $ \pi \otimes\tilde \pi \incl \Hom(H^1(M_{ F^\sep},\bar\BQ_l),H^1(M_{ F^\sep},\bar\BQ_l))$ given by Theorem \ref{semisimple}  factors through $\Hom(J,J^\vee)_{\bar\BQ_l}$ via \eqref{inclJ}.
 

 \end{cor} 
 
 \begin{proof}By the discussion above, the image of any $$x\in \pi \otimes\tilde \pi \incl 
 \bigoplus_{\pi'\in \CA_{U_\infty}(\BB^\times,\bar \BQ_l)}  \pi' \otimes\tilde \pi' ,$$ equals $\rho  (f)$ for a certain $f\in \CH _{\bar\BQ_l}$. Since Theorem \ref{semisimple} identifies the usual Hecke action $\rho  (f)$  and the Hecke  action  $T  (f) $ on   $H^1(M_{F^\sep},\bar \BQ_l)$,  the image of $x$  in $\Hom(H^1(M_{ F^\sep},\bar\BQ_l),H^1(M_{ F^\sep},\bar\BQ_l))$ equals $T (f)$. Then the Corollary follows from Lemma \ref{compareHecke} by letting $\tilde Z(f)_*$ be the image of $x$ in $\Hom(J,J^\vee)_{\bar\BQ_l}$.  \end{proof}
 
 Fix an isomorphism $c:\BC\cong \bar\BQ_l$. For $\pi\in \CA_{U_\infty}(\BB ^\times,\BC)$, let  $\pi^c:=\pi\otimes_{\BC,c}\bar\BQ_l \in \CA_{U_\infty}(\BB ^\times,\bar\BQ_l)$. 
Corollary  \ref{ebdj}, applied to $\pi^c$, 
 defines an injective morphism 
 $$ T_{\pi}':\pi\otimes \tilde\pi \to \Hom(J ,J^\vee )_\BC.$$  
 Then for   $f\in \CH_{\BC}$, we have an equation  in $\Hom(J ,J^\vee )_\BC$:
\begin{equation}\label{specdecomht} \tilde Z(f)_{ *}= 
 \bigoplus_{\pi\in \CA_{U_\infty}(\BB^\times,\BC)}  T'_{\pi  } \circ \rho_{\pi }(f). 
  \end{equation} 
   Let $c_U$ be as in \eqref{cu1}, then   $ \Vol(M_U)/c_U$ is the number in
  Corollary \ref{the constant},   and  independent of $U$.
 \begin{defn}\label{cohproj}Define the cohomological projector associated to $\pi$ to be
\begin{equation*}T_\pi:=\frac{ \Vol(M_U)}{c_U}T_{\pi}':\pi\otimes \tilde\pi \to \Hom(J ,J^\vee )_\BC.\end{equation*} \end{defn}

      \subsubsection{The Gross-Zagier-Zhang formula, projector version}   \label{Spectral decomposition of  cohomology'}
   \begin{defn}[{\cite[1.6.7]{YZZ} }]
\label{regint} Let $f$ be a locally constant function on $E^\times\bsl \BA_E^\times $ invariant by  $\Xi_\infty$. Define $$ \fint _{\BA_F^\times} f(z)dz= \frac{1}{ |F^\times\bsl \BA_F^\times/ \Xi_\infty  V|}\sum_{z\in F^\times\bsl \BA_F^\times/ \Xi_\infty  V}f(z)  $$
where $V$ is  any open  compact subgroup of  $\BA_{F,\mathrm{f}}^\times$ such that     $f$  is invariant by $V$.
Define the regularized integral
  $$\int_{E^\times\bsl \BA_E^\times  }^* f(t) dt=\int_{E^\times\bsl \BA_E^\times / \BA_F^\times} \fint _{\BA_F^\times} f(z)dz .$$ 
 \end{defn}  
  
  Fix a  CM point $P_0\in M^{e_0(E^\times)}(F^\sep)$. 
  For $h  \in \BB^\times$, let  $h^\circ$ be the image of  $T_{h}P_0$ in $J(F^\sep)_\BQ$ via \eqref{xia}.           
Let $\Omega$ be a Hecke character of $E^\times$ valued  in $\BC$. Define
\begin{equation} H_\pi^{\sharp,\proj}(\phi\otimes\varphi):=\int_{E^\times\bsl \BA_E^\times / \BA_F^\times}\int^*_{E^\times\bsl \BA_E^\times   }   \pair{ T_\pi( \phi\otimes \varphi)(t_1^\circ),  t_2^\circ}_{\NT} \Omega^{-1}(t_2)\Omega(t_1) dt_1dt_2.\label{Hproj}\end{equation}  

Let $\omega$ be the restriction of $\Omega$ to  $\BA_F^\times/F^\times$. 
Assume that  the central character of $\pi$ is $\omega^{-1}$. Otherwise both sides of \eqref{Waldeq} are 0.

 Regard $E^\times/F^\times$ as an algebraic group over $F$.
 Fix   the Tamagawa measure on $(E^\times/F^\times) (\BA_F)$.
  Then   $\Vol(E^\times\bsl \BA_E^\times / \BA_F^\times)=2L(1,\eta)$. 
 \begin{thm}[The Gross-Zagier-Zhang formula, projector version] 
 \label{projversion} 
 For $\phi \in \pi ,\varphi\in\tilde \pi $, 
\begin{equation} H_\pi^{\sharp,\proj}(\phi\otimes\varphi) =\frac{L(2,1_F)L'(1/2,\pi,\Omega)}{ L(1,\pi,\ad)}  \alpha_{\pi}(\phi ,\varphi ) . \label{eqprojversion}
\end{equation}
 
\end{thm} 
  
Now we reduce  Theorem \ref{projversion}  to its distribution version (see  Theorem \ref{GZdis'} below).       \subsubsection{Height distribution }
\label{The CM height distribution}

 \begin{defn}  \label{CM height}  
 Define   the    height  distribution on $\BB^\times$ by assigning to $f\in    \CH_{\BC}$ the complex number  \begin{align}  H(f) :&=  \int_{E^\times\bsl \BA_E^\times / \BA_F^\times}\int_{E^\times\bsl \BA_E^\times  }^*  \pair{\tilde Z(f) _*t_1^\circ,t_2^\circ}_\NT\Omega^{-1}(t_2)\Omega(t_1) dt_2dt_1 .\label{3.8}  \end{align} 
  Here 
  the   \Neron-Tate height pairing is  the one on   $J^\vee\times J $.
     \end{defn}

  Let $\pi\in \CA_{U_\infty}(\BB ^\times,\BC)$.
Define a linear functional  $H_\pi ^\sharp$ on $\CH_\BC $ by \begin{equation}H_\pi^\sharp (f) =    \int_{E^\times\bsl \BA_E^\times / \BA_F^\times}\int_{E^\times\bsl \BA_E^\times  }^*  \pair{T_\pi\circ \rho_\pi(f)t_1^\circ,t_2^\circ}_\NT \Omega^{-1}(t_2)\Omega(t_1) dt_2dt_1 . \label{HHpi0}
  \end{equation}  Define 
\begin{equation}H_\pi (f) =  \frac{  [F^\times\bsl \BA_F^\times/\Xi]}{\Vol(\tilde U/\Xi )\Vol(M_U)}   H_\pi^\sharp (f). \label {specdecomht0} \end{equation}
Here the coefficient  is independent of $U$ (see Corollary \ref{the constant}).
By    \eqref {specdecomht},   we have 
\begin{equation}H(f )=
\sum_{\pi\in \CA_{U_\infty}(\BB^\times,\bar \BQ_l)}  H_\pi(f ).\label {specdecomht'} \end{equation}

  Immediately  from the definition of $H_\pi^{\sharp,\proj}$ in \eqref{Hproj}, we have   \begin{equation*}H_\pi^\sharp(f)=H_\pi^{\sharp,\proj}(\rho_\pi(f)).  \end{equation*} 
According to the global Hecke action \eqref{pixi},  define the  action of $f_\infty$ on $\phi_\infty\in \pi_\infty$ to be  
  \begin{equation}\pi_\infty(f_\infty)\phi_\infty(x):=\int_{  \BB_\infty ^\times/\Xi_\infty} f_\infty (g)\phi_\infty(xg)dg.\label{HHprojpi}\end{equation} 
Let $\alpha^\sharp_{\pi_v}( f_v)$  be defined as in \eqref{alpha} (for $v=\infty$, with the action of $f_\infty$ as in \eqref{HHprojpi}).

Let $\omega$ be the restriction of $\Omega$ to  $\BA_F^\times/F^\times$.    Let  
$\Sigma(\pi,\Omega)$ be defined by the  formula   in \eqref{sigset}.
By the theorem of  Tunnell-Satio (see Theorem \ref{TSlocal}) and the same reasoning as in \ref{The global relative trace formulaintro},
 Theorem \ref{projversion} is implied by  the following theorem. 
 
\begin{thm}[The Gross-Zagier-Zhang formula, distribution version]\label{GZdis'}
 Assume that the  central character of $\pi$ is $\omega^{-1}$ and  $\Ram=\Sigma(\pi,\Omega).$
 There exists $f=\bigotimes_{v\in |X| }f_v\in \CH_\BC$,  such that \begin{equation}  H_\pi^\sharp(f )  =\frac{L(2,1_F)L'(1/2,\pi_{E_v})}{ L(1,\pi,\ad)} \prod_{v\in |X| } \alpha_{\pi_v}^\sharp (f_v),\label{GZdiseq}\end{equation}
and $ \alpha_{\pi_v}^\sharp (f_v)\neq 0$
for every $v\in |X|$.
 \end{thm}

We will prove  Theorem  \ref{GZdis'} in Section \ref{Proof of Theorem}.

 \subsection{Reduction of  Theorem \ref{GZ}}
 \label{Abelian  varieties parametrized by  modular curves}
 We reduce Theorem \ref{GZ} to   Theorem \ref{GZdis'}.

    \subsubsection{Modular abelian varieties}\label{3.3.1Abelian  varieties parametrized by  modular curves}

\begin{defn}An abelian variety $A$ over $F$ is modular if $\Hom(J,A )_\BQ$ is nontrivial.
\end{defn} 
For  $
 \pi\in  \CA_{U_\infty}(\BB^\times,\bar \BQ_l)$, we   construct a simple modular abelian  variety parametrized by $M$. For simplicity, assume that   $\dim \pi_\infty =1$. (Otherwise, replace $U_\infty$ by  an open, not necessary normal, subgroup of $\BB^\times/U_\infty$  such that $\dim \pi_\infty^{U_\infty'} =1$.) 
Choose  $U$   such that $\dim \pi^{U }=1$. (For the existence of such $U$ and $U_\infty'$, see \cite{Cas}.)
    Let $S $ the   subset of $v\in |X|  $ such that $U_v$ is not maximal   or $\BB_v$ is  a division algebra.
  Let $\BS^S$ be the subalgebra of $\End(J_{U}) $ generated by the image of $Z(x)_{U,*}$ such that  $x_v=1$ at all places $v\in S$.  
Let   $K_\pi\subset \bar \BQ_l$  
 be the  image of the  spherical Hecke character  associated to $\pi$ on the 
 $\BQ$-valued spherical Hecke algebra of $(\BB^S)^\times$.
Below, we understand $K_\pi$ as an abstract field and ignore the embedding into $\bar\BQ_l$.  
    Lemma \ref{compareHecke} 
   gives a morphism $\BS^S  \to K_\pi  $.   Let $\ker_\pi $ be the kernel. Let $A_\pi=J_{U }/\ker_\pi J_{U }$.
 Define $\pi^\BQ:=\Hom(J,A_\pi)_\BQ,$
 which is a $\BQ$-coefficient representation of $\BB^\times.$
Then  we have    inclusions of $\BQ$-algebras:  \begin{equation}K_\pi\incl \End(A_\pi)_\BQ\incl \End_{\BB^\times}(\pi^\BQ)\label{lpiact}.\end{equation}
   In particular, $K_\pi
$ is a finite field extension of $\BQ$.  
  (The finiteness can also be obtained from \cite[Proposition 10.5]{JL}, see \cite[p 73]{YZZ}.)
   By the strong multiplicity one theorem, the $\Aut(\bar\BQ_l)$-orbit $O_\pi$ of $\pi$ has $[K_\pi:\BQ]$ elements, indexed by 
   embeddings of $K_\pi$ into $\bar\BQ_l$ which are given by their spherical Hecke characters.  
   
 
  \begin{lem}\label{Jdecom} 
 (1) The $\Gal(F^\sep/F)$-representation associated to $H^1(A_{\pi,F^\sep},\bar\BQ_l) $ is  the direct sum of  $\LC(  \pi')$ for $\pi'\in O_\pi$. Here $\LC$ is as in \ref{LLSmain}.
 
 (2) The   inclusions \eqref{lpiact} are isomorphisms. In particular, the abelian variety $A_\pi$ is simple of dimension $[K_\pi:\BQ]$, and the $\BB^\times$-representation $\pi^\BQ $ is irreducible.

 \end{lem}
\begin{proof}  
By Theorem \ref{semisimple} and that $\dim \pi^U=1$, the subspace of $H^1(J_{U ,F^\sep},\bar\BQ_l) $
 where $\BS^S $ acts via $K_\pi$ is a  direct sum of    $\LC(  \pi')$ for $\pi'\in O_\pi$.
By the definition of $A_\pi$, this subspace is also $H^1(A_{\pi,F^\sep},\bar\BQ_l) $.
 Then  (1) follows.  (1) implies that the $\BB^\times$-representation 
$$\Hom_{\Gal(F^\sep/F)}(H^1(A_{F^\sep} ,\bar\BQ_l),H^1(J_{\pi,\sep},\bar\BQ_l))$$ is the direct sum of representations in $  O_\pi$. 
   By the natural inclusion $$  \Hom(J,A_\pi)_{\bar\BQ_l}\incl \Hom_{\Gal(F^\sep/F)}(H^1(A_{F^\sep} ,\bar\BQ_l),H^1(J_{\pi,\sep},\bar\BQ_l)),$$
we have $\dim_{\bar \BQ_l}   \End_{\BB^\times}(\pi^\BQ)\otimes_{\BQ}{\bar \BQ_l}\leq [K_\pi:\BQ].$
Then (2) follows.  \end{proof}

Conversely, let  $A$ be a simple modular abelian variety over $F$, and  $\pi_A:=\Hom(J,A) _ \BQ$.
By Zarhin's theorem on the Tate conjecture and  Theorem \ref{semisimple}, $\pi_A\otimes_\BQ\bar\BQ_l$ 
is the direct sum of    $  \pi\in  \CA_{U_\infty}(\BB^\times,\bar \BQ_l)$ such that $\LC(\pi)$ is a  direct summand of the 
$\Gal(F^\sep/F)$-representation $H^1(A_{\pi,\sep},\bar\BQ_l)$.  
Let $\pi$ be an direct summand of $\pi_A\otimes_\BQ\bar\BQ_l$.
By Lemma \ref{Jdecom} (2) and the simplicity of $A$, $A$ is   isogeny to $A_\pi$ and $\pi_A\otimes_\BQ\bar\BQ_l$ is 
 the direct sum of  $\pi'\in O_\pi$.

   \subsubsection{$L$-functions}\label{Lf}
  Denote $\pi_A$ by $\pi$ in this paragraph. Let   $K=  \End_{\BB^\times}(\pi ).$ 
Use the duality pairing \eqref{ duality pairing} to identify $\pi_{A^\vee}$ with the contragradient  $\tilde\pi$ of $\pi_{A }$.
  Then for every $\iota: K\incl \BC$, $\tilde\pi \otimes _{K,\iota }\BC$ is irreducible, and 
$\tilde\pi \otimes _{K,\iota }\BC$ is the contragradient of $\pi\otimes _{K,\iota }\BC$.
 Then for every $\iota: K\incl \BC$, $\tilde\pi \otimes _{K,\iota }\BC$ is irreducible.
Let $\Omega$ be a Hecke character of $E^\times$ with coefficients in a finite field extension $K'$ of $K$. 
For $v\in |X|$, define the local $L$-factor
   $L(s,\pi_v,\Omega_v)$ to be the unique rational function of $q_v^{-s}$ with coefficients in $K'$  such that   for every embedding $\iota:K'\incl \BC$, we have 
 $$ \iota(L(s,\pi_v,\Omega_v))=L(s ,(\pi^\iota)_{E_v}\otimes \Omega_v^\iota),$$  where  $ \iota$ acts on  the coefficients of $L(s,\pi_v ,\Omega_v)$, see \cite[3.2.2]{YZZ}. 
     Define 
  $$L(1/2,\pi,\Omega):=\prod_{v\in |X|} L(s,\pi_v,\Omega_v).$$
which  is a    polynomial on $q^{-s}$ with coefficients in $K'$. 
It
 satisfies a functional equation $$L(s,\pi,\Omega)=\vep(s,\pi,\Omega)L(1-s,\tilde \pi,\Omega^{-1})$$ where 
 $\vep(s,\pi,\Omega)$ is an exponential  $\pm c^{s-1/2}$ where $c$ is a positive number.

Assume that   $\Omega$ comes from a continuous  character $\Omega$ of $ \Gal(E^\ab/E)$
 via the reciprocity map
 $\BA_E^\times/E^\times\to \Gal(E^\ab/E).$
  For a place 
$v\in |X|$ nonsplit in $E$, regard $v$ as a place of $E$. Let
  $I_v$ be the inertia group of $E$ at $v$, $\Fr_v $  the geometric Frobenius at $v$.
Let 
 \begin{equation}\label{ldef0}P_v(T):=\det_{K'\otimes_\BQ  \BQ_l}(1-\Fr_v\cdot  T|(H^1(A_{F^\sep}, \BQ_l) \otimes _K\Omega)^{I_v}) \in K'\otimes_\BQ  \BQ_l[T].\end{equation}
 The definition for $v $ split in $E$ is similar.  
 By Deligne's work on the Weil conjecture \cite{Del} (or by Lemma \ref{Jdecom} (1)),  
 \begin{equation}\label{ldef}L(s,A_E,\Omega):=\prod_{v\in |X|} P_v(q_v^{-s})^{-1}\end{equation}
  is well-defined and valued  in $K'\otimes_\BQ\BC$. Moreover, Lemma \ref{Jdecom} (1)  implies that   $$L(s,A_E,\Omega)=L(s-1/2,\pi , \Omega ).$$

   \subsubsection{Height pairing} \label{Height pairing}
We follow \cite[1.2.4]{YZZ}. By  \cite[Proposition 7.3]{YZZ},  the  usual \Neron-Tate height pairing
 on  $A(F^\sep)_\BQ  \otimes  A^\vee(F^\sep)_\BQ$ descends to a pairing
   $$\pair{\cdot,\cdot}_\NT :A(F^\sep)_\BQ \otimes_K  A^\vee(F^\sep)_\BQ \to \BC.$$
   For $x\in A(F^\sep)_\BQ, y\in  A^\vee(F^\sep)_\BQ$,  
   $$(a\mapsto\pair{ax,y}_\NT)\in \Hom_\BQ(K,\BC)\cong K\otimes_\BQ \BC, $$ where the last isomorphism is from the trace map $K\otimes K\to \BQ$.
    Thus we have a pairing   $$\pair{\cdot,\cdot}_\NT^K :A(F^\sep)_\BQ \otimes_K  A^\vee(F^\sep)_\BQ \to  K\otimes_\BQ \BC.$$
    Let $K'/K$ be a finite  field extension. We have the extension of $\pair{\cdot,\cdot}_\NT^K$ by scalar
     $$\pair{\cdot,\cdot}_\NT^{K'} :(A(F^\sep)_\BQ\otimes_K K')\otimes_{K'} (A^\vee(F^\sep)_\BQ\otimes_K K') \to K' \otimes_\BQ \BC.$$

     For   an embedding $\iota:K\incl \BC$,  there is a canonical isomorphim 
   $$(A(F^\sep)_\BQ \otimes_K  A^\vee(F^\sep)_\BQ )\otimes_{K,\iota} \BC\cong  (A(F^\sep)_\BQ\otimes_{K,\iota}\BC) \otimes_\BC (A^\vee(F^\sep)_\BQ\otimes_{K,\iota}\BC).$$
 Let $\BC \otimes_\BQ \BC\to \BC$ be the  multiplication map. Then we define     the $\iota$-component of $ \pair{\cdot,\cdot}_\NT^K$ by   \begin{equation}\pair{\cdot,\cdot}_\NT^\iota:(A(F^\sep)_\BQ\otimes_{K,\iota}\BC) \otimes_\BC (A^\vee(F^\sep)_\BQ\otimes_{K,\iota}\BC)\to \BC \otimes_\BQ \BC\to \BC.\label{pairiota}\end{equation}

   \subsubsection{Height identity} \label{Duality pairing, algebraic projectors and the height identity}
   Now we  deduce Theorem \ref{GZ} from Theorem \ref{projversion}, following \cite[3.3.3, 3.3.4]{YZZ}.    
      Let $\iota: K'\incl \BC$ be an embedding and we also use $\iota$ to deduce its restriction to $K$.
 Let $\phi=\sum_{i}\phi_i\otimes_{K,\iota } c_i\in \pi_A^\iota.$
 Each $\phi_i$ gives a morphism $\phi_i:M(F^\sep)\incl J(F^\sep)\to  A(F^\sep)$.
 Then we have an induced map
$$\phi:M(F^\sep)\to  A(F^\sep)_\BQ\otimes_{K,\iota}\BC.$$
Similarly for $\varphi\in \pi_{A^\vee}^\iota.$ 
  On the targets of $\phi$ and $\varphi$ we have the $\iota$-component of the \Neron-Tate height pairing   $\pair{\cdot,\cdot}_\NT^\iota $  (see \eqref{pairiota}). 
Define \begin{equation}
\pair{P_\Omega(\phi),P_{\Omega^{-1}} (\varphi)}_{\NT} ^\iota:=\int_{E^\times\bsl \BA_E^\times / \BA_F^\times}\int^*_{E^\times\bsl \BA_E^\times   }
    \pair{ \phi(t_1),\varphi(t_2)}_{\NT} ^\iota\Omega^{\iota,-1}(t_2)\Omega^\iota(t_1) dt_1dt_2.\label{GZcomplex}\end{equation} 
   The complex version of  Theorem \ref{GZ} is as follows (compare with \cite[Theorem 3.13]{YZZ}).  
   \begin{thm}\label{GZiota} 
 For $\phi \in \pi _A^\iota$ and $\varphi\in\tilde \pi _{A^\vee}^\iota$,  we have 
\begin{equation} \pair{P_\Omega(\phi),P_{\Omega^{-1}} (\varphi)}_{\NT} ^\iota=\frac{L(2,1_F)L'(1/2,\pi_A^\iota,\Omega^\iota)}{ L(1,\pi_A^\iota,\ad)}   \alpha_{\pi_A^\iota } (\phi ,\varphi ) . \label{eqcomversion}
\end{equation}

  \end{thm}
  Applying Theorem  \ref{GZiota} to all embeddings $\iota: K'\incl \BC$, Theorem \ref{GZ} follows from Theorem  \ref{GZiota} and Corollary \ref{TSCM}.
     Now we deduce Theorem  \ref{GZiota} from Theorem \ref{projversion}.

Abusing notations,  we use $\iota$ to denote the natural inclusions $$\iota:\pi_A \incl \pi_A\otimes _{K,\iota}\BC,\ \iota: \pi_{A^\vee}\incl \pi_{A^\vee}\otimes _{K,\iota}\BC.$$
The embedding $\iota:K\to \BC$ also induces natural inclusions
  $$\iota:A(F^\sep)_\BQ \incl A(F^\sep)_\BQ\otimes_{K,\iota}\BC,\ \iota:A^\vee(F^\sep)_\BQ \incl A^\vee(F^\sep)_\BQ\otimes_{K,\iota}\BC.$$
  By     the definition of  the  cohomological projector    $T_{\pi_A^\iota}$ (see   Definition \ref{cohproj}),
  Lemma \ref{compareHecke} and the same discussion as in \cite[3.3.3, 3.3.4]{YZZ}, we have the following lemma.
      \begin{lem} \label{lheightid}
     Let $\phi\in\pi_A,\ \varphi\in \pi_{A^\vee}$, $x,y\in J (F^\sep)  $, 
     then $$\pair{T_{\pi_A^\iota}(\iota(\phi)\otimes\iota(\varphi))x,y}_\NT=\pair{ \iota(\phi(x)),\iota(\varphi(y))}_\NT^\iota$$
     where the  height pairing on the left hand side is the   one   on $J  (F^\sep)_\BC \times J ^\vee(F^\sep)_\BC$  (see \ref{Jacobians and Height pairings}).
      \end{lem}


       Then  Theorem \ref{GZiota} follows from Theorem \ref{projversion}, and we have the following Corollary.

      \begin{cor}   \label{GZfromproj} Theorem \ref{GZ} follows from Theorem \ref{projversion}.  
      \end{cor}

 \subsubsection{Choices of $\infty$ and $\varpi_\infty$}\label{choices'}
 
Let $\pi$ be an  irreducible  $\BC$-coefficient admissible representation of $\BB^\times$.  
\begin{lem}\label{chi} For a place $\infty$ of $F$ and a uniformizer $\varpi_\infty$ of $F_\infty$,
there exists a Hecke character $\tau$ of $F^\times$ such that $(\pi\otimes \tau)(\varpi_\infty)=1$.
Moreover, if the  central character of $\pi_\infty$   has finite order, such $\tau$ can be chosen to have finite order.
\end{lem}
\begin{proof} Let $a$ a square root of  the inverse of   $\pi(\varpi_\infty)$. It is enough to find $\tau$ such that $\tau_\infty(\varpi_\infty)=a$.
Let $V$ be an open compact subgroup of $\BA_F^\times$. The Pontryagin dual of 
the exact sequence 
$$1\to \varpi_\infty^{\BZ}\to F^\times\bsl\BA_F^\times/V\to F^\times\bsl\BA_F^\times/V\varpi_\infty^{\BZ}\to 1$$
is the exact sequence  $$1\to (F^\times\bsl\BA_F^\times/V\varpi_\infty^ \BZ )\hat{} \to (F^\times\bsl\BA_F^\times/V)\hat{}\to (\varpi_\infty^\BZ )\hat{}\to 1.$$
Choose $\tau_\infty\in  (\varpi_\infty^\BZ )\hat{}$ such that $\tau_\infty(\varpi_\infty)=a$.
Let $\tau$ be in the preimage of $\tau_\infty$ in  $(F^\times\bsl\BA_F^\times/V)\hat{}$, then $\tau_\infty(\varpi_\infty)=a$. This proves the first statement.
Since $F^\times\bsl\BA_F^\times/V\varpi_\infty^ \BZ $ is a finite group, if $a$ is a root of unity,   $\tau$ is of finite order.
This  proves the second statement.
\end{proof}

 Let  $\Omega$  a Hecke character   of $E^\times$.
  Let $\pi'=\pi\otimes \chi$, and $\Omega'=\Omega  \chi_E^{-1}$.
 Then $$L(s,\pi,\Omega)=L(s,\pi',\Omega'),\ L(1,\pi ,\ad)=L(1,\pi',\ad),$$
and similar relations hold for local  periods.

Assume that the Jacquet-Langlands correspondence of $\pi$ to $\GL_{2,F}$ is cuspidal.
 Choose the place $\infty$ in Lemma \ref{chi} inside the ramification set  $\Ram$ of $\BB$.
 Then 
  Theorem \ref{GZiota}    can be applied to get a  complex Gross-Zagier-Zhang formulas for $\pi$.  
 If $\pi$  is  an irreducible $\BQ$-coefficient representation of $\BB^\times$, then choose $\chi$ to be of finite order and define 
 $\pi'=\pi\otimes \chi$   by regarding $\chi$ as a finite dimensional
 $\BQ$-coefficient representation. Then 
  Theorem \ref{GZ}    can be applied to get a   Gross-Zagier-Zhang formulas for $\pi$.

      \section{Automorphic distributions}\label{The automorphic distributions}
We review the relative trace formulas in \cite{Jac87}\cite{JN}.
  In \ref{matchorb},  \ref{local orbital integrals0}, \ref{Split case}, we classify orbits  and define local orbital integrals. 
In \ref{Automorphic distributions},
   \ref{Decomposition under the pure  matching condition}, we define the automorphic distributions and study their  orbital     decompositions. 
    In \ref{Spectral decomposition of the automorphic distributions}, we study the spectral decompositions.

\subsection{Orbits}\label{matchorb}
Let $F$ be a  field,  $E$    a separable  quadratic  field extension   of $F$, and   $\Nm:E^\times\to F^\times$   the norm map.
Let $B$ be a quaternion algebra    over $F$ containing $E$.
By the Noether-Skolem theorem, there exists $j\in B^\times$   such that $B=E\oplus Ej$,  $j^2=\ep\in F^\times$ and $jz=\bar zj$ for $z\in E$, where $\bar z$ is the Galois conjugate of $z$.
Then $B$ is determined by $\ep$,
and we denote $B$ by $\left(\frac{E,\ep}{F}\right)$.
Embed $B$  in $
\RM_2(E)$ as an $F$-subalgebra via \begin{equation}a+bj\mapsto \begin{bmatrix}a&b\ep\\
 \bar b&\bar a\end{bmatrix}\label{(5.1)}\end{equation}  where $a,b\in E$. 
We  use the symbol $\det$ to denote the reduced norm map on $B^\times$. Then for $a,b\in E$, $\det (a+bj)=\Nm (a)-\Nm (b)\cdot\ep$.
 Under the embedding \eqref{(5.1)}, 
 the reduced norm map   
 is just the  determinant map.
 Note that if $c\in E^\times$,  then $cjcj=c\bar cjj=c\bar c\ep$, and $cjz=\bar z cj$. So
$B\cong \left(\frac{E,\Nm (c) \cdot \ep}{F}\right)$.
Thus we have a bijection $\ep\mapsto \left(\frac{E,\ep}{F}\right)$
 between $F^\times/ \Nm (E^\times)$ and the set 
of isomorphism classes of quaternion algebras over $F$ containing $E$.

For $\ep\in F^\times$, $G_\ep=B^\times$  where $B = \left(\frac{E,\ep}{F}\right)$. 
  Let $T_\ep\subset G_\ep$  be the subgroup induced by the canonical embeddings of $E^\times$, and let $Z_\ep\subset G_\ep$ be the center.
  Let $T_\ep\times T_\ep$ act on $G_{\ep }$ by $$(h_1,h_2)\cdot \gamma=h_1^{-1}\gamma h_2.$$
Let $ \inv_{T_\ep}(a+bj)= \ep  \Nm (b)/\Nm (a).$
 This defines  a bijection $$\inv_{T_\ep} :     T_\ep \bsl G_{\ep  } /T_\ep  \cong   \ep \Nm   E^\times \cup\{0,\infty\}    -\{1\}.
 $$   Regard  $G_\ep$ as a subgroup of $\GL_2(E)$   via \eqref{(5.1)}. Then     $$\inv_{T_\ep}\left(\begin{bmatrix}a&b\ep\\
 \bar b&\bar a\end{bmatrix}\right)=\ep\frac{  b \bar b }{a\bar a} .$$ 
 Let $\delta=\begin{bmatrix}a&b\ep\\
 \bar b&\bar a\end{bmatrix}\in G_\ep $. We say that $\delta$ is $T_\ep$-regular (regular for short) if $\inv_{T_\ep}(\delta)\in  \ep  \Nm (E^\times)-\{1\}$, equivalently   if
 $a  b  \neq 0$.  The stabilizer of the $T_\ep\times T_\ep$-action on a regular element is  the diagonal embedding of $Z_\ep$. 
 Let $G_{\ep,\reg}\subset G_\ep$  be the subset of regular elements.
 For  $x   \in \ep  \Nm (E^\times) -\{1\}$,  choose $b\in E^\times$ such that $x=\ep \Nm(b)$.
Let $$\delta(x) :=\begin{bmatrix}1&b\ep\\
 \bar b&	1\end{bmatrix} =1+bj\in G_{\ep,\reg}$$ which is  
  a representative of the corresponding $T_\ep\times T_\ep$-orbit, i.e.    $\inv_{T_\ep}(\delta(x))=x$.
 
  \begin{rmk}This definition   depends on the choice between $b$ and  $\bar b$.
  However, the orbital integrals of $\delta(x)$ (which will be defined in the next subsection) do not depend on this choice.  
 \end{rmk}

Define   $\inv_{T_\ep}' : G_\ep \to F$ by $\inv_{T_\ep}':= \inv_{T_\ep}/(1-\inv_{T_\ep})$, i.e.
\begin{equation}\label{inv'}\inv_{T_\ep}'(a+bj)=\ep\frac{  \Nm (b)}{\Nm (a+bj)}.  \end{equation}

Now we turn to the $\GL_2$-side. Let $G=\GL_{2,E}$. Let $\CS\subset G $ be the  subset of invertible Hermitian matrices over $F$ with respect to the  separable quadratic extension $E$. We also regard  $\CS$ as a subvariety  of $G$ if necessary.
Let $E^\times\times F^\times$ act on $\CS$ via $$(a,z)\cdot s= \begin{bmatrix}a&0\\
0&1\end{bmatrix}s\begin{bmatrix}\bar a&0\\
0& 1\end{bmatrix}z.$$
 There is a bijection $$\inv_{\CS} :   E^\times\bsl \CS/ F^\times   \cong  F^\times \cup\{0,\infty\}    -\{1\},
 $$$$\inv_{\CS}\left(\begin{bmatrix}a&b \\
 \bar b&d\end{bmatrix} \right )= \frac{ad}{b\bar b}.$$
Define $\gamma\in S$ to be   regular  if $\inv_{\CS}(\gamma)\in F^\times-\{1\}$.  The stabilizer of the $E^\times\times F^\times$-action on a regular element is  trivial.
 Let $\CS_{ \reg}\subset \CS$  be the subset of regular elements.
 For  $x\in F^\times  -\{1\}$, 
let $$\gamma(x) =  \begin{bmatrix}x&1\\
1&1\end{bmatrix}\in \CS_\reg $$ which is  
  a representative of the corresponding regular $E^\times\times F^\times$-orbit.     
  Define   
\begin{equation}\inv_{\CS}'=\frac{\inv_{\CS}}{1-\inv_{\CS}}:\CS\to F  \label{xi'} .\end{equation}
 
 %

Let $g\in G$ act  on $\CS$ by 
$g\cdot s:=gs  \bar g^t$ where $ \bar g^t$ is the Galois conjugate of the transpose of $g$.  
Let $H_0\subset G$ be  the unitary group associated to $$w=\begin{bmatrix}0&1\\
 1&0\end{bmatrix},$$ i.e. the stabilizer of $w$ in  $G$.
Let $H\subset G$ be the  similitude unitary group associated to $w$, and 
 let $\kappa$ be the  similitude character.
    If $F$ is a local field, for $f\in C_c^\infty(G )$, let  $ \Phi_{f}\in C_c^\infty(\CS )$   supported on $Gw$ such that  
    \begin{equation}\label{Phi}\Phi _f(g\cdot w  )=\int_{H_0 } f(gh) dh.\end{equation}  If $F$ is a global field, apply this definition  to $f\in C_c^\infty(G(\BA_F) )$.
 
\subsection{Local orbital integrals: nonsplit case}\label{local orbital integrals0}
 Let $F$ be a non-archimedean local field.
 Let $E$ be a separable  quadratic  field extension,  $\eta$   the associated quadratic character of $F^\times$.

\begin{lem}\label{goodx} Let $\gamma\in \CS$, then $\gamma\in Gw$ if and only if $-\det(\gamma)\in \Nm(E^\times)$.
For  $x\in F^\times  -\{1\}$, 
 $\gamma(x)   \in Gw$ if and only if $ 1-x\in \Nm(E^\times)$.
\end{lem}

Let $\Omega$ be a continuous character of    $E^\times$  and   $\omega$ be its restriction to $F^\times$. 
 Endow $\CS\subset G$ with the subspace topology.   For  $\Phi\in  C_c^\infty(\CS)$, $\gamma\in \CS_\reg$, and $s\in \BC$, define the orbital integral   \begin{equation}\CO(s, \gamma,\Phi):=\int_{E^\times}\int_{F^\times} \Phi\left(z\begin{bmatrix}a&0\\
0&1\end{bmatrix}\gamma\begin{bmatrix}\bar a&0\\
0&1 \end{bmatrix} \right)\eta\omega^{-1}(z)\Omega^{-1}(a) |a|_E^sd  zd  a. \label{int0}
\end{equation}
For  $x\in F^\times-\{1\}$,  define $\CO(s, x,\Phi):=\CO(s, \gamma(x),\Phi).$
Let $\CO(x,\Phi):=\CO(0, x,\Phi).$


 \begin{lem}\label{intwell}   The integral \eqref{int0} converges absolutely and defines a holomorphic  function on $s$. Its derivative at $s=0$ is  the following 
 convergent integral: \begin{equation*}\CO'(0, x,\Phi)=\int_{E^\times}\int_{F^\times} \Phi\left( \begin{bmatrix}za\bar a x&za\\
 z\bar a&z\end{bmatrix}\right) \eta\omega^{-1}(z)\Omega^{-1}(a)\log |a|_E  d  zd  a.
\end{equation*}
\end{lem}
\begin{proof} 
  The map $$E^\times \times F^\times\to G,\ (a,z)\mapsto  z\begin{bmatrix}a&0\\
0&1\end{bmatrix}\gamma(x)\begin{bmatrix}\bar a&0\\
0&1 \end{bmatrix}   $$ is a  closed embedding since $\gamma(x)$ is  regular. 
Therefore $$\supp (\Phi)\cap\left \{z\begin{bmatrix}a&0\\
0&1\end{bmatrix}\gamma\begin{bmatrix}\bar a&0\\
0&1 \end{bmatrix} :(a,z)\in E^\times \times F^\times\right \}$$ is compact, and so is its preimage in $E^\times \times F^\times.$ The lemma follows.  \end{proof}

 As an example, we have the  following lemma which is easily proved by a direct computation.
Let $K=\GL_2(\CO_E)$ be the standard maximal compact subgroup of $G$.
   
   \begin{lem}\label{sint=1} Let  $\Phi=1_{K\cap \CS}$, and let $\eta$ and $\Omega$ be unramified. Suppose $v(x)=v(1-x)=0$, then $$\CO(s,\gamma(x),\Phi ) =\Vol(\CO_F^\times)\Vol(\CO_E^\times) .$$\end{lem}

  Now we turn to $ G_\ep  $. We do not need the parameter $s$.
 For $f\in C _c^\infty (G_\ep)$  and 
$\delta\in {G_{\ep, \reg}} $, define  \begin{align*}\CO(  \delta,f):
=\int_{T_\ep/Z_\ep}\int_{T_\ep } f(h_1^{-1}\delta h_2) \Omega(h_1) \Omega^{-1}(h_2)  dh_2dh_1. 
\end{align*}
For an open subgroup  $\Xi$ of $Z_\ep$  
  such that $\Omega$ is $\Xi$-invarian,  define  \begin{align*}\CO_\Xi(  \delta,f):= \int_{T_\ep/Z_\ep}\int_{T_\ep/\Xi } f(h_1^{-1}\delta h_2) \Omega(h_1) \Omega^{-1}(h_2)  dh_2dh_1. 
\end{align*}
Finally, for $x \in \ep\Nm (E^\times)-\{1\}$,  let
$\CO(x,f):= \CO(\delta(x),f)  $
 and 
  $\CO_\Xi(x,f):= \CO_\Xi(\delta(x),f)$.
 

  It is easy to check that the integrals defining $\CO(\delta ,f) , \CO_\Xi(  \delta,f)$ converge absolutely.

  \begin{defn} \label{matchingdef}(1) Let $\Phi\in C_c^\infty (\CS)$, and $f_\ep\in C_c^\infty(G_\ep)$. We say  that $\Phi$ and   $f_\ep $ have matching orbital integrals (match, for short) if  for    every  $x\in \ep\Nm(E^\times)-\{1\}$, the following equation holds:
$$   \CO( x,\Phi)=\CO(x,f_\ep).$$

(2) 
We say that $\Phi$ and  $(f_\ep)_{\ep\in F^\times/ \Nm (E^\times)}$  have matching orbital integrals  (match, for short)       if $\Phi$ and   $f_\ep $ match for each $\ep$.

(3) We say that  $\Phi$ and   $f_\ep$ purely match if $\Phi$ matches $(f_\ep, 0)$.

(4) For $f\in C_c^\infty(G )$,  the  definitions in (1) (2) (3)  apply if they apply to $\Phi_{f}$.

(5)   
  The  definitions in (1) (2) (3) (4) apply to $f_\ep\in C _c ^\infty (G_\ep/\Xi )$ if corresponding conditions hold when $\CO(x,f_\ep)$  is replaced by $\CO_\Xi(x,f_\ep)$.
\end{defn} 
 
 \subsection{Local orbital integrals: split case}\label {Split case}  Let $F$ be a non-archimedean local field as in \ref{local orbital integrals0}, and   let $E=F\times   F$.
 Let $a\mapsto \bar a$ be the standard involution on $E$ w.r.t. $F$. 
     Let $G=\GL_2(E)$. Then there is a canonical isomorphism \begin{equation}G \cong\GL_2(F)\times \GL_2(F).\label{541}\end{equation}
   With respect to the   involution $a\mapsto \bar a$,   the   space   $\CS $ of invertible hermitian matrices  and  the unitary group (resp. similitude unitary group) $H_0$  (resp. $H$)  associated to    $w=\begin{bmatrix}0&1\\
 1&0\end{bmatrix}$ as in \ref{matchorb}  is well defined. 
   Under the  isomorphism \eqref{541},   $H_0\subset G$  is the subgroup of  elements of the form $ (h,w(h^t)^{-1} w).$
Thus  we have an isomorphism \begin{equation}H_0\cong \GL_2(F) \label{543}\end{equation}
  given by $ (h,w(h^t)^{-1} w)\mapsto h.$
  Similarly,  under the isomorphism  \eqref{541},  $\CS\subset G$  is the subset of  elements of the form $(s,s^t).$
Thus   
  we have an isomorphism  \begin{equation}\CS\cong \GL_2(F) \label{544}\end{equation} 
  given by $ (s,s^t) \mapsto s.$

    Let $f_1,f_2\in C_c^\infty(\GL_2(F))$, which induce a function  $f_1\otimes f_2\in C_c^\infty(G)$ under the isomorphism \eqref{541}. Define $\Phi_{f_1\otimes f_2}\in C_c^\infty(\CS)$ as in   \eqref{Phi}.    Regarded  as a function on $\GL_2(F)$ via  \eqref{544},  $\Phi_{f_1\otimes f_2}  $  has the following expression:   let $s\in   \GL_2(F)$, and $(g,g_2)\in G$ such that $$(g,g_2)(w,w)(g_2^t,g^t)=(s,s^t),$$
then
 \begin{align*} \Phi_{f_1\otimes f_2}(s)&=\int_{\GL_2(F)} f_1(gh )f_2(g_2w(h ^t)^{-1} w)dh . \end {align*}
   Take $g_2=1$ , we have 
 \begin{align} \Phi_{f_1\otimes f_2}(g  w)=\int_{\GL_2(F)} f_1(gh)f_2(w(h^t)^{-1}w)dh=(f_1\ast \tilde f_2)(g)\label{phi12}, \end{align}
   where $\tilde f_2$ is defined by
    $$\tilde f_2(g):=f_2(wg^tw).$$

Let  $\Omega=\Omega_1\boxtimes \Omega_2$ be a  continuous unitary character of $E^\times=F^\times\times F^\times$. Then  $\omega:=\Omega_1\otimes \Omega_2$ is the restriction of $\Omega$ to the diagonal embedding  $F^\times\incl E^\times$. Let $\Phi\in C_c^\infty(\CS)$.  For $x\in F^\times-\{1\}$, define $\CO(x,\Phi) $ as in \eqref{int0} (with $s=0$).
Regarded $\Phi$ as a function on $\GL_2(F)$ via  \eqref{544}. Using the isomorphism  \eqref{541},   we have
\begin{equation}\CO(x,\Phi)=\int   _{F^\times}\int   _{F^\times}\int   _{F^\times}
 \Phi\left(\begin{bmatrix}z&0\\ 0&z\end{bmatrix}\begin{bmatrix}a&0\\ 0&1\end{bmatrix}\begin{bmatrix}x&1\\ 1&1\end{bmatrix}\begin{bmatrix}b&0\\ 0&1\end{bmatrix}\right)\Omega_1^{-1}(a)\Omega_2^{-1}(b)\omega^{-1}(z) d  a d   bd   z.\label{545}\end{equation}
  Similar to Lemma \ref{sint=1}, we have the following lemma.
     \begin{lem}\label{sint=1'}Let  $\Phi=1_{K\cap \CS}$, $\eta$ and $\Omega$ be unramified. Suppose $v(x)=v(1-x)=0$, then $$\CO(s,\gamma(x),f ) =\Vol(\CO_F^\times)\Vol(\CO_E^\times) .$$\end{lem}

Let   $G_\ep=\GL_2(F)$.  Here  $\ep$ is just an abstract subscript. 
 Let $T_\ep$ be the   diagonal torus of $\GL_2(F)$.  
For $f_\ep\in C_c^\infty(G_\ep)$, define $\CO(x,f_\ep) $ as in \ref{local orbital integrals0}. 
 Then  \begin{equation*}\CO(x,f_\ep)=\int   _{F^\times}\int   _{F^\times}\int   _{F^\times}
 f_\ep\left(\begin{bmatrix}c^{-1}&0\\ 0&1\end{bmatrix}\begin{bmatrix}x&1\\ 1&1\end{bmatrix}\begin{bmatrix}a&0\\ 0&b\end{bmatrix}\right)\Omega_1(c)\Omega_1^{-1}(a)\Omega_2^{-1}(b)d  cd  a d  b.\label{546}\end{equation*}
 Change variable:  $b=at $, 
 we have 
\begin{equation}\CO(x,f_\ep)=\int   _{F^\times}\int   _{F^\times}\int   _{F^\times}
 f_\ep\left(\begin{bmatrix}a&0\\ 0&a\end{bmatrix} \begin{bmatrix}c&0\\ 0&1\end{bmatrix}   \begin{bmatrix}1&x\\ 1&1\end{bmatrix} \begin{bmatrix}t&0\\ 0&1\end{bmatrix} w\right)\Omega_1^{-1}(c)\Omega_2^{-1}(t)\omega^{-1}(a)d  c d  t d  a.\label{547}\end{equation}

Define $\Phi\in C_c^\infty (\CS)$  and $f_\ep\in C_c^\infty(G_\ep)$ to have matching orbital integrals (match, for short) if  for    every  $x\in F^\times-\{1\}$, $   \CO( x,\Phi)=\CO(x,f_\ep).$
  By  \eqref{545} and \eqref{547},   we have the following lemma.
  \begin{lem} \label{afneq000} 
  Let   $f_\ep(g)=\Phi(gw),$ then 
  $f_\ep$ matches $\Phi$.
  
  \end{lem}
  
\begin{lem} \label{afneq00} 
 (1) Let $f_\ep \in  C_c^\infty(G_\ep) $. Then there 
 exists $ f_1,f_2\in C_c^\infty(\GL_2(F))$   such that
\begin{equation}f_\ep=f_1\ast \tilde f_2  \label{splitplaces}.\end{equation} 

(2) For every   $f_1,f_2$ as in (1), $f_\ep(g)=\Phi_{f_1\otimes f_2}(gw)$.

(3)  For every  $f_1,f_2$ as in (1), $f_\ep$  matches $\Phi_{f_1\otimes f_2}$.
\end{lem}
   \begin{proof}  
(1) Suppose $f_\ep$ is bi-$U$-invariant for a open compact subgroup $U\subset G_\ep $,   choose $f_1=f_\ep$, $f_2=   {1_U}/\Vol(  U)$.
(2) is a restatement of  \eqref{phi12}. 
(3) follows from (2) and Lemma \ref{afneq000}.\end{proof}

          \subsection{Automorphic distributions}\label{Automorphic distributions}
     We come back to the global situation. Let $F$ be a global function field, and let $E$ be a separable quadratic  field  extension.
            Let $\Omega$ be  a   Hecke character of    $E^\times$,  and let $\omega$ be its restriction to $\BA_{F}^\times$, and  $\omega_E:=\omega\circ\Nm$. 
       
 \subsubsection{}    Let $G=\GL_{2,E}$, let $A$ be the diagonal torus  and   let $Z$  be the center of $G$. Let $\CS\subset G$ be the subvariety  of 
invertible Hermitian matrices.                             Let $H,H_0,\kappa $   be  defined as in the end of \ref{matchorb}.
For
$f'\in  C_c^\infty(G(\BA_E))$, define a  kernel function on $G(\BA_E)\times G(\BA_E)$:
\begin{equation}\label{autker}K(x,y)=K_{\omega_E,f'}(x,y):=\int_{ Z(E)\bsl Z(\BA_E) }\left(\sum_{g\in G(E)} f'(x^{-1}gzy)  \right)\omega_E^{-1}(z)dz.\end{equation}

\begin{lem} \label{finitesum3}  The inner  sum  of \eqref{autker} has only finitely many  nonzero terms.
\end{lem}
\begin{proof}  
If $f'(x^{-1}gy)\neq 0$, $g\in x\supp (f') y^{-1}$ which is compact. Since $G(E)\subset \GL_2(\BA_E)$ is closed and discrete  with the subspace topology,  
$x\supp (f') y^{-1}\cap G(E)$ is   finite.
\end{proof}

 For  $a= \begin{bmatrix}a_1&0\\ 0&a_2\end{bmatrix}$,  let $\Omega(a)=\Omega(a_1/a_2)$.
For $s\in \BC$,
formally define the   distribution $\CO(s,\cdot)$ on $G(\BA_E)$ by assigning to $f' \in C_c^\infty( G(\BA_E))$ the integral \begin{align}\CO( s, f')=  \int_{Z(\BA_E)A(E)\bsl A(\BA_E)} \int_{Z(\BA_E)H(F)\bsl   H(\BA_F)} K_{\omega_E,f'}(a,h)\Omega(a) \eta\omega^{-1}(\kappa(h) ) |a|_E^sdh da\label{45} .\end{align}

 \begin{lem} \label{5.5.4} Assume $\Phi_{f'}(g)=0$ for $g\in   \BA_E^\times (\CS(F)-\CS(F)_\reg)  \BA_F^\times  .$ Then 
 the integral $\CO( s, f')$ in \eqref{45}    converges absolutely.
 \end{lem}   
  For $\Phi\in C_c^\infty(\CS(\BA_F))$ and  $x\in F^\times-\{1\}$, define $$\CO(s,x,\Phi):=\int _{\BA_E^\times}\int_{\BA_F^\times} \Phi\left( \begin{bmatrix}za\bar a x&za\\
 z\bar a&z\end{bmatrix}\right)  \eta\omega^{-1}(z)\Omega^{-1}(a)|a|_E^s d  zd  a.$$ 
   By the same reasoning as in the proof Lemma \ref{intwell}, we have the following lemma.
 \begin{lem}\label{AV}
Let $\Phi\in C_c^\infty(\CS(\BA_F))$.  The integral  $\CO(s,x,\Phi)$ converges absolutely.
 \end{lem}

 \begin{proof} [Proof of Lemma \ref{5.5.4}]
Formally, we have \begin{align} 
 \CO( s, f')=\sum_{x\in F^\times-\{1\}}\CO(s,x,\Phi_{f'}).\label{hohoho}
\end{align}
 Extend  $\inv_\CS'$ defined in \eqref{xi'} to $\CS(\BA_F)\to \BA_F$.   Since $\inv_\CS'(\supp \Phi_{f'})$ is  compact, it has finite intersection with  the closed and discrete subset $F=\inv_\CS'(\CS(F))$ of $\BA_F$. In particular, the right hand side is a finite sum. 
The by Lemma \ref{AV}, 
and  Fubini's theorem,
the integral $\CO( s, f')$   converges absolutely and  \eqref{hohoho}  holds.
  Thus the lemma follow.\end{proof}

 \begin{asmp}   We only use pure tensors $f'\in C_c^\infty(G(\BA_E))$
 and  $\Phi\in C_c^\infty(\CS(\BA_F))$:  $$\Phi=\bigotimes_{v\in |X|}\Phi_v,\ f'=\bigotimes_{v\in |X|}f_v',$$  where $\Phi_v\in C_c^\infty(\CS(F_v))$
   and 
  $f'_v\in C_c^\infty(G(E_v))$.
  \end{asmp}
 Lemma \ref{sint=1}, \ref {sint=1'} show that $\CO(s,x,\Phi_v )=1$ for almost all $v$.  
  Thus $$\CO(s,x,\Phi) = \prod_{v\in |X|}\CO(s,x,\Phi_v) ,$$ and the infinite product  converges absolutely. 
And the following sum is a finite sum: \begin{align} \CO'(0,x,\Phi)=\sum_{v\in |X|}\CO'(0,x,\Phi_v)\CO(x ,\Phi^v)\label{513} .\end{align}

 \begin{asmp}\label{freg} Assume that $\Phi_{f'}(g)=0$ for $g\in   \BA_E^\times (\CS(F)-\CS(F)_\reg)  \BA_F^\times  .$
\end{asmp}

  By \eqref{hohoho} and  \eqref{513}, we have a decomposition
  \begin{align}\CO'( 0, f')&=\sum_{x\in F^\times-\{1\}} \sum_{v\in |X|}\CO'(0,x,\Phi_{f',v})\CO(x ,\Phi_{f'}^v) , \label{decder} \end{align}    
 By the same reasoning as in the proof of  Lemma \ref{5.5.4},  there
 are only  finitely many $x$ such that   $\CO'(0,x,\Phi_{f',v})\CO(x ,\Phi_{f'}^v)$ is nonzero for some $v$. In particular, the sum is a finite sum.
  
\subsubsection{} Let  $B$ be  a quaternion algebra  over $F$.
For  $f\in  C_c^\infty (B^\times(\BA_F))$, define a kernel function on $B^\times(\BA_F)\times B^\times(\BA_F)$: 
\begin{equation} k (x,y):=\sum_{g\in B^\times} f(x^{-1}gy).\label{kxy}\end{equation}
 Define  a distribution  $\CO(\cdot)$ on $B^\times(\BA_F)$ by  assigning to $f\in C_c^\infty( B^\times(\BA_F))$ the integral $$\CO(  f):=\int_{E^\times  \bsl \BA_E^\times / \BA_F^\times} \int_{E^\times  \bsl \BA_E^\times  }
k(h_1,  h_2) \Omega(h_1)\Omega^{-1}(h_2) dh_2dh_1.$$

 \begin{asmp}\label{freg'} 
Assume that  $f$ vanishes on $\BA_E^\times (B^\times-B^\times_\reg) \BA_E^\times $.
\end{asmp}
\begin{lem} \label{decquat}  (1)  The integral defining $\CO( f)$ converges absolutely.

(2) We have a decomposition \begin{align*} 
\CO(   f) =\sum_{x\in\ep  \Nm ( E^\times)-\{1\}}\CO(x,f ),
\end{align*}
where $$\CO( x,f)=\int_{  \BA_E^\times / \BA_F^\times}\int_{  \BA_E^\times  }  f(h_1^{-1}\delta(x) h_2)\Omega(h_1)\Omega^{-1}(h_2) dh_2dh_1 .$$
  \end{lem}

We  also need the following modification of $\CO(\cdot)$.
For  $\infty\in |X|$ and  $\Xi_\infty \subset F_\infty^\times$,   the image of $B^\times\incl B^\times(\BA_F)/\Xi_\infty$   is closed and discrete. 
  Thus, for   $f   \in  C_c^\infty (B^\times(\BA_{F })/\Xi_\infty)$, the formula \eqref{kxy} gives a well-defined function 
$ k (x,y) $ 
  on $B^\times(\BA_F)\times B^\times(\BA_F)$.

\begin{defn}\label{COXi}  The distribution $\CO_{\Xi_\infty}$ on $B^\times(\BA_F)/\Xi_\infty$    assigns to $f\in C_c^\infty( B^\times(\BA_F)/\Xi_\infty)$ the integral  $$\CO_{\Xi_\infty}(  f):=\int_{E^\times  \bsl \BA_E^\times / \BA_F^\times}\int_{  E^\times\bsl \BA_E^\times/\Xi_\infty }  k(h_1,  h_2) \Omega(h_1)\Omega^{-1}(h_2) dh_2dh_1 .$$
 
\end{defn}

  \subsection{Pure  matching conditions}
   \label{Decomposition under the pure  matching condition}
   \begin{asmp} \label{fpure}  Assume that $f\in \CH_\BC$  
 is a pure tensor: $f=\bigotimes_{v\in |X|}f_v.$
  \end{asmp}
  Recall that $|X|_s\subset |X|$ is the subset of places split in $E$.
\begin{defn} \label{globalpurematch} Let $f\in   \CH_\BC$ or $ C_c^\infty ( B^\times(\BA_F) )  $ 
and   $\Phi\in  C_c^\infty (\CS (\BA_F))$.  We say that $f$ and $\Phi$ purely match if they  purely match 
at   $v\in |X|-|X|_s$  and match at $v\in |X|_s$.

\end{defn}

Let $f\in   \CH_\BC$ and $\Phi\in  C_c^\infty (\CS (\BA_F))$ purely match. 
Suppose that there exists $f'\in C_c^\infty(G(\BA_E)$ satisfying Assumption \ref{freg} such that $\Phi=\Phi_{f'}$.
We  rearrange the decomposition of  $\CO'( 0, f')$ in \eqref{decder} according to the decomposition  $$F^\times-\{1\}=\coprod 
 \inv_{E^\times}( B ^\times_\reg) ,$$
 where the union is over all   quaternion algebras over $F$ containing $E$ as an $F$-subalgebra.
Let $B$ be such a quaternion  algebra, and  $x\in  \inv_{E^\times}( B ^\times_\reg) $. Consider  $$\CO'(0,x,\Phi)=\sum_{v\in |X|}\CO'(0,x,\Phi_v)\CO(x ,\Phi^v) .$$
 If $B $ and $ \BB$ are not isomorphic at more than one place   (which must be in $|X|-|X|_s$), then by the pure matching condition,    for every place $v$, the infinite product
$$\CO(x ,\Phi^v)=\prod_{u\neq v} \CO(x ,\Phi_u) $$ 
contains at least one  local component 
 with value 0.
So $\CO(x ,\Phi^v)=0$, thus $\CO'(0,x ,\Phi)=0.$

Suppose $B=B(v)$  is a $v$-nearby quaternion algebra of $\BB$
for some place $v\in |X|-|X|_s$.  For $x\in  \inv_{E^\times}( B ^\times_\reg) $,
let   $\CO(x,f^v) $  be the orbital integral defined by 
regarding   $f^v$ as a function on $B^\times(\BA_F^v)$.
By the pure matching condition, $\CO(x,\Phi^u)\neq 0 $ only if $u=v$.   In this case $ \CO(x,\Phi^v)=\CO(x,f^v)$.
To sum up, we have the following lemma.
\begin{lem}\label{decpure} There is a decomposition  
$$\CO'( 0,f')=\sum_{v\in |X|-|X|_s}\sum_{x\in\inv_{E^\times}( B(v)^\times_\reg) }  \CO'(0,x,\Phi_v) \CO(x,f^v).$$
 \end{lem}

  Let $f\in  C_c^\infty ( B^\times(\BA_F) )$ and $\Phi\in  C_c^\infty (\CS (\BA_F))$ purely match as in Definition \ref{globalpurematch}. 
 \begin{lem}\label{decpure'} There is a decomposition
$$\CO( f')= \sum_{x\in\inv_{E^\times}( B ^\times_\reg) }   \CO(x,f).$$ 

\end{lem}

      \subsection{Spectral decomposition}\label{Spectral decomposition of the automorphic distributions}
       Let $\CA_c(G,\omega_E^{-1})$ be the set of all $\BC$-coefficient  cuspidal representations  of   $\GL_2(\BA_E)$ with  central character $\omega_E^{-1}$. 
       From now on, we always identify  the complex conjugate of 
       $\sigma\in \CA_c(G,\omega_E^{-1})$ 
       with the contragradient representation  $\tilde\sigma$ of $\sigma$ via the Petersson pairing.  
       Let $\CF_c(G,\omega_E^{-1})$ be the space of  $\BC$-valued cusp forms which transform by the Hecke character $\omega_E^{-1}$ under the action of $Z(\BA_E)$.
For $f'\in C_c^\infty(\GL_2(\BA_E))$,   
 let $K (x,y)  $ be the associated kernel function (see  \eqref{autker}).  
  Let   $$K_{c}(x,y):=\sum_{\phi}\rho(f')\phi(x)\bar\phi(y)$$ where the sum is over an orthonormal basis of $\CF_c(G,\omega_E^{-1})$ w.r.t the Petersson pairing. By \cite[Proposition 10.5]{JL},
   this sum contains only finitely many nonzero terms.
Then   $K_{c}(x,y)$ is the kernel function of the Hecke action of $f'$ on $\CF_c(G,\omega_E^{-1})$.  
  For $\sigma\in \CA_c(G,\omega_E^{-1})$, let $$K_{\sigma}(x,y)=\sum_{\phi}\sigma(f')\phi(x)\bar \phi(y),$$
 where the sum is over an orthonormal basis of $\sigma$.   By the finiteness of above sums, for $*=c,\sigma $,  $K_*$ has compact support in $(Z(\BA_E)\GL_2(E)\bsl\GL_2(\BA_E))^2$. In particular,
\begin{equation}  \CO_* (s, f') :=  \int_{Z(\BA_E)A(E)\bsl A(\BA_E)} \int_{Z(\BA_E)H(F)\bsl   H(\BA_F)} K_*(a,h)\Omega(a) \eta\omega^{-1}(\kappa(h) ) |a|_E^sdh da  . \label{co} 
\end{equation} converges absolutely.
By the   multiplicity one theorem, we have 
 $$K_{c}(x,y)=\sum_{\sigma\in \CA_c(G,\omega_E^{-1})  }  K_{\sigma}(x,y).$$   Similarly, define the kernel functions $K_\Sp$ and $ K_\Eis$ of the Hecke actions of $f'$ on the residual spectrum and continuous 
  spectrum of representations with   central character $\omega_E^{-1}$ (see \cite[5.5]{Jac87}, \cite[8.2]{JN}).  
 Then $$K (x,y)=K_c(x,y)+K_\Eis(x,y)+K_\Sp(x,y).$$

For $\sigma\in\CA_c(G,\omega_E^{-1})$,    $\phi\in \sigma$, $\chi:\BA_E^\times/E^\times\to \BC^\times$, define the toric period
$$\lambda (s,\phi)=\int_{Z(\BA_E)A(E)\bsl A(\BA_E)}\phi(a)\Omega(a) |a|_E^sda$$
which absolutely converges for all $s\in \BC
$.
 Also define the base change period  $$\CP(   \phi)=    \int_{Z(\BA_E)H(F)\bsl   H(\BA_F)} \phi(h) \eta\omega (\kappa (h) ) dh   .$$
Then \begin{equation} \CO_\sigma (s,f') =\sum_{\phi}  \lambda(s,\sigma( f')\phi)\overline {\CP (  \phi)}  \label{COSigma}\end{equation}
where the sum is over an orthonormal basis of $\sigma$.

However, for $*=   \Eis,    \Sp$, the same  integral as in  \eqref{co} needs to be regularized. Apply the truncation operators   $\Lambda_d^{T}$ and  $ \Lambda_m^{T}$ (\cite[Section 8]{JN}) on the kernel
$K_*$ and take \begin{equation*}  \CO_{*,T_1,T_2} (s, f') := \int_{Z(\BA_E)H(F)\bsl   H(\BA_F)}   \int_{Z(\BA_E)A(E)\bsl A(\BA_E)} \Lambda_d^{T_1}  \Lambda_m^{T_2} K_*(a,h)\Omega(a) \eta\omega^{-1}(\kappa(h) ) |a|_E^sda dh.   \end{equation*}  
Define \begin{equation} \CO_{* } (s, f'):=\lim_{T_1\to 0}\lim_{T_2 \to 0} \CO_{*,T_1,T_2}' (0, f') \label{co*} .
\end{equation} 
Define a regularized integral $\CO_{\reg } (s, f')$ similar to $\CO_{* } (s, f')$ with $K_* $ replaced by  $K .$
\begin{lem}\label{modi}
Assume that  $f'$ satisfies Assumption \ref{freg}, 
then  $\CO_{\reg} (s, f')=\CO  (s, f')$. \end{lem}
\begin{proof} In the number field setting, this is implied by \cite[Lemma 10]{JN} and \cite{Jac87} (see also the end of \cite[Section 3]{JN}). We modify the proof in our  setting.  There is another truncation 
$\Lambda_c^{T}$ in  \cite{Jac87}. 
 By Lemma \ref{5.5.4} and \cite[p. 49]{JN}, the regularized integral of $K$ using $\Lambda_c^{T}$ is the same as $\CO  (s, f')$.
Then we use  discussion in \cite[8.1]{JN} with the following modifications: 
\begin{itemize}\item[(1)]  "rapid decreasing" in \cite[p. 71]{JN} is replaced by ``of compact support" (see \cite[I.2.9]{MW});
\item[(2)] to prove \cite[Lemma 9]{JN}, we note that the function $m(g)$ on $Z(\BA_E)\GL_2(E)\bsl\GL_2(\BA_E)
$ defined below
 \cite[Lemma 9]{JN} 
has compact support (which follows from (1) and definition). 
\end{itemize}
\end{proof}
Below,  when  citing a result from \cite{Jac87}\cite{JN}, we mean the same result hold in our setting with a  modification 
of the proof, as we do in the proof of Lemma \ref{modi}.
By \cite[(5.7)]{Jac87}\cite[p. 81, p. 83]{JN},     $\CO_\Sp (s, f')=0$. 
Thus 
\begin{equation}\CO(s, f')=\sum_{\sigma\in \CA_c(G,\omega_E^{-1})} \CO_\sigma (s, f')+ \CO_\Eis(s, f') .  \label{COSigma'}\end{equation}
 
   Let $N$ be the upper unipotent subgroup and let $K$ be the standard  maximal compact compact subgroup of $G(\BA_E)$. 
 We use the standard notation:  for $t\in \BC$ and $g\in G(\BA_E)$, let
$$e^{\pair{t+\rho,H(g)}}:=|a/d|_E^{1/2+t} $$
if $\begin{bmatrix}a&0\\ 0&d\end{bmatrix}^{-1}g \in NK.$
 Fix   $\alpha\in \BA_E^\times$ with $|\alpha|=q$.   Let $\Lambda_E$ be the set of all Heche characters $\lambda$ of $E^\times$ such that $\lambda(\alpha)=1$ (which corresponds to the condition ``normalized" in \cite{JN}) modulo the equivalence relation $\lambda\cong \lambda^{-1}\omega_E^{-1}$.  
For $\lambda\in \Lambda_E$,  $t\in \BC$, the admissible representation of $G(\BA_E)$ associated to the data $(t,\lambda,\lambda^{-1}\omega_E^{-1})$ is realized
 on the space of smooth functions $\phi: G( \BA_E )\to \BC$ such that
 $$\phi\left(\begin{bmatrix}a&b\\ 0&d\end{bmatrix} x\right)=\lambda(a)\lambda^{-1}\omega_E^{-1}(d)\phi(x) $$
and the action of $g\in G(\BA_E)$ is given by 
$$g\cdot \phi(x)=e^{\pair{t+\rho,H(g)}}\phi(xg).$$

Let $S$ be a finite subset of $|X|$. 
Let $\CT'^S $ be the   spherical Hecke algebra  of $G$ away from    the set of places  of $E$ over $S$, i.e.  the algebra of  bi-$K^S$-invariant functions in $C_c^\infty(G(\BA_E^S)$).  
  
\begin{thm}[{\cite[(5.6)]{Jac87}\cite[Theorem 3]{JN}}] \label{Eisterm} Suppose that $S$   contains all  ramified places  of $E/F$  and  all 
places below ramified places of $\Omega$.  Given  $ f'_{S }\in C_c^\infty(G(\BA_{E,S }))$, there exists  
\begin{itemize}
 \item[(1)]
for each  $\chi\in\Lambda_E$ which is a lift of a Hecke character of $F^\times$,
  a continuous function $\Phi_\chi(s,t)$ on $\BC\times [0,2\pi i/\log q]$, entire on the first variable with $\partial_s\Phi_\chi(s,t)$ continuous. Moreover, only finitely many $\Phi_\chi\neq 0$;
\item[(2)] 
for each  $ {\xi} \in {\Lambda}_E$ which is not  a lift of a Hecke character of $F^\times$ and has restriction $\eta\omega^{-1}$ to $\BA_F^\times$, 
  an entire function $\Phi_ {\xi}(s )$. 
  Moreover, only finitely many $\Phi_ {\xi}\neq 0$;
\item[(3)] an entire function $\Phi_\Omega(s )$;
 \end{itemize}
  such that  as a linear functional on $\CT'^S$, we have

   \begin{equation}\begin{split} \CO_\Eis (s, f'_Sf'^S)&=\sum_{\chi}\int_{0}^{2\pi i/\log q}\Phi_\chi(s,t) \widehat{f'^S}(t,\chi,\chi^{-1}\omega_E^{-1}) dt\\
&+\sum_{ {\xi}}\Phi_ {\xi}(s) \widehat{f'^S}(0, {\xi}, {\xi}^{-1}\omega_E^{-1}) \\
&+\Phi_\Omega(s) \widehat{f'^S}(1/2, \Omega, \Omega)  .  \label{COEis}
     \end{split}
 \end{equation}
Here  $\widehat{f'^S}$ is the Satake transform of $f'^S$, 
$\widehat{f'^S}(t,\chi,\chi^{-1}\omega_E^{-1})$ is nonzero only when representation of $G(\BA_E^S)$ associated to the data $(t,\chi^S,\chi^{S,-1}\omega_E^{S,-1})$ is unramified, and in this case,  $\widehat{f'^S}(t,\chi,\chi^{-1}\omega_E^{-1})$ is the 
  value of $\widehat{f'^S}$ on the  Satake parameter of the  representation of $G(\BA_E^S)$ associated to the data $(t,\chi^S,\chi^{S,-1}\omega_E^{S,-1})$, etc..

\end{thm}  
\begin{cor}\label{Eisterm'} As a linear functional on $\CT'^S$, we have    \begin{align*}(\CO_\Eis)' (0, f'_Sf'^S)&=\sum_{\chi}\int_{0}^{2\pi i/\log q}\Phi_\chi'(0,t) \widehat{f'^S}(t,\chi,\chi^{-1}\omega_E^{-1}) dt\\
&+\sum_{ {\xi}}\Phi_ {\xi}'(0) \widehat{f'^S}(0, {\xi}, {\xi}^{-1}\omega_E^{-1}) \\
&+\Phi_\Omega'(0) \widehat{f'^S}(-1/2, \Omega, \Omega)   .
\end{align*}
    \end{cor}
 
    
Let  $\sigma=\sigma_{\xi}$ be  the   representation of $G(\BA_E)$ associated to the data $(0, {\xi}, {\xi}^{-1}\omega_E^{-1})$. 
Define  $\CO_\sigma(s,f')$ by a trucantion process as in \eqref{co*}. If $f'=f_S'f'^S$, then  $\CO_\sigma(s,f')$  equals the term $\Phi_ {\xi}(s) \widehat{f'^S}(0, {\xi}, {\xi}^{-1}\omega_E^{-1})$ in \eqref{COEis}. 
   
 \part{Local theory}

 \section{Notations and measures}\label{notations and measures }
  \subsection{Local  setting}\label{localnotations and measures }

Let $F$ be a non-archimedean    local field   of   residue characteristic $p>0$.
Let $\varpi $  be   a uniformizer,
and  $k =\CO_F/\varpi$    the residue field with  $q $ elements.
   Let  $ v $ be the discrete valuation on  $F$ such that $v(\varpi)=1$, and   $|\cdot| $   the absolute value on $F$ such that $|\varpi|=q^{-1}$.  Let $E $ be a separable quadratic extension  of $F$, 
 and $\eta$    the associated quadratic character field of $F^\times$. We use subscript to distinguish   notations for $E$ and $F$ when necessary, 
 such as  $|\cdot|_F $ and  $|\cdot| _E$.
 Let $E^1\subset E^\times$  be the subgroup of norm 1. 
 Make the following assumption   from now on.  
  \begin{asmp}\label{asmpep} We fix a  representative $\ep$ for each coset in $F^\times/\Nm(E^\times)$. If $E/F$ is unramified,    choose $\ep=1$ or  $\varpi$.
   If $E/F$ is ramified,
     choose $\ep=1$  or   $\ep\in \CO_{F}^\times$  and  $ \ep\neq 1 (\mod \Nm ( E^\times))$.
 \end{asmp} 

Let $G,H,H_0,\CS$ and $G_\ep $   be defined as in  \ref{matchorb}.  Let 
$K=\GL_2(\CO_E)$ be the standard maximal compact subgroup of $G=\GL_2(E)$, 
and let $K_{H_0}=K\cap H_0$. 
  Note that $G_\ep\cong G_{\ep\Nm (c)}$ for every $c\in E^\times$, and
   $F^\times/ \Nm (E^\times)$ consists of two elements.   
Embed $G_\ep$   into   $G $  via  \eqref{(5.1)}. 
Let    \begin{equation*}K_\ep:=G_\ep\cap \GL_2(\CO_E). \end{equation*}
Then if $E/F$ is unramified, then $K_\ep$ is a maximal compact subgroup of $G_\ep$.
           Let the tori   $T_\ep,Z_\ep$ of $G_\ep$ be as in \ref{matchorb}.
   

 We fix measures as in \cite[2.1, 2.2]{JN}.   Let $\psi$ be a nontrivial additive character of $F$.  Let $ d_Fx$ be the self-dual Haar measure on $F$ w.r.t. $\psi$, then $ L(1,1_F)|x|_F^{-1}d_Fx$    is a Haar  measure on $F^\times$.  If not confusing, we just use $dx$ to indicate these Haar measures on $F$ or $F^\times$.
Let $c(\psi)$ be the conductor of $\psi$.
 Then we have $\Vol (\CO_F)=\Vol (\CO_F^\times)=q^{c(\psi)/2} $
where the volumes are computed w.r.t. the additive measure on $F$ and the multiplicative measure on $F^\times$.
 Let $\tr=\tr_{E/F}$ be the trace map from $E$ to $F$, and let  $\psi_E:=\psi\circ \tr $ which is an additive character on $E $. 
 Define measures on $E$ and $E^\times$ in the same way using   $\psi_E$.
  Endow $E^\times/F^\times$ with the quotient measure.   
  
  Let the measures on    $T_\ep $ and  $Z_\ep $ be induced from the Haar measures on $E^\times$ and $F^\times$.      
  
Define the    measure on $G$  by $$dg=L(1,1_E)\frac{\prod\limits_{i,j=1}^2d_Eg_{i,j}}{|\det g|_E^2}  .$$
 Then   $\Vol(K)=L(2,1_E)^{-1}\Vol(\CO_E)^4. $
The same formula also gives the measure on $\GL_2(F)$. 
 
  Define the    measure on $G_\ep$  by $$dg=L(1,1_F)|\ep|_F\frac{d_Ead_Eb}{|\det g|_F^2}  $$
 if $g=\begin{bmatrix}a&b\ep\\
 \bar b&\bar a\end{bmatrix}$. 
 In particular, if $\ep=1$, this recovers the measure on $\GL_{2}(F)$. 
 
Define the    measure on $H$ as follows. Let $H'\subset G$ be  the image $\GL_{2,F}$ under the natural embedding.
 If $p>2$, let $\xi\in E-F$ be a trace free element. Then \begin{equation}\begin{bmatrix}\xi&0\\ 0&a\end{bmatrix}H\begin{bmatrix}\xi^{-1}&0\\ 0&1\end{bmatrix}=ZH'.\label{HH'}\end{equation}
  Define the measure on $H'Z$ by 
  $$dg=L(1,1_E) \frac{d_Fad_Ez}{|a|_F|z|_E}d_Fxd_Fy  $$ if $$g=\begin{bmatrix}z&0\\ 0&z\end{bmatrix}\begin{bmatrix}a&0\\ 0&1\end{bmatrix}\begin{bmatrix}1&0\\ y&1\end{bmatrix}\begin{bmatrix}1&x\\ 0&1\end{bmatrix}.$$
Define  the measure on $H$ by  \eqref{HH'} and the measure on $HZ'$. 
If $p=2$, then   $ H =ZH'$. The measure on $H$ is defined to be the measure on $HZ'$. 
 
 \begin{lem}\label{Kmeasure}If  $E/F$  and $\psi$ are unramified, then 
 $$ 
 \Vol(K_1)=\Vol(K_{H_0})=L(2,1_F).$$
\end{lem}

The following lemma plays an important role in \cite{JN} and  \cite[(4.3)]{Jac86}. 
   \begin{lem}[{\cite[(15)]{JN}}]\label{JC15} Let $f$ be an integrable  function on $G_\ep$, then 
 $$\int_{G_\ep}f(g)d  g =\frac{1}{L(1,\eta)^2}\int_{ {x=\ep a\bar a\in   F^\times} } \int_{T_\ep/Z_\ep}\int_{T_\ep}f\left(h_1^{-1} \begin{bmatrix}1&a\ep\\ \bar a &1\end{bmatrix}  h_2\right)   dh_1dh_2 \frac{d_F x}{|1-x|_F^2} .$$ 
  \end{lem}

  \subsection{Global setting}\label{Global measures} 
 
      For $v\in |X|$, let   $\varpi_v$ be a uniformizer of $\CO_{F_v}$, $k(v)$   the residue field with cardinality $q_v $. 
The absolute value on $\BA_F$ is the product of the local ones.  Similar notations apply to finite extensions of $F$.

  For $v\in |X|$, $\BB_v^\times$    is  denoted by $G_\ep$    (for  the corresponding $\ep$) in  the local setting \ref{Split case} $\ep$ is an  abstract subscript and  \ref{localnotations and measures } ($\ep$ is as in Assumption \ref{asmpep}).   
  Fix an isomorphism $ D^\times(\BA_{\mathrm{f}})\cong \BB_{\mathrm{f}}^\times  $ such that the image of $\cD_v^\times$  is $K_\ep\subset G_\ep$ for every $v\in |X|-\{\infty\}$.

   Fix a nontrivial additive character $\psi$ on $F\bsl \BA_F$, and a decomposition $\psi=\prod_ {v\in |X|}\psi_v$. 
Choose local measures on $F_v^\times,$ $E_v^\times,$     $\BB_v^\times$, $G(E_v)$,  and $\GL_2( F_v)$ as in \ref{localnotations and measures }.  
 Take the product measures on  adelic objects. 

  \section{Local distributions}\label{local relative trace formula}
Let $F$ be a local field,
 $E/F$   a separable  quadratic field extension.
Let $\Omega$ be a unitary character of 
$E^\times$, $\omega$  the  restriction of $\Omega$  to 
$F^\times$ and $\omega_E:=\omega\circ\Nm.$ 
 We define distributions on $G$ and $G_\ep$ associated to their representations. (These  distributions  are called spherical characters in \cite{Zha14b}).
The values of these distributions  at certain matching functions will satisfy a ``spherical character identity" (following the terminology of \cite{Zha14b}).
The split case will be dealt in \ref{secSplit places}.   
\subsection{A distribution on $G$ }\label{the    local distribution I(f)}

Let $\sigma $ be an infinite dimensional  irreducible unitary     representation of $G=\GL_{2,E}$ with   central character $\omega_E^{-1}$. 
Let $W(\sigma,\psi_E)$ be the $\psi_E$-Whittaker model of $\sigma$.
Define a $G$-invariant inner product   on $W(\sigma,\psi_E)$ by  \begin{equation}\pair{W_1,W_2}:=
\int_{E^\times}W_1\left(\begin{bmatrix}x&0\\ 0&1\end{bmatrix}\right )\overline {W_2\left(\begin{bmatrix}x&0\\ 0&1\end{bmatrix}\right )}d  x.\label{inprod}\end{equation}
For $W\in W(\sigma,\psi_E)$, define the local toric period
 \begin{equation}\lambda (s,W):=\int_{E^\times}W \left(\begin{bmatrix}x&0\\ 0&1\end{bmatrix}\right )|x|_E^s\Omega(x)d  x\label{lambda}\end{equation}  
 and  the local base change period 
 \begin{equation}\CP (W):=\int_{F^\times}W \left(\begin{bmatrix}x&0\\ 0&1\end{bmatrix}\right ) \eta\omega(x)d  x\label{cp}. \end{equation}

 Assume that $\sigma$ is tempered, then  \eqref {lambda}  converges  absolutely for $\Re(s)>-1/2$.    \begin{defn} For $f\in C_c^\infty(G)$ and $\Re (s)>-1/2$, define 
$$ I_\sigma(s, f)= \sum_W \lambda (s,\pi(f)W) \overline {\CP(   W)}$$
where the sum is over an orthonormal basis of $W(\sigma,\psi_E)$.  
 Denote $I_\sigma(0,f)$ by   $I_\sigma(f)$. 
 
\end{defn}

 \subsection{Distributions on $G_\ep$}\label{   local distribution J(f)}

 Let  
  $\pi $ be an irreducible unitary     representation of $G_\ep$  with   central character $\omega^{-1}$. 
Let $\vep(1/2,\pi,\Omega):=\vep(1/2,\pi_E\otimes \Omega)$.
  Define $$\CP_\Omega(\pi):=\Hom_{T_\ep}(\pi\otimes \Omega,\BC).$$
  
\begin{thm} [{Tunnell \cite{Tun}, Saito \cite{Sai}}]\label{TSlocal} The space 
$ \CP_\Omega(\pi )$ is at most one dimensional. 
Moreover,  $\dim\CP_\Omega(\pi )=1$  if and only if the $\vep$-factor satisfies: $$\vep(1/2,\pi,\Omega) =\eta(\ep)\Omega(-1).$$

\end{thm} 

 \begin{rmk}The  proof  in  \cite{Tun} for non-supercuspidal representations holds for all local fields.
The  proof  in  \cite{Sai} for  supercuspidal representations holds for all local fields. 
\end{rmk}

Assume that $\vep(1/2,\pi,\Omega) =\eta(\ep)\Omega(-1)$. 
Let   $\tilde e_\pi\in\tilde  \pi$ (resp. $e_{\pi}\in \pi$) be the unique  up to scalar vector of $\tilde \pi$ (resp. $ \pi$) 
such that 
the linear form $$v\mapsto (v,\tilde e_\pi)\mbox{ (resp. }\tilde v\mapsto ( e_\pi,  \tilde v) \mbox{)}$$ generates $\CP_\Omega(\pi) $ (resp. $\CP_{\Omega^{-1}}(\tilde \pi) $). 
Here  $(\cdot,\cdot)$ is  the natural pairing between $\pi$ and $\tilde \pi$.
  Assume that  $(e_\pi,\tilde e_\pi)=1$.
For  $f\in C_c^\infty(G_\ep)$, define $$J_\pi(f):=\sum_v( \pi(f)v, \tilde  e_\pi) ( e_\pi,\tilde v) $$
where the sum is over an orthonormal   basis $\{v\}$ of $\pi$  and $\{\tilde v\}$  is the dual  basis of $\tilde \pi$.

Let $w_{\pi}$ be the function on $G_\ep$ defined by $$w_\pi(g)=(\pi(g)e_\pi,\tilde e_\pi).$$
Then $w_{\pi}$ is $T_\ep\times T_\ep$-invariant, locally constant with $w_{\pi}(1)=1$, and 
  \begin{equation}J_\pi(f)=\int_{G_\ep}f(g)w_\pi(g)dg.\label{7.5}\end{equation}
  Let $f=1_{U}$ where $U$ is a small enough open compact subgroup of $G_\ep$, then  \begin{equation}J_\pi(f)=J_{\pi\otimes \eta}=\int_{G_\ep}f(g)dg\neq 0.\label{7.4}\end{equation}

    Let  $\alpha_{\pi}\in \CP_\Omega(\pi)\otimes  \CP_{\Omega^{-1}}(\tilde \pi) $ be   
    $$\alpha_\pi(u,\tilde v)=\int_{T_\ep/Z_\ep} (\pi(t) u,\tilde v) \Omega(t)dt ,$$ which is essentially a finite sum.
By Theorem \ref{TSlocal},
$\alpha_{\pi}({\cdot ,\cdot })$  is a multiple of $(\cdot,\tilde e_\pi)( e_\pi,\cdot )$.
Taking the first variable to be $e_\pi$, we  find that   the ratio  is $ \Vol(E^\times/F^\times)$.
  \begin{defn}  \label{api}For  $f\in C_c^\infty(G_\ep)$, we 
 abuse notation and define  $$\alpha_{\pi}(f):=\sum_v\alpha_{\pi}( \pi(f)v  ,\tilde v)$$
where the sum is over an orthonormal    basis $\{v\}$ of $\pi$  and $\{\tilde v\}$  is the dual  basis of $\tilde \pi$.
\end{defn}Then $\alpha_{\pi}(f)= \Vol(E^\times/F^\times)J_{\pi}(f).$
By \eqref{7.4}, the following lemma holds.
 \begin{lem}\label{pipieta} Let $f=1_{U}$ where $U$ is a small enough  open compact subgroup of $G_\ep$.
Then
  $\alpha_\pi(f)=\alpha_{\pi\otimes \eta}(f)\neq 0.$
  \end{lem}

    \subsection{Spherical character identity} \label{the proof is more important for us}
    Assume that $\vep(1/2,\pi,\Omega) =\eta(\ep)\Omega(-1)$.  Let
  $\sigma$ be the base change of  
  $\pi$ to $E$.   The values of the  distributions on $G$ and $G_\ep$ for $\sigma$ and $\pi$ at matching functions are expected to satisfy  a ``spherical character identity" (following the terminology of \cite{Zha14b}). Such an identity is proved in  \cite[Proposition 5]{JN}   conditionally.
   In particular, in  \cite[Proposition 5]{JN},  
  $\sigma$ is a local component of a global representation with non-vanishing central $L$-value. Thus the case we need for Theorem \ref{GZ} is not covered.   
  
    The main result of this subsection is   Corollary \ref{7.9}, which is an explicit example of matching functions satisfying 
 the spherical character identity.  We find this example using a construction from the proof  of  \cite[Proposition 5]{JN}, which we recall now.

 \begin{prop}  \label{localRTF}
Suppose there is a  constant  $c$ such that for  every pair of (not necessarily purely) matching functions
 $f\in C_c^\infty(G)$ and $f_\ep\in C_c^\infty(G_\ep)$ supported on $\{g\in G_\ep:\det g\in \Nm (E^\times)\}$, the following equation holds:
$ I_\sigma(f)=c J_{\pi}(f_\ep).
$Then
     \begin{align*} c=2 \vep(1,\eta,\psi) L(0,\eta)  .   \end{align*} 

 \end{prop}

 \begin{proof}
  
We follow the proof of  \cite[Proposition 5]{JN}.
  Let $f_n \in C_c^\infty(G)$ approximate   the Dirac function $\delta_1$ at $1\in G$, then $\Phi_{f_n}$ approximates  $\delta_w$. Let $\phi\in C_c^\infty(E)$,  $\hat\phi$ be the Fourier transform of $\phi$  w.r.t. $\psi$.
Choose $\phi$ such that $\hat\phi(0)=0$
and $\int_{F^\times}\hat \phi (b)\eta(b)d  b\neq 0.$
Such $\phi$ exists (see \ref{7.3.1}).
Let
\begin{equation}f^n(g):=\int_{E}\phi(x)f_n\left(\begin{bmatrix}1&x\\ 0&1\end{bmatrix} g\right)dx.\label{fn}\end{equation} 
As in   \cite[p 57]{JN}, for $n$ large enough w.r.t. the choice of $\phi$, we have  $$I_\sigma(f^n)=\int_{F^\times}\hat \phi (b)\eta(b)d  b\neq 0.$$

Let $f_\ep^n$ match $f^n$.
To find the constant $c$ between $ I_\sigma(f^n)$ and $J_{\pi}(f_\ep^n)$,
 the key is  to compute  $J_{\pi}(f_\ep ^n)$ using $f^n$.    We have the following result deduced  from Lemma \ref{JC15} and    \eqref {7.5}.
  \begin{lem}[{\cite[5.3]{JN}}]\label{JC5.3} Let $f$ and $f_\ep$ match, then
 $$J_{\pi}(f_\ep)  =\frac{1}{L(1,\eta)^2}\int_{x=\ep a\bar a\in   F^\times }\omega_{\pi}\left(\begin{bmatrix}1&a\ep\\ \bar a &1\end{bmatrix}\right)\CO(x,\Phi_{f})\frac{d _Fx}{|1-x|_F^2}.$$
     \end{lem}

 The computation of $J_{\pi}(f_\ep)$  for $p>2$ is similar to  \cite[5.2]{JN}. 
 To integrate on $E$, one write $E=F+F\xi$ where $\xi$ is  a trace free element. When  $p=2$, the difference is as follows.
Write $E=F+F\xi$ where $\xi$ has trace 1.
 So if we write $t\in E$ as $a+b\xi$,  the roles of $a$ and $b\xi$  in the cases $p>2$ and $p=2$ are exactly opposite.
 \end{proof}
 The proof of Proposition \ref{localRTF} implies the following proposition.
  \begin{prop}  \label{localRTF'}
  Let $f_n \in C_c^\infty(G)$ approximate    $\delta_1$. Let $\phi\in C_c^\infty(E)$ such that $\hat\phi(0)=0$  
and $\int_{F^\times}\hat \phi (b)\eta(b)d  b\neq 0.$
Let $ f^n$ be defined as in \eqref{fn}.
Then for $n$ large enough w.r.t. the choice of $\phi$,   the following equation holds 
      for any $f_\ep\in C_c^\infty(G_\ep)$ matching $f^n$:  \begin{align*} I_\sigma(f^n)=\frac{2 \vep(1,\eta,\psi) L(0,\eta)}{\Vol(E^\times/F^\times)}  \alpha_{\pi}(f_\ep).\end{align*}

 \end{prop}

 \subsubsection{An explicit example}\label{7.3.1}

 For $f_n\in C_c^\infty(G) $,
let $f^{\phi,n}$   denote the function $f^n$ defined in \eqref{fn}, indicating the role of $\phi$.
 We construct    $f_n$ and $\phi$ satisfying the requirements  in Proposition \ref{localRTF'}.
 Then we  compute  $\Phi_{f^{\phi,n}}$  explicitly, which will be useful for     smooth matching in \ref{Explicit computations   for smooth matching}. 
 
   Let $\fp_E$ be the maximal ideal of $\CO_E$. Let
   \begin{equation*} K_{n}:=\left\{\begin{bmatrix}a&b\\ c&d\end{bmatrix} \in  \GL_2( \CO_E):\begin{bmatrix}a&b\\ c&d\end{bmatrix} \equiv w  (\mod \fp_E^n)\right \}. \end{equation*} 
  For $n$ large enough,we have $$-\det( K_{n}) =1+\fp_E^n\subset \Nm (E^\times).$$   By Lemma \ref{goodx},
$K_{n}\cap \CS\subset Gw .$
Endow $Gw=G/H_0$ with the quotient measure. 
 Then $1_{K_{n}\cap \CS}/\Vol(K_{n}\cap \CS)$ approximates $\delta_w$.
   By   \cite[V.2, Theorem 11]{HC}, we have the following lemma.
 \begin{lem} \label{HCused} There exists $f_n\in C_c^\infty(G) $  approximating $\delta_1$ such that $$\Phi_{f_n}=\frac{1_{K_{n}\cap \CS}}{\Vol(K_{n}\cap \CS) }. $$
\end{lem}

 Now we construct the function $\phi$ we want.
  For $l\in \BZ$, $\xi\in E$, let
  $$\phi_{l,\xi}:=\frac{1_{\xi+\fp_E^l}}{\Vol(\fp_E^l)}.$$
Let $c(\psi_{E})$ be the  conductor of $\psi_{E}$.
 If $b\in \fp_E^{-(l+c(\psi_{E}))} $, then $$\widehat {1_{\xi+\fp_E^l}}(b)=\psi_E(\xi b) \int_{\fp_E^l}\psi_E(ab)da=\psi_E(\xi b)\Vol(\fp_E^l);$$
  otherwise $\widehat {1_{\xi+\fp_E^l}}(b)=0.$
Let $\phi :=\phi_{l',\xi'}-\phi_{l,\xi}$
  where $l>l'$,  then 
  $$\widehat {\phi}(b)=\psi_E(\xi'b)-\psi_E({\xi b})$$ on  $\fp_E^{-(l'+c(\psi_{E}))} $.
  In particular,    \begin{equation}\widehat {\phi}(0)=0\label{phi0}\end{equation}
 
Now we compute the integral  $\int_{F^\times}\hat \phi (b)\eta(b)d  b .$
Recall the following fact   about Gauss sums.

  \begin{lem}\label{Gsum} (1)  Let $ \chi $ be a unitary character of $  F^\times$ of conductor $c(\chi)$, and 
 let  $$\tau_n(\chi,\psi):=\int_{\CO_F^\times} \chi(\varpi^n x)\psi(\varpi^nx)d  x.$$ 
  Then $\tau_n(\chi,\psi)\neq 0$ if and only if $n= -(c(\chi)+c(\psi))$. 
  
  (2)  For $a\in \CO_F^\times$, let $\psi_a$ be the character $x\mapsto \psi(ax)$. Then 
  $$\tau_n(\chi,\psi_a)=\chi(a^{-1})\tau_{n }(\chi,\psi).
  $$
  
  (3) If $\chi$ is unramified, then $$\tau_{-c(\psi)}(\chi,\psi)=  \chi(\varpi^{-c(\psi)} )\Vol(\CO_F^\times).$$
  \end{lem}

  Note that the restriction of   the character $b\mapsto \psi_E(\xi b)$ to $F$ is 
 $\psi_{ \tr(\xi)} $.  Let $c=c(\eta)+c(\psi)$.
  We always choose $l,l'$ such that  $$l>l'> c+\max\{v_F( \tr(\xi)), v_F( \tr(\xi')) \}.$$
  Choose  $\xi,\xi' \in E $ with nonzero trace as follows. Note that the trace map from $E$ to $F$ is surjective.
  If  $E/F $ is ramified, then we   choose $\xi$ and $\xi'$ such that $  \tr(\xi) ,   \tr(\xi')\in \CO_F^\times $ and $ \tr(\xi)\neq \tr(\xi)(\mod \Nm(E^\times)).$ By Lemma \ref{Gsum} (2), we have  \begin{equation}|\int_{F^\times} \hat\phi (b)\eta(b)d  b |=|2\tau_{n}(\eta,\psi)|\neq 0 \label{gausssum0}. 
 \end{equation} 
   If   $E/F $ is unramified, then we   choose $\xi,\xi'$ such that $v_F( \tr(\xi))  $ and $ v_F( \tr(\xi')) $ are nonnegative  and have different parities.
   By   Lemma \ref{Gsum} (3), we have  \begin{equation}|\int_{F^\times} \hat\phi (b)\eta(b)d  b| =2\Vol(\CO_F^\times)\neq 0 \label{gausssum}.
  \end{equation}

 Let $n$ be large enough depending on the choices of $l,l'$, $\xi$ and $\xi'$.
Since we have \eqref{phi0},  \eqref{gausssum0} and   \eqref{gausssum},    Proposition \ref{localRTF'}  implies the following corollary.
\begin{cor}\label{7.9}  For any $f_\ep\in C_c^\infty(G_\ep)$ matching $f^{\phi,n}$, we have 
    \begin{align*} I_\sigma(f^{\phi,n})=\frac{2 \vep(1,\eta,\psi) L(0,\eta)}{\Vol(E^\times/F^\times)}  \alpha_{\pi}(f_\ep).\end{align*} 
       
\end{cor}
Now we compute $\Phi_{f^{\phi,n}}$.
We use the relation
\begin{equation}\Phi_{f^{\phi,n}}(s) =\int_{E}\phi(x)\Phi_{f_n}\left(\begin{bmatrix}1&x\\ 0&1\end{bmatrix} s\begin{bmatrix}1&0\\ \bar x&1\end{bmatrix}\right )dx.\label{79}\end{equation} 
  Let \begin{equation} K_{l,\xi,n}:=\left\{\begin{bmatrix}a&b\\ c&d\end{bmatrix} \in  \GL_2( \CO_E):\begin{bmatrix}0&b\\ c&d\end{bmatrix} \equiv \begin{bmatrix}0&1\\ 1&0\end{bmatrix}  (\mod \fp_E^n),\ a\in -\tr(\xi)+\tr(\fp_E^l) \right\}.\label{orbintK1st}\end{equation} 
Let $X_n:=\tr^{-1}(\CO_F\cap \fp_E^n)\cap \fp_E^l$.
   A direct computation using \eqref{79} shows that 
\begin{equation*}\Phi_{f^{\phi_l,n}}(s)=\frac{ \Vol(X_n)}{\Vol(K_{n}\cap \CS)\Vol(\fp_E^l)} 1_{\CS\cap K_{l,\xi,n}} .\end{equation*} 
Thus \begin{equation}\Phi_{f^{\phi,n}}(s)= \frac{ \Vol(X_n)}{\Vol(K_{n}\cap \CS) } 
\left (\frac{1_{\CS\cap K_{l',\xi',n}}}{ \Vol(\fp_E^{l'})} -\frac{1_{\CS\cap K_{l,\xi,n}}}{\Vol(\fp_E^l)}\right).\label{ppn}\end{equation}

 \subsection{Split extension}\label{secSplit places}
  Recall the setting in  \ref{Split case}.  
  Let $G= \GL_2(F)\times \GL_2(F)$
 and  $G_\ep= \GL_2(F). $ Here $\ep$ is just an abstract subscript.
 Let 
 $\pi$ be an infinite dimensional irreducible unitary  representation of $G_{\ep }$. Let $\sigma=\pi\boxtimes \pi$ which is an  irreducible unitary  representation of $G $.
 Let  $\Omega=\Omega_1\boxtimes \Omega_2$ be a unitary character of $ F^\times\times F^\times$. Then  $\omega:=\Omega_1\otimes \Omega_2$ is the restriction of $\Omega$ to the diagonal embedding  of $F^\times$ in $ F^\times\times F^\times$.

 Let  $W(\sigma,\psi\boxtimes \psi)=W(\pi,\psi)\boxtimes W(\pi,\psi)$ be the Whittaker model of $\sigma$, on which we have the inner product \eqref{inprod}. Let $W= W_1\otimes W_2\in W(\sigma,\psi\boxtimes \psi)$. Define  two   periods on $W(\sigma,\psi\boxtimes \psi)$:
 \begin{equation}\lambda _\sigma( s, W):=\prod_{i=1,2}\int_{F^\times}W_i \left(\begin{bmatrix}x&0\\ 0&1\end{bmatrix}\right ) |x|^s\Omega_i(x)d  x\label{lambda1},\end{equation} 
 and
     \begin{equation}\CP (W):=\int_{F^\times}W_1W_2 \left(\begin{bmatrix}x&0\\ 0&1\end{bmatrix}\right ) \eta\omega(x)d  x\label{lambda2}. \end{equation}  
  
   Assume that $\sigma$ is tempered, then the integral \eqref {lambda1}  converges for $\Re(s)>-1/2$.  
 
\begin{defn} For $f\in C_c^\infty(G)$ and $\Re(s)>-1/2$, define 
$$ I_\sigma(s, f)= \sum_W \lambda (s,\pi(f)W) \overline {\CP(   W)}$$
where the sum is over an orthonormal basis of $W(\pi,\psi)\boxtimes W(\pi,\psi)$.  Denote $I_\sigma(0,f)$ by   $I_\sigma(f)$.

  \end{defn}

Now we turn to the $G_\ep $-side.    
For  $W_1,W_2\in W(\pi,\psi)$, we have an absolute convergent integral   \begin{align*}
\alpha_\pi(W_1,W_2):&= \int_{E^\times/F^\times} \pair{\pi(t)W_1,W_2}\Omega(t)dt .\end{align*}
 \begin{defn} \label{def58} For  $f\in C_c^\infty(G_\ep)$, we 
 abuse notation and define   $$ \alpha_\pi ( f_\ep)= \sum_W\alpha_\pi(\pi(f_\ep)W , W)$$
where the sum is over an orthonormal basis of $W(\pi,\psi) $. 

  \end{defn}

 The following spherical character identity  can be directly verified.
 \begin{prop}[{\cite[Proposition 6]{JN}}]\label{localRTF2} Let $ f_1,f_2\in C_c^\infty(\GL_2(F))$ and 
 $f_\ep=f_1\ast \tilde f_2 \in  C_c^\infty(G_\ep)$
(see \eqref{splitplaces}),
then $  I_\sigma(f_1\otimes f_2)=  \alpha_\pi(f_\ep).$

  \end{prop} 
  
  \begin{lem}[{\cite[Theorem A.2]{Zha14}}] \label{nontrisplit} There exists $ f_\ep\in C_c^\infty(G_{\ep,\reg})$ such that   $\alpha_{\pi }(f_\ep)\neq 0.$  
\end{lem}

\section{Smooth matching and fundamental lemma} \label{review}
  We have defined the local orbital integrals $\CO( x,\Phi)$   and $\CO( x,f_\ep)$   for 
   $\Phi\in C^\infty_c(\CS)$ and $ f_\ep\in C^\infty_c( G_\ep)$ respectively  in \ref{local orbital integrals0}.  We compute  explicit examples of  $\CO( x,\Phi)$  to prove special cases of the smooth matching. Using the results in \ref {7.3.1}, we prove the spherical character identity for these functions.  Finally we prove the fundamental lemma for the full Hecke algebra.   

 \subsection{Local orbital integrals on $G_\ep$}

  We have the following characterization of $\CO(x,f)$.
\begin{prop}[{\cite[p. 332, Proposition]{Jac87}}]\label{Proposition 2.4, Proposition 3.3Jac862} 
Let $\phi$ be a function on $\ep \Nm (E^\times)-\{1\}$.  Then $\phi(x)=\CO( x,f)$    for some  $f\in \CC_{ c}^\infty(G _\ep)$ if and only if  the following  conditions hold:
\begin{itemize}
    
\item[(1)]the function $\phi$  is locally constant   on $\ep  E^\times-\{1\}$;

\item[(2)] the function $\phi$ vanishes near 1; 

\item[(3)]there exists a constant $A$,  such that $\phi(x)=A $ for $ x$ near $0$;

\item[(4)] 
 there exists a constant $B$  such that $\phi(x)=\Omega(a)B $ for $x =\ep a\bar a$ and near $\infty$; 

 \end{itemize}

      \end{prop} 
      We remind the reader that \cite[p. 332, Proposition]{Jac87} contains a mistake: it swapped the behaviors of $\CO( x,f)$ for $x$ near 0 and near $\infty$.
   Indeed, it is easy to detect by  taking $f$ to be supported near 1
   and $x$ near 0.
                  
       \subsection{Explicit computations   for smooth matching}\label{Explicit computations   for smooth matching}
 To avoid confusion, let $\Vol^\times$   indicate the volume of an open subset of   $F^\times$ or $E^\times$ w.r.t. the multiplicative measure,  
 let
  $\Vol^+$  indicate the volume    of an open subset of  $F$ or $E$ w.r.t. the additive measure.
Let   $\xi\in E$ with $\tr(\xi)\neq 0$,  $n,l  $  be  integers which are large enough. Let $K_{l,\xi,n}$ be as in \eqref{orbintK1st}. Similarly define
\begin{equation} K'_{l,\xi,n}:=\left\{\begin{bmatrix}a&b\\ c&d\end{bmatrix} \in  \GL_2( \CO_E):\begin{bmatrix}a&b\\ c&0\end{bmatrix} \equiv \begin{bmatrix}0&-1 \\ -1 &0\end{bmatrix} (\mod \fp_E^n),\ d\in -\tr(\xi)+\tr(\fp_E^l)\right \}.\label{orbintK'}\end{equation}
   \begin{eg} \label{fepm1}  Let $l, n$ be large enough  such that $\Omega(1+\fp_E^n)=1$ and $\eta(-\tr(\xi)+\tr(\fp_E^l))=\eta(-\tr(\xi))$.
      Then  for  $x\in F^\times-\{1\}$,  $\CO(x,1_{ K_{l,\xi,n}\cap \CS})=0$ unless $v_E(x)\geq n+v_E(\tr(\xi))$.
      In this case, we have $$\CO(x, 1_{ K_{l,\xi,n}\cap \CS})=\eta(-x\tr(\xi))\Vol^{\times}(1+\fp_E^{n })\Vol^{\times}(-\tr(\xi)+\tr(\fp_E^l)) .$$    
  \end{eg}

  \begin{eg} \label{fepm2}  Let $l, n$ be     large enough such that $\Omega(1+\fp_E^n)=1$ and $\eta(-\tr(\xi)+\tr(\fp_E^l))=\eta(-\tr(\xi))$.
      Then   for  $x\in F^\times-\{1\}$, $\CO( x,1_{ K_{l,\xi,n}'\cap \CS})=0$ unless $v_E(x)\geq n+v_E(\tr(\xi))$.
     In this case, we have $$\CO(x, 1_{ K_{l,\xi,n}'\cap \CS})=\eta(-\tr(\xi)) \Vol^{\times}(1+\fp_E^{n })\Vol^{\times}(-\tr(\xi)+\tr(\fp_E^l))\cdot  \Omega(-1).$$    
 \end{eg}

  We also have some explicit computations on $G_\ep$ following 
    \cite[2.3, 2.4]{Guo}. 
   Recall that $G_\ep$ is embedded in  $G$ via  \eqref{(5.1)} and Assumption \ref{asmpep}.
       For an integer  $m\geq 1$, let
\begin{equation} K_{\ep,m}:=\left\{\begin{bmatrix}a&b\\ c&d\end{bmatrix} \in K_\ep=G_\ep \cap \GL_2( \CO_E):\begin{bmatrix}a&b\\ c&d\end{bmatrix} \equiv 1 (\mod \fp_E^m) \right\}.\label{orbint1K}\end{equation}
 This is a congruence subgroup of the maximal compact subgroup $K_\ep$ in the usual sense.
      \begin{eg} \label{fepm} Let    $x\in \ep \Nm(E^\times)-\{1\}$.
      If $v_E(x)\geq 2 m+v_F(\ep)$, 
then       $$\CO(x, 1_{ K_{\ep,m}})=\Vol^{\times}(1+\fp_E^{m })\Vol (E^\times/F^\times).$$
Otherwise,  $\CO(x,1_{ K_{\ep,m}})=0$.

\end{eg} 

Let     $\pi $ be an irreducible unitary     representation of $G_\ep$  such that $\vep(1/2,\pi,\Omega) =\eta(\ep)\Omega(-1).$
  Let $\sigma$ be the base change of 
 $\pi$ to $E$.
 Let $I_\sigma $  (resp. $ \alpha_{\pi} $)  be the local distribution on $G$ (resp. $G_\ep$) defined in Section \ref{local relative trace formula}.

 \begin{prop}\label{expmat}Let    $m$ be  large enough.  There exists $f\in C_c^\infty(G)$ such that 
\begin{itemize}\item[(1)] $ f$ purely matches $1_{ K_{\ep,m}}$;
\item[(2)]    the following equation holds:       \begin{align*} I_\sigma(f)=\frac{2 \vep(1,\eta,\psi) L(0,\eta)}{\Vol(E^\times/F^\times)}  \alpha_{\pi}(1_{ K_{\ep,m}}).\end{align*} 
\end{itemize}  
 
\end{prop}

 \begin{lem} \label{822}Let    $m$ be  large enough.  There exists $f_1\in C_c^\infty(G)$ such that 
\begin{itemize}\item[(1)]  for  $x\in \ep \Nm(E^\times)-\{1\}$, if $v_E(x)\geq 2 m+v_F(\ep)$, 
then       $$\CO(x, \Phi_{f_1})=\eta(\ep x) \Vol^\times(1+\fp_E^{m })\Vol(E^\times/F^\times),$$
otherwise  $\CO(x,\Phi_{f_1})=0$;
 
\item[(2)]        the following equation holds:   \begin{align*} I_\sigma(f_1)=\frac{2 \vep(1,\eta,\psi) L(0,\eta)}{\Vol(E^\times/F^\times)}  \alpha_{\pi}(1_{ K_{\ep,m}}).\end{align*} 
\end{itemize}  
 
 \end{lem}

\begin{proof}We use the results in \ref{7.3.1}.  
  Choose  $\xi,\xi' \in E^\times$ and positive  integers $l,l'$  as under  Lemma \ref{Gsum}.  In particular, $\eta(\tr(\xi'))/\eta(\tr(\xi))=-1$.
  If $E/F$ is unramified, we further  require that  $v_E(\tr(\xi)=1,v_E(\tr(\xi')=0$. 
Let $m$ be large enough  w.r.t. $l,l'$, $\xi$ and $\xi'$.
Let $n=2m+v_F(\ep)-v_E(\tr(\xi)) $.
Let $x_{l}, x_{l'}\in \BR$  such that  
\begin{equation}\begin{split}&x_{l'}   \eta(-\tr(\xi'))  \Vol^\times(1+\fp_E^{n' })\Vol^\times(-\tr(\xi')+\tr(\fp_E^{l'}))
\\-&x_l \eta(-\tr(\xi))  \Vol^\times(1+\fp_E^{n })\Vol^\times(-\tr(\xi)+\tr(\fp_E^l))
=\eta(\ep) \Vol^\times(1+\fp_E^{m })\Vol^\times(E^\times/F^\times) \label{8242'}\end{split}\end{equation}
and 
\begin{equation}x_{l'} \Vol^+(\fp_E^{l'})=x_l \Vol^+(\fp_E^{l}).\label{8242}\end{equation}
 Indeed, since \begin{align*}&\frac{\eta(\tr(\xi'))\Vol^\times(1+\fp_E^{n' }) \Vol^\times(-\tr(\xi')+\tr(\fp_E^{l'}))}{\eta(\tr(\xi)) \Vol^\times(1+\fp_E^{n })\Vol^\times(-\tr(\xi)+\tr(\fp_E^l))}\\
 \neq & \frac{\Vol^+(\fp_E^{l'})}{ \Vol^+(\fp_E^{l})},\end{align*}
  such $x_{l}, x_{l'}$ exists.
  Let $$\Phi  :=x_{l'}1_{ K_{l',\xi',n}\cap \CS}-x_l 1_{ K_{l,\xi,n}\cap \CS}.$$
 By    \eqref{ppn} and condition \eqref{8242}, there exists $f_1\in C_c^\infty(G)$ such that  $\Phi=\Phi_{f_1}$.
 By our choice of $n$, Example \ref{fepm1} and condition \eqref{8242'}, (1) holds.
 By   Example \ref{fepm}  and (1), $f_1$ and $1_{ K_{\ep,m}}$ match.
 By
 Corollary \ref{7.9}, (2) holds.  \end{proof}
 
 \begin{lem}\label{823} Let    $m$ be    large enough. There exists $f_2\in C_c^\infty(G)$ such that 
\begin{itemize}\item[(1)] for  $x\in F^\times-\{1\}$,       $\CO(x, \Phi_{f_2})= \eta(\ep x) \CO(x, \Phi_{f_1}),$

\item[(2)]          the following equation holds:   \begin{align*} I_\sigma(f_2)=\frac{2 \vep(1,\eta,\psi) L(0,\eta)}{\Vol(E^\times/F^\times)}  \alpha_{\pi}(1_{ K_{\ep,m}}).\end{align*} 
\end{itemize}  
 
\end{lem}
 
\begin{proof}  Let $w'=\begin{bmatrix}0&1 \\- 1 &0\end{bmatrix}$.
 Let $f_1$ be as in Lemma \ref{822}, and $f_2(g):=f_1(w'g)\cdot \eta(\ep)\Omega(-1).$
Then $$\Phi_{f_2}(s)=\Phi_{f_1}(w'sw'^t)\eta(\ep)\Omega(-1).$$
By Example \ref{fepm1} and \ref{fepm2},  (1) holds.
By the   condition  $\vep(1/2,\pi,\Omega) =\eta(\ep)\Omega(-1)$, the local functional equation of $ L(s,\pi,\Omega)$   and Lemma \ref{822} (2), (2) holds.   \end{proof}

\begin{proof}[Proof of Proposition \ref{expmat}] Let $f_1$ and $f_2$ be  as in Lemma \ref{822}, \ref{823}.
Let $f=\frac{1}{2} (f _1+f_2).$
 By 
 Example   \ref{fepm},    Lemma \ref{822} (1) and  \ref{823} (1),  (1) holds.
 By  Lemma \ref{822} (2) and  \ref{823} (2),   (2) holds. \end{proof}

 We have a mild modification when  $1_{ K_{\ep,m}}$ is replaced by   $ 1_{ K_{\ep,m}\varpi^\BZ }$. 
 Let $\Xi=K_{\ep,m}\varpi^\BZ\cap Z_\ep$.  
 Redefine  the Hecke action $\pi(f_\ep)$   as in  \eqref{HHprojpi}, $\alpha_{\pi}(f_\ep)$ be as in Definition \ref{api} w.r.t. the new Hecke action $\pi(f_\ep)$. Proposition \ref{expmat} has the following variant.
 \begin{prop}\label{expmat1}Let    $m$ be   large enough. There exists $f\in C_c^\infty(G)$ such that 
\begin{itemize}\item[(1)] $ f$ purely matches $f_\ep$ (see Definition \ref{matchingdef});
\item[(2)]    the following equation holds:       \begin{align*} I_\sigma(f)=\frac{2 \vep(1,\eta,\psi) L(0,\eta)}{\Vol(E^\times/F^\times)}  \alpha_{\pi}(1_{ K_{\ep,m}\varpi^\BZ }).\end{align*} 
\end{itemize}  
 
\end{prop}

 \subsection{Fundamental lemma}  
 
 Assume   $E/F$, $\psi$  and $\Omega$ (so $\omega$) are 
  unramified  through this subsection.
Recall that $K$ is the standard maximal compact subgroup $\GL_2(\CO_E)$ of $G$, 
  and  $K_{H_0}=H_0\cap K$ (see \ref{localnotations and measures }). 
  Let $v=v_F$.
  Let $\bc$ be the  base change homomorphism from the spherical Hecke algebra of $G$  to  the spherical Hecke algebra of $G_1$ (see \cite[Section 1]{Lan}). 
 
\begin{prop} [Fundamental lemma]\label{FLgeneral} 
  Let  $f\in C_c^\infty(G)$ be  bi-$K$-invariant, then $$\bc\left(\frac{\Vol(K_{H_0}) \Vol(K_1)}{\Vol(K)}f\right)$$ purely matches $\Phi_{f}$. In particular,  $\frac{1_{K_1}}{ \Vol(K_1)}$ purely matches $\Phi_{\frac{1_{K}}{\Vol(K_{H_0})\Vol(K_1)}}.$

    \end{prop}    
  In odd characteristic, the fundamental lemma is proved in  \cite[Section 4]{Jac87}, and another proof  is given in 
  \cite[Proposition 3]{JN} using 
 the reduction method in \cite[Section 3]{JLR}. 
  This  the reduction method works in arbitrary characteristic, and we prove only the following  cases of Proposition \ref {FLgeneral}  incharacteristic $2$.     
     \begin{prop}  \label{FLgeneral1} Let $m\geq 0$ and be even.
   Let $f$ be the characteristic function of the set of matrices $g\in G_1$ with integral entries such that $v(\det g)=m$, 
  $\Phi$  be the characteristic function of the set of matrices $g\in \CS$ with integral entries such that $v(\det g)=m$.
  Then    for $x\in \Nm (E^\times)-\{1\}$,   we have
$$ \CO(x,\Phi)= \CO( x, f ) ;$$  
 for $x\in F^\times-\Nm (E^\times)$,     we have
$$   \CO(x,\Phi)=0.$$  
    \end{prop}

 We compute    $\CO(x,\Phi)$ explicitly. Let   $\xi:=\Omega^{-1}(\varpi)$.   
\begin{eg}\label{m=0} Let $\Phi=1_{K\cap \CS} $.  

Suppose $v(x)>0$. If $v(x)$ is odd, then  $\CO(x,\Phi)=0.$
If $v(x)$ is even, then $\CO(x,\Phi)=1.$

Suppose $v(x)<0$. If $v(x)$ is odd, then  $\CO(x,\Phi)=0.$
If $v(x)$ is even,  then $\CO(x,\Phi)=\xi^{-v(x)/2}.$

Suppose $v(x)=0$. If  $v(1-x)> 0$, then  $\CO(x,\Phi)=0.$  If $v(1-x)=0$, then  $\CO(x,\Phi)=1 .$
  
 \end{eg}
  
 \begin{eg}\label{935}
 Let $m>0
 $ and be even. 
 Let $\Phi$ be the characteristic function of the set of matrices $g\in \CS$ with integral entries such that $v (\det g)=m$. 
 
 Suppose $v(x)> 0$.
If $v(x)$ is  odd, then 
$\CO(x,\Phi) = 0.$  If $v(x)$ is  even, then $\CO(x,f)= \xi^{m/2}.
$
  
  Suppose $v(x)<0$.
If $v(x)$ is odd, then 
$\CO(x,\Phi) = 0.$  If $v(x)$ is even, then $\CO(x,\Phi) 
  = \xi^{(m-v(x))/2}.
$
 
Suppose $v(x)= 0$,   then 
$\CO(x,\Phi) =  \xi^{(m-v(1-x))/2}.$  
   
 \end{eg}

Now we compute the orbital integrals for the $G_1$-side  in characteristic 2.  We   follow \cite[5.4, 5.5]{Jac86}.   Let $p=2$. 
 Let $E=F[\xi]$ be the Artin-Schreier unramified  quadratic field extension  of $F$, where  $\xi\in E$ satisfies $\xi^2+\xi+\tau=0$  for some $\tau\in k^\times$. Let  $E^\times\incl G_1\cong \GL_2(F_v)$ be given by  \begin{equation}\label{echar2} a+b\xi\mapsto \begin{bmatrix}a&b \\ b\tau&a+b\end{bmatrix} \end{equation} where $a,b\in F$.
Let   $j= \begin{bmatrix}1&0 \\ 1 &1\end{bmatrix}.$ 
The $j$ satisfies the conditions in \ref{matchorb}.
 
 The embedding $E^\times\incl G_1\cong \GL_2(F) $ is as   \eqref{echar2}.
Let  $x=a^2+ab+b^2\tau =\Nm (a+b\xi)$.  Take \begin{equation}\delta(x)=1+ (a+b\xi)j= 1_2+ \begin{bmatrix}a&b \\ b\tau&a+b\end{bmatrix}j=\begin{bmatrix}1+a+b&b \\ a+b+b\tau&1+a+b\end{bmatrix}\label{dx2}.\end{equation}
Then   $\det \delta(x)=1+x.$        \begin{lem}\label{v(x)} We have $ v(x)=\min\{v(a^2),v(b^2)\}.$
      In particular $v(x)$ is even.
      \end{lem}
      \begin{proof} If $v(a)\neq v(b)$, $v(ab)$ is between $v(a^2)$ and $v(b^2)$. Then
       $ v(x)=\min\{v(a^2),v(b^2)\}$.         If $v(a)=v(b)$, we may assume $b=1$. Then   $v(a^2+ab+b^2\tau)=0$.
      \end{proof}

  Let
$$  C_m:=K_1\begin{bmatrix}\varpi^m&0 \\ 0&1\end{bmatrix}K_1.$$ 
\begin{prop}\label{5.4,5.5'} 

(1) Suppose $m>0$, then $\CO (x,1_{C_m})=0$ unless $v(x)=0$ and $v(1+x)=m$, in which case  $\CO (x,1_{C_m})=1  $.

(2) Suppose $m=0$.

 If $v(x)>0$,  then $\CO(x,1_{C_m})= 1.$
 
If  $v(x)<0$,  then $\CO(x,1_{C_m})=\xi ^{-v(x)/2}.$

 If $v(x)=0$, and $v(1-x)=0$, then  $\CO(x,1_{C_m})=1 .$
 
  If $v(x)=0$, and $v(1-x)> 0$, then  $\CO(x,1_{C_m})=0.$

\end{prop}
\begin{proof}
   
  Now we begin to compute $\CO (x,1_{C_m})$.
  Let $\xi=\Omega^{-1}(\varpi)$. 
     Since $T_1 \subset K_1Z_1$, we have
     $$  \CO(x,1_{C_m} )= \int_{Z}1_{C_m}(z\delta(x))dz= \sum_{k\in\BZ} 1_{C_m}(\varpi^k \delta(x))\Omega^{-1}(\varpi^k).$$
       The conditions \cite[(5.4.7),(5.4.8),(5.4.9)]{Jac86} on $1_{C_m}(\varpi^k \delta(x))\neq0$ (so $=1$), i.e. $\varpi^k \delta(x)\in K_1$  are exactly the same.
 In particular,   \cite[(5.4.7)]{Jac86} says that if  $  \CO_E(\delta(x),1_{C_m} )\neq 0$, then $ n:=m-v(1+x)$ is even. In this case,  only the term for $k=n/2$ possibly contributes to the sum.

 The proof of (1). The cases $v(x)<0,v(x)>0$ are exactly the same as in \cite[(5.4)]{Jac86}. Let us prove the case $v(x)=0$. Suppose $\CO (\delta(x),1_{C_m})\neq0$.
 The entries of $\varpi^{n/2} \delta(x)$, where $ n:=m-v(1+x)$ is even, are
$$ \varpi^{n/2}(1+a+b),\ \ \  \varpi^{n/2}b,\ \ \  \varpi^{n/2}(a+b+b\tau),\ \ \  \varpi^{n/2}(1+a+b).$$
Since  $v(x)= 0$. By Lemma \ref{v(x)}, at least  one of $v(a),v(b)$ is 0, and the other is non-negative.
So at least one of $$    b,\ \ \   a+b+b\tau   $$
 is a unit in $\CO_F$.
  By \cite[(5.4.8),(5.4.9)]{Jac86},
$n=0$.
Thus, $\varpi^{n/2} \delta(x)\in K_1$ and 
 $\CO_E(\delta,1_{C_m})=1$.

 The proof of (2). We only compute the case $v(x)=0$, $v(1+x)>0$. 
Suppose $\CO_E(\delta(x),1_{C_m})\neq0$.
The entries of $\varpi^{n/2} \delta(x)$, where $ n:= -v(1+x)<0$  and even, are
$$ \varpi^{n/2}(1+a+b),\ \ \  \varpi^{n/2}b,\ \ \  \varpi^{n/2}(a+b+b\tau),\ \ \  \varpi^{n/2}(1+a+b).$$
Since  $v(x)= 0$, by Lemma \ref{v(x)}, at least  one of $v(a),v(b)$ is 0.
So at least  one of $$\varpi^{n/2}b,\ \ \  \varpi^{n/2}(a+b+b\tau)$$ is not contained in $\CO_F$. By \cite[ (5.4.8) ] {Jac86},
 $\CO_E(\delta(x),1_{C_m})=0$, a contradiction.
       \end{proof}

 
   Comparing Example \ref{m=0} and  Example \ref{935} with Proposition \ref{5.4,5.5'}, Proposition \ref{FLgeneral1}  for $p=2$  follows.
   
 \section{Global and local periods,  proof of  Theorem \ref{The Waldspurger formula over function fieldsintro} }
 \label{Global and local periods}
  We come back to the global situation.    We decompose the global automorphic periods into products of local distributions  and $L$-functions. Then we prove the Waldspurger formula.
\subsection{Cuspidal representations}
 Let $\sigma\in\CA_c(G,\omega_E^{-1})$ (see  \ref{Spectral decomposition of the automorphic distributions}) and $\Omega$  a Hecke character of $E^\times$. 
We have the global  distribution
$ \CO_\sigma (s, \cdot) $ (see \eqref{co})
and   the local distribution $I_{\sigma_v} ( s, \cdot )$ (see \ref{the    local distribution I(f)}).
 Let $S\subset |X|$ be a finite set  containing    all  ramified places  of $B$,  $E/F$, $\psi $, $\pi$, and all 
places below ramified places of $\Omega$.  Let  $K_{H_0}^S$ (resp. $K^S$) be the product of local maximal compact groups  of $H_0$ (resp. $G$) defined in \ref{localnotations and measures }   outside $S$ (resp. places of $E$ over $S$).
\begin{prop}\label{CO1} Assume that   $\sigma $ is the base change of a cuspidal representation of $ \GL_{2,F}$.
Let  $f'^S=\frac{1_{K^S}}{\Vol(K_{H_0}^S)\Vol(K'^S)}$, then
  \begin{equation*} \CO_\sigma (s, f'_Sf'^S) =\left(\prod_{v\in S}I_{\sigma_v} ( s, f'_v)\right) L_S(1,\eta)   \frac{L^S(1/2+s,\pi,\Omega)L^S(2,1_F)}{L^S(1,\pi,\ad)}. \end{equation*}
\end{prop}
This proposition  is the analog of  \cite[Proposition 4]{JN} for   function fields.
We follow  the proof in loc. cit..
 We have 
two 
global periods $\lambda (s,\phi)$ and $\CP(   \phi)$ for     $\phi\in \sigma$ (see  \ref{Spectral decomposition of the automorphic distributions}).   Let 
 $W $  be the $\psi_E$-th coefficient of $\phi$.
Fix a decomposition  $W=\prod_{v\in |X|} W_v$  where  $W_v$ is the Whittaker newform of $\sigma_v$ for $v\not\in S$ (in the sense of \cite[2.3]{Zha01}).
  Then we have a decomposition 
\begin{equation}\lambda (s,\phi)=\prod_{v\in S}\lambda_v (s,W_v) L^S(s+1/2,\sigma_E\otimes\Omega)\label{decomlambda}\end{equation}
for $\Re(s)>0$,
  where $\lambda_v$ is the local period   defined in \eqref{lambda} and \eqref{lambda1}.
  
  Now we decompose $\CP$. Let $H'=\GL_{2,F} $ which is naturally a subgroup of $G$. Let $ B'\subset H'$ be the  subgroup of  upper triangular  matrix, 
  $N'
  $   be the  upper unipotent subgroup, and $Z'\subset H'$ be the center.
 If $p>2$, let $\xi\in E-F$ be a trace free element. Then $$\begin{bmatrix}\xi&0\\ 0&1\end{bmatrix} H\begin{bmatrix}\xi^{-1}&0\\ 0&1\end{bmatrix} =ZH'.$$
 The pullback of the similitude character $\kappa   $ on  $H$ to  $H'=\GL_{2,F}$ is the determinant. 
  Thus $$\CP(   \phi)=    \int_{H'(F)\bsl H'(\BA_F)} \phi\left(h\begin{bmatrix}\xi&0\\ 0&1\end{bmatrix}\right ) \eta\omega (\det(h) ) dh   .$$ 
    If $p=2$, let   $ H =ZH'$, and  $$\CP(   \phi)=    \int_{H'(F)\bsl H'(\BA_F)} \phi(h ) \eta\omega (\det(h) ) dh   .$$

  Let $\Phi=\otimes'\Phi_v\in C_c^\infty(\BA_F^2)$ be a pure tensor. Suppose that $\Phi^S$ is the characteristic function of $(\hat\CO_F^S)^2$.
  Let $e_2=(0,1)\in \BA_F^2 $,   and $H'(\BA_F)$ act on $\BA_F^2$ from right naturally. 
  Let  $E(g,\Phi,s)$  be the Eisenstein series associated to $\Phi$ (see \cite[Section 4]{JN}).
    Consider the integral 
 $$\Psi(s,  \phi,\Phi):=\int_{Z'(\BA_F)H'(F)\bsl H' (\BA_F)}E(g,\Phi,s)\phi(g)\eta\omega (\det(g) ) |\det g|_E^sdg.
$$

\begin{lem}[{\cite[p. 50]{JN}}]\label{tangent}If $p>2$ ,let $\xi\in E-F$ be a trace free element. Then  
 $$\Psi(s,\phi,\Phi)=\int_{N' (\BA_F)\bsl H' (\BA_F)}W\left(\begin{bmatrix}\xi&0\\ 0&1\end{bmatrix} g\right)\Phi(e_2g)\eta\omega (\det(g) )|\det g|_E^sdg.$$
 
If $p=2$, then
 $$\Psi(s,\phi,\Phi)=\int_{N' (\BA_F)\bsl H' (\BA_F)}W( g)\Phi(e_2g)\eta\omega (\det(g) )|\det g|_E^sdg.$$
\end{lem}

The expressions of $\Psi(s,\phi,\Phi)$ in the above lemma have obvious local-global decompositions.
Define $\Psi^S(s,W,\Phi) $ to be the  away from $S$  part.
Computing the residue of $\Psi$ at $s=1$, we have a decomposition (see \cite[(23)]{JN})
\begin{align}\CP(\phi)=\prod_{v\in S}\CP_v(W_v) \frac{2\Res_{s=1}\Psi^S(s,W,\Phi)
}{\Res_{s=1} L(s,1_F)}\label{101}\end{align} where $\CP_v$ is the local  period  associated to $\sigma_v$ defined in \eqref{cp} and \eqref{lambda2}.
We remark that the proof of   \cite[(23)]{JN}  uses a certain invariance property of the local  period $\CP_v$  which can be established for all local fields. The same proof  of   \cite[(23)]{JN} works for \eqref{101}.
If $p>2$, let $\xi$ be as in Lemma \ref{tangent}, and enlarge $S$ such that  $\xi$ is a unit in $\CO_{E_v}$ for $v\not\in S$. 
Applying the   formula in \cite{Sch} for the unramified Whittaker newforms, we have the following lemma (for all $p$). 
\begin{lem} \label{cp'}  
Assume that $\sigma $ is the base change of  $\pi\in \CA_c(H')$. Then
$$  \frac{ \Psi^S(s,W,\Phi)
}{L^S(s,\sigma\times \tilde \sigma)}=\frac{\Vol(K'^S)}{L^S(s,\eta)L^S(s,\pi,\ad)}.$$ 
  Here 
$K'^S$ is the standard maximal compact subgroup of $H'(\BA_F^S)$,  and   $L^S(s,\sigma\times \tilde \sigma)$ is the product of local $L$-factors outside the places of $E$ over $S$.
\end{lem}

Recall that the Petersson pairing on $\sigma$ satisfies the following local-global decomposition \cite[Proposition 2.1]{CST}: let $\phi_i\in \sigma$, where $i=1,2$, and $W_i $ be its $\psi_E$-th Whittaker coefficient, then
\begin{equation}\int_{Z(\BA_E)G(E)\bsl G(\BA_E)}\phi_1\bar \phi_2 (g)dg=2\prod_{v }\frac{1}{L(1,1_{E_v})}\pair{W_1,W_2}_v  ,\label{peterdecom}\end{equation}
where the product is over all places of $E$. 
 By the explicit  formula of $W_v$ for $v\not\in S$ (see \cite{Sch}), we have   \begin{equation}\frac{1}{L(1,1_{E_v})}\pair{W_v,W_v}_v=    \frac{L(1,\sigma_v,\ad) }{L(2,1_{E_v})}.\label{decompet}\end{equation} 
By \eqref{COSigma}, Lemma \ref{Kmeasure}, \eqref{decomlambda}, \eqref{101}, Lemma \ref{cp'}, \eqref{peterdecom} and \eqref{decompet},   Proposition \ref{CO1} follows.

\subsection{Non-cuspidal representations}\label{Local-global decomposition of  periods 2}

 Let $\sigma=\sigma_{\xi}$ be  the admissible representation of $G(\BA_E)$  associated to the data $( 0, {\xi}, {\xi}^{-1}\omega_E^{-1})$ (see  \ref{Spectral decomposition of the automorphic distributions}).  
The non-cuspidal version of Proposition \ref{CO1}  is Proposition \ref
{CO1d} below, which is the analog of  \cite[Appendiex, Corollary 1]{MW} for function fields.     The proof is the same with \cite[Appendiex, Lemma 2]{MW} with mild modifications when $p=2$.
   \begin{prop} \label{CO1d} Let   $f'_Sf'^S$ be as in Proposition \ref{CO1}. Then
  \begin{equation*}4 \CO_\sigma ( s, f'_Sf'^S) =\left(\prod_{v\in S}I_{\sigma_v} (  s,f'_v)\right) L_S(1,\eta)   \frac{L^S(1/2+s,\pi,\Omega)L^S(2,1_F)}{L^S(1,\pi,\ad)}.  \end{equation*}
\end{prop}

\subsection{The Waldspurger formula over function fields}\label{The Waldspurger formula over function field}
    In this subsection, let notations be as in \ref{The global relative trace formulaintro}. We prove   Theorem \ref{The Waldspurger formula over function fieldsintrof}.
We choose $ f_v$'s  in Theorem \ref{The Waldspurger formula over function fieldsintrof} as follows. 
    Let  $S\subset |X| $ be a finite set. Suppose $S$ contains  all  ramified places  of $B$,  $E/F$, $\psi $, $\pi$,  all 
places below ramified places of $\Omega$, and $S \bigcap |X|_s$ is nonempty.
Here  $|X|_s\subset |X|$ is the subset of places split in $E$.
  Let 
the functions $f^S\in C_c^\infty( B^\times(\BA_F^S))$ and $f'^S\in  C_c^\infty (G(\BA_E^S))$  be as follows:
 \begin{itemize}
 \item[(1)] for   $v\in |X|-S-|X|_s $,   let $ f'_v $ be spherical,  
 and $f_v$  be  the matching function of $f_v'$ on $B_v ^\times$ given by   Proposition \ref{FLgeneral}; 
   \item[(2)] for   $v\in |X|_s-S $, let $ f'_v=(f_{1,v},f_{2,v}) $ be spherical,
     and $f_v$ be the matching function of $f_v'$ on $B_v ^\times$ as in \eqref{splitplaces}.
   \end{itemize}
    For $v\in |X|_s\cap S$, choose $ f_v$  as in  Lemma \ref{nontrisplit}. Then  $\alpha^\sharp_{\pi_v}(f_v)\neq 0$ and $f$ satisfies Assumption  \ref{freg'}.
  Choose $f'_v $  as in  Lemma \ref{afneq00} (1).    Then $f'_v$ purely matches $f_v$.
  By Lemma \ref{afneq00} (2),
 $f'$ satisfies Assumption  \ref{freg}. 
For $v\in S-|X|_s$,  let $f_v=1_{K_{\ep,m}}$ and  $f_v'$    its pure matching function  in  Proposition \ref{expmat}.  
 Let $m$ be large enough, such that $\det (K_{\ep,m})\subset \Nm(E_v^\times)$
and 
   $\alpha^\sharp_{\pi_v}(f_v)\neq 0$
 (see Lemma \ref{pipieta}).
   By  Lemma  \ref{decquat} and \ref{decpure'}, we have the   relative trace formula identity:
 \begin{align}\CO(f)=\CO(f')\label{RTF}\end{align}

  Suppose that  the Jacquet-Langlands correspondence of  $\pi$  to $\GL_{2,F}$ is not of the form $\pi_\xi$   as in the end of \ref{Local-global decomposition of  periods 2}. Then $\pi\not \cong  \pi\otimes\eta$,    
  and  $\sigma:=\pi_E$  is cuspidal. By   \eqref{COSigma}, \eqref{COSigma'},   Theorem \ref{Eisterm} (for $s=0$),  and similar results for $B^\times$ (see \cite[(7.6)]{Jac87} and \cite[Section 9]{JN} in the number field setting, and the same proof applies here),
we have the following spectral decomposition of  \eqref{RTF}:
  \begin{align*}\CO_\pi(f) +\CO_{\pi\otimes\eta}(f)= \CO_\sigma(f') . \end{align*}  
Now choose $$f'^S=\frac{1_{K^S}}{\Vol(K_{H_0}^S)\Vol(K'^S)},\ f^S= \frac{1_{K_1^S}}{ \Vol(K'^S)}= \frac{1_{K_1^S}}{ \Vol(K_1^S)} .$$
 Then $f$ is supported on $\{g\in B^\times(\BA_F):\det (g)\in \Nm (\BA_E^\times )\}$. In particular, we have $\CO_\pi(f) =\CO_{\pi\otimes\eta}(f).$ 
  Thus  $2 \CO_\pi(f) = \CO_\sigma(f').$
    From Proposition   \ref{CO1}, we have
  \begin{align*}2 \CO_\pi(f) = \CO_\sigma(f')=\left(\prod_{v\in S}L(1,\eta_v)I_{\sigma_v} (  f'_v)\right)     \frac{L^S(1/2,\pi,\Omega)L^S(2,1_F)}{L^S(1,\pi,\ad)}. \end{align*}  
    By   Proposition  \ref{localRTF2}, \ref{expmat},    the computation of $\ep$ (or $\gamma$)-factor 
 \cite[2.5]{Tat}, and  that the product of the local root numbers of $\eta$  is 1, we have 
  \begin{align*}2 \CO_\pi(f)&= \CO_\sigma(f')=     \frac{L (1/2,\pi,\Omega)L (2,1_F)}{L (1,\pi,\ad)}  \prod_{v\in |X| } \alpha_{\pi_v}^\sharp (f_v) . \end{align*}

Suppose that  the Jacquet-Langlands correspondence of $\pi$  to $\GL_{2,F}$ is the  representation $\pi_\xi$ as in the end of \ref{Local-global decomposition of  periods 2}. Then  $\pi\cong \pi\otimes \eta $.
 Let   $\sigma=\sigma_\xi $ be the  base change  of $\pi_\xi$ as in the last subsection.
We have the following spectral decomposition of the relative trace formula
   \begin{align*} \CO_\pi(f)= \CO_{\sigma_\xi}(f')+ \CO_{\sigma_{{\xi}^{-1}\omega_E^{-1}}}(f')= 2  \CO_\sigma(f')   . \end{align*} 
 From Proposition \ref{localRTF2}, \ref{expmat}   and \ref{CO1d}, we have
  \begin{align*} 2\CO_\pi(f) = 4\CO_\sigma(f')&=\left(\prod_{v\in S}I_{\sigma_v} (  f'_v)\right) L_S(1,\eta)   \frac{L^S(1/2,\pi,\Omega)L^S(2,1_F)}{L^S(1,\pi,\ad)}\\
  &=\frac{L (1/2,\pi,\Omega)L (2,1_F)}{L (1,\pi,\ad)}  \prod_{v\in |X| } \alpha_{\pi_v}^\sharp (f_v) .  \end{align*}

\section{Decomposition of the height distribution}\label{Local intersection multiplicity}
Let     $I$ be a nonempty  finite closed subscheme of $X-\{\infty\}$,  $U=U(I)\subset \cD\otimes\hat \CO_F $   the corresponding principal congruence subgroup. Identify $U$ as a subgroup of $\BB_{\mathrm{f}}^\times$ via the isomorphism $ D^\times(\BA_{\mathrm{f}})\cong \BB_{\mathrm{f}}^\times  $ fixed in \ref{Global measures}.  Let  $\CH_{U,\BC}$ the Hecke algebra of $\BB^\times$ (see \ref {measures}).
For $f \in \CH_{U,\BC}$, we compute  the height distribution $H(f)$ (see Definition \ref{CM height}) under  the following assumptions.    

     \begin{asmp}\label{asmpe0}
   
                
  For every $v\in |X|-|X|_s$, the inclusion $ E _v^\times  \incl \BB_v^\times=G_\ep $ (for the corresponding $\ep$   in Assumption \ref{asmpep}) induced by $ e_0:E  \incl \BB $ (see \eqref{e0}) is $T_\ep\incl G_\ep$.

   \end{asmp} 
   
  By   the discussion in \ref{CMuni0} and invariance of the \Neron-Tate height pairing   $\pair{\cdot,\cdot}_\NT$  on $J$ under the diagonal $\BB^\times$-action on the two variables, the truth of Theorem \ref{GZ} for one embedding $ e_0:E  \incl \BB $   is equivalent to the truth of Theorem \ref{GZ} for every such embedding by the Noether-Skolem theorem.
  Thus we can choose $e_0$ such that Assumption \ref{asmpe0}   holds.    
   
    Let $ \BB_v^n:=\{g\in  \BB_v^\times: v(\det (g))=n\}$. 
    \begin{asmp} \label{fvan} Assume   that $f$ is a pure tensor with $f_\infty=1_{U_\infty}$ and there exists  two disjoint finite subsets  $S_{s,\reg} $ and $S_{s,\ave}$   of $|X|_s$, both of cardinality $\geq 2$, such that
         \item[(1)]       for $v\in S_{s,\reg}$, $f_v$ is supported on   the regular locus of    $ \BB_v^\times$ for the  $E _v^\times$-action;         \item[(2)]    for $v\in S_{s,\ave}$, $U_{v}$ is maximal and for every $n\in \BZ$, the following equation holds:
$$  \sum_{g_{v}\in  \BB_{v}^n /U_{v}} f_{v}(g_{v})   =0.$$         
\end{asmp}     
 The main results of this section are summarized in the following theorem.
\begin{thm}\label{summary}Assume  Assumption \ref{asmpe0}  and Assumption  \ref{fvan}. Then  we have \begin{equation*}H(f)= \sum_{v\in |X|}   -( i(f)_v+j(f)_v))\log q_{v},   \end{equation*}
  where $i(f)_v$ and $j(f)_v$ are given as follows:
  
    \begin{itemize}
  \item[(1)]    For   $v\in |X|-|X|_s-\{\infty\}$, we have $$i(f)_v=\sum_{\delta\in E^\times\bsl B(v)^\times_\reg/E^\times} \CO _{\Xi_\infty}(\delta,f^v) i (\delta,f_v),$$ 
where $ i (\delta,f_v)    $  is given by an  orbital  integrals of an intersection  multiplicity function $m_v $ on $B(v)_v^\times\times \BB_v^\times $ (see Definition \ref{mf1}, \ref{mf2}) weighted by $f_v$ as in Definition \ref{iterm1}. For nonvanishing conditions on $m_v$, see Lemma \ref{vdetun1}, \ref{vdetun"}, \ref{vdetun'}. For computations  on $m_v(g_1,1)$, see Lemma \ref{U'}, \ref{U''}. 
  
  \item[(2)]    For   $v\in |X|-|X|_s-\{\infty\}$,     there exists $\overline {f_v}\in C_c^\infty(B_v^\times)$ such that 
  $$j(f)_v= \sum_{\delta\in E^\times\bsl B_\reg^\times/E^\times}\CO _{\Xi_\infty}(\delta,f^v)   
\CO (\delta ,\overline {f_v})    .$$  
  
      \item[(3)] For $\infty$, similar properties for  $ i(f)_\infty$ and $j(f)_\infty$ hold (see \ref{ijinfty}).
        \item[(4)] For $v$ split in $E$, $ i(f)_v=j(f)_v=0$.  

    \end{itemize}    
      
 \end{thm}

To prove Theorem \ref{summary}, 
  we first define   integral models of the modular curve $M_U$, and use intersection theory on the  integral models to decompose $H(f)$ over places of $F$.
The intersection number of horizontal divisors is  the $i$-part  in Theorem \ref{summary},
and the rest  the $j$-part.
   Then we decompose  the $i$-part 
and   the $j$-part  at different places case by case.

   \subsection{CM points and models}\label{Points and models}
\subsubsection{CM points}\label{CM points}
Let $P_0 \in  M^{E^\times}(F^\sep )$.
 For $h\in \BB ^\times$, the image of $T_{h}P_0$  under the rigid analytic uniformization of $M$  at $\infty$ is $[z_0h_\infty,h_{\mathrm{f}}]$.
 Let $(T_{h}P_0)_U $ be the image of  $T_{h}P_0$ in $M_U$.
  The map $   h \mapsto (T_{h}P_0)_U  $
 induces a bijection \begin{equation} \label{111}  E^\times\bsl \BB^\times/  \tilde U\cong CM_U.  \end{equation}
 Thus we identify $CM_U$ with $  E^\times\bsl \BB^\times/  \tilde U$, and denote $ (T_{h}P_0)_U  $  by $h$.
Let $\tilde U_ E=\tilde U\cap \BA_E^\times$.
Regard $E^\times\bsl  \BA_E^\times /\tilde U_E$ as a subset of $CM_U$.
Let $H$ be a  finite abelian extension of $E$  such  that all geometrically connected components of $M_U$    
     and all points in $E^\times\bsl  \BA_E^\times /\tilde U_E$  are defined over $H$.  
     
For $v\in |X|$, let $\overline v$    be   an extension  of $v$ to $  F^\sep$, and let $w$ be the restriction of $\overline v$ to $H$,
Let $\CO_{\overline v}$ be the ring of integers of the completion of $F^\sep $ at $\overline v$. 
 Let   $ \hat F_v^\ur $ be the completion of the  maximal unramified extension of $ {F_v}$ w.r.t. the restriction of $\overline v$,
  $ \hat \CO_{F_v}^\ur $ be its ring of integer.
   Define $ \hat H_w^\ur $ and $ \hat \CO_{H_w}^\ur $ simiarly.
   In particular, we have   embeddings $\hat \CO_{F_v}^\ur \incl  \hat \CO_{H_w}^\ur\incl \CO_{\overline v}.$
\subsubsection{Models}
 
  Let $I'=I- \Ram $ which we assume to be nonempty by enlarging $I$. We  first define a regular projective model   $\CM_{U(I')\BB_\infty^\times}$     of  $M_{U(I')\BB_\infty^\times}$ over $X$ as follows.

If $\Ram\neq \{\infty\}$, 
 the notion of  $\cD$-elliptic sheaves  with level-$I'$ structures, 
  is  generalized in    \cite{BS},  \cite{Boy} and \cite{Hau}.
   The image of the morphism $\zero_\BE$ of a  $\cD$-elliptic sheaf $\BE$ (see Definition \ref{DES}) is allowed to be in  $X-I'-(\Ram-\{\infty\})$, $X-\Ram$,  and $X-I'-\{\infty\}$
      respectively.
   Moreover, the corresponding   moduli spaces are obtained over respective open subschemes of $X$ such that every two of these moduli spaces are isomorphic over the open subschemes of $X$ where both are defined. The  define
       $\CM_{U(I')\BB_\infty^\times}$ by gluing these moduli spaces.

     If $\Ram= \{\infty\}$,   define $\CM_{U(I')\BB_\infty^\times}$  using the moduli spaces defined in \cite{DriEll1} and  \cite{BS}. 
     
         \begin{defn}  (1)   Let  $\CM_{U} $  be the minimal desingularization of the  normalization  of   $\CM_{U(I')\BB_\infty^\times} $  in the function field of $M_{U}  $.   Let $\cusp$ be the Zariski closure in $\CM_{U} $ of the cusps in $M_U$.

         (2)   Let  $\CN_{U} $    be the minimal desingularization of the  normalization  of   $\CM_{U(I')\BB_\infty^\times} \times_X X'$  in the ring of rational functions of $M_{U}\otimes_F H $.
Here  $X'$ is the smooth projective curve  corresponding to $H$.
         
           \end{defn}

  Then there is a natural morphism from  $\CN_{U} $  to $\CM_U$.

     \subsubsection{Moduli interpretations of the  $\CM_U$ outside $\Ram$}   \label{mio}    
      
        For $v\in |X|-\Ram$, let $X_{(v)}$ be the spectrum of the localization of $\CO_X$
 at $v$.  In this paragraph, our schemes over $X$ are restricted to $X_{(v)}$. 
  For $v\not \in I\bigcup \Ram$, then  $\CM_U$ equals the moduli space  $\CM_{U(I)U_\infty}$ (over  $X_{(v)}$) defined in  \ref{LSI}.    
 For $v\in I $,  if $\Ram=\{\infty\}$,  then $\CM_{U } $  equals  the quotient by $U_\infty$ of the smooth compactification of the moduli space  of $\cD$-elliptic sheaves with  Drinfeld level-$I$ structures  and level structures  at $\infty$ (see   \ref{Equivalence}  and   \cite{DriCar}).   
    Now   assume that $\Ram\neq \{\infty\}$.    
   Let  $\CM_{U\BB_\infty^\times} $     be the    moduli space of $\cD$-elliptic sheaves with Drinfeld level-$I$   structures    in \cite{Boy}, which is a regular projective model of $M_{U\BB_\infty^\times}$  over $X_{(v)}$.
  Note that the natural morphism from  $\CM_{U(I^v)U_\infty } $ to $\CM_{U(I^v)\BB_\infty^\times}$ is finite \etale. 
Then it is easy to check that $\CM_{U  } =\CM_{U(I^v)U_\infty } \times_{\CM_{U(I^v)\BB_\infty^\times}} \CM_{U\BB_\infty^\times} $.   
             Thus $\CM_{U  } $ is  the quotient by $U_\infty$ of the moduli space of $\cD$-elliptic sheaves  with  Drinfeld level-$I$ structures   and  level structures at $\infty$.
 
\subsection{Local-global decomposition  of the  height distribution}
 \subsubsection{Admissible pairing}
If $U_\infty=\BB_\infty^\times$, define the arithmetic Hodge class $\CL_U$ be the  sum of  the divisor  classes of $2\cdot\cusp$ and of the relative dualizing sheaf of $\omega_{\CM_U/X}$. 
In general, let  $\CL_U$ be the pullback of  the arithmetic Hodge class of  $\CM_{U\BB_\infty^\times}/X$. 
(If  $\Ram\neq \{\infty\}$, then $\CL_U=\omega_{\CM_U/X}$.) 
Then the generic fiber of $\CL_U$  is the hodge class   defined in \ref{Hodge classes}.
            \begin{defn}   \label{ZNU}

Let  $\CZ $ be the pullback of   $\CL_U$ to  $\CN_{U} $, divided by the degree of the restriction of $L_U$   to any   connected component of $M_{U,H}$.
\end{defn}

 We interpret the  \Neron-Tate height pairing on $J_U$   as the   $\CZ$-admissible  pairing in the sense of  \cite[7.1.6]{YZZ}.   
 For a divisor (class) of $M_{U,H}$ or $ \CN_U$ and $\alpha\in \pi_0(M_{U,H})$, use the subscript  $\alpha$ to denote the restriction of this divisor (class) to the connected component of $M_{U,H}$ indexed by $\alpha$.
 Let $\xi_\alpha$ be the generic fiber of $\CZ_\alpha$.
 
 \begin{defn}\label{admext} For  $D \in \Div(M_{U,H})$, let $\overline D$ be the Zariski closure. 
A divisor $\widehat D=\overline D+V_D$ of $\CN_U$, where $V_D$ is a vertical divisor is called the admissible extension of $ D$ if  
\begin{itemize}   \item[(1)] the intersection of  
$\widehat D-\sum_{\alpha\in \pi_0(M_{U,H})} \deg D_\alpha  \cdot \CZ_\alpha$ with every vertical divisor is 0;
\item[(2)] $V_D\cdot \CZ_\alpha=0$ for every $\alpha\in \pi_0(M_{U,H})$.
\end{itemize}
 \end{defn}  
The  admissible extension  of $D$ exists and is unique.
Using a regular model of $M_U$ over  a finite extension  $L$ of $H$ which dominates $\CN_U$ and the pullback of $\CZ$,   extend   Definition \ref{admext} to   $\Div(M_{U,L})$.
For $D_1,D_2\in \Div(M_{U,L})$,
define the   $\CZ$-admissible  pairing $$\pair{D_1,D_2} :=-\frac{1}{[L:F]}\widehat {D_1}\cdot\widehat{ D_2}.$$ 
     Then by the arithmetic Hodge index theorem  (see
\cite[7.1.4]{YZZ}), we have 
    \begin{align}    \pair{ \tilde Z(f)_{*}t_1^\circ ,t_2^\circ }_\NT =  \pair{ \tilde Z(f)_{U,*}t_1  ,t_2  } - \pair{ \tilde Z(f)_{U,*}
\xi_{t_1}  ,t_2  }-  \pair{ \tilde Z(f)_{U,*}t_1^\circ  ,\xi_{t_2} }.\label{Zttdecom}\end{align}
 Here the  left hand side is (part of) the integrand in the  integral \eqref{3.8} defining $H(f)$, and $\xi_{t}=\xi_{\alpha}$ if 
 $t$ is in the connected component   of $M_{U,H}$ indexed by $\alpha$.

 \begin{lem}\label{van23} The second and third term  in the right hand side of \eqref{Zttdecom}  vanish. In particular, the following equation holds \begin{equation}H(f)= \int_{E^\times\bsl \BA_E^\times / \BA_F^\times}\int_{E^\times\bsl \BA_E^\times  }^*    \pair{ \tilde Z(f)_{U,*}t_1  ,t_2  }  \Omega^{-1}(t_2)\Omega(t_1) dt_2dt_1 .\label{nocirc}\end{equation}\end{lem}
\begin{proof} 
For $g\in \BB^\times$ with  $g_\infty\in U_\infty$, we have  $
  Z(g)_{U,*}( 
\xi_{t_1})=\deg Z(g)_U\cdot  \xi _{t_1}$.  (See   \cite[Lemma 7.6]{YZZ}. Indeed, this is  true for all $g\in \BB^\times$ if $\Ram\neq \{\infty\}$.)
Then for $v\in |X|$, 
the   term  $\pair{ \tilde Z(f)_{U,*}
\xi_{t_1}  ,t_2  } $ in \eqref{Zttdecom} equals 
\begin{equation}\Vol(\Xi_U)|F^\times\bsl \BA_F^\times/\Xi|\sum_{n\in \BZ}( \sum_{g_v\in U_v\bsl \BB_v^n /U_v} f_v(g_v)\deg Z(g_v) \pair{ Z(f^v)_{U,*} \xi_{t_1 g_v}  ,t_2  })\label{ccc} .\end{equation}
Let $v\in S_{s,\reg}$.
Fix   $h_n\in  \BB_v^n$. Then   for every  $g_v\in  \BB_v^n$, we have $\xi_{t_1 g_v}=\xi_{t_1 h_n}$.
Thus the inner sum of \eqref{ccc} equals 
\begin{align*}& ( \sum_{g_v\in U_v\bsl \BB_v^n /U_v} f_v(g_v)\deg Z(g_v))\cdot \pair{ Z(f^v)_{U,*} \xi_{t_1 h_n}  ,t_2  }\\
=&( \sum_{g_v\in  \BB_v^n /U_v} f_v(g_v)) \cdot \pair{ Z(f^v)_{U,*} \xi_{t_1 h_n}  ,t_2  }=0\end{align*}
where the last equation follows from Assumption \ref{fvan} (2).
For the third term, similar to  \cite[Lemma 7.7]{YZZ},  we have 
$$\pair{Z(g_v)_{U,*} t_1^\circ  ,\xi_{t_2} }=\deg Z(g_v)\pair{Z(f^v)_{U,*} t_1^\circ,\xi_{t_2}}. $$ 
Note that here we use the fact that  $Z(f^v)_{U,*} t_1^\circ$ has degree 0 which implies that the second term in the right hand side of  \cite[Lemma 7.7]{YZZ} vanishes.
 So the  term $ \pair{ \tilde Z(f)_{U,*}t_1^\circ  ,\xi_{t_2} }$ in \eqref{Zttdecom} vanishes by the same reasoning for  the second term.
   \end{proof}   
  \subsubsection{Decomposition of the height distribution}\label{Decomposition of the height distribution}
  Let $\overline v  $ be   an extension  of $v$ to $  F^\sep$.
For $D_1\in \Div(M_{U,F^\sep})$ and $D_2\in  \Div(M_{U,H})$, let $i_{\overline v }(D_1,D_2)$ and $j_{\overline v }(D_1,D_2)$ be defined as in \cite[7.1.7]{YZZ}.
More precisely, let $Y$ be a finite cover of $X'$ such that $D_1$ is defined over the function field of $Y$. Let   $\CN_U'$ be the base change of $\CN_U$ to $Y$.  Define 
 $i_{\overline v }(D_1,D_2) $ to be the  intersection number   (normalized by the ramification index)  of the Zariski closures of $D_1,D_2$ in $\CN'$, 
  and define  $j_{\overline v }(D_1,D_2)$ to be  the intersection number of the  Zariski closure of $D_1$ with the pullback to $\CN_U'$ of the vertical part in the $\CZ$-admissible extension of $D_2$ in $\CN_U$. 
   
  If  $v$ is not split in $E$,  let              \begin{align*}  i(f)_v= \Vol(\Xi_U)|F^\times\bsl \BA_F^\times/\Xi|\int_{E^\times\bsl \BA_E^\times / \BA_F^\times}\int_{E^\times\bsl \BA_E^\times  }^*  \sum_{g\in \BB^\times/\tilde U} f(g) i_{\overline v}( t_1g,t_2) \Omega^{-1}(t_2)\Omega(t_1) dt_2dt_1      \end{align*} 
 and  \begin{align*} j(f)_v =  \Vol(\Xi_U)|F^\times\bsl \BA_F^\times/\Xi|\int_{E^\times\bsl \BA_E^\times / \BA_F^\times}\int_{E^\times\bsl \BA_E^\times  }^*  \sum_{g\in \BB^\times/\tilde U} f(g) j_{\overline v}( t_1g,t_2) \Omega^{-1}(t_2)\Omega(t_1) dt_2dt_1 .        \end{align*} 
 Here the regularized integral  is as in Definition \ref{regint}.
          If $v$ is split in $E$, let $v_1,v_2$ be the two places over $v$.  Let $\overline { v_i}$ be   an extension  of $v_i$ to $  F^\sep$.
Let   \begin{align*}i(f)_{v}=\frac{1}{2}\sum_{n=1,2} i(f)_{v_n}, \ j(f)_{v}=\frac{1}{2}\sum_{n=1,2} j(f)_{v_n}        \end{align*} 
 where         \begin{align*} i(f)_{v_n} =   \Vol(\Xi_U)|F^\times\bsl \BA_F^\times/\Xi|\int_{E^\times\bsl \BA_E^\times / \BA_F^\times}\int_{E^\times\bsl \BA_E^\times  }^*  \sum_{g\in \BB^\times/\tilde U} f(g) i_{\overline  {v_n}}( t_1g,t_2) \Omega^{-1}(t_2)\Omega(t_1) dt_2dt_1 ,        \end{align*} 
      and
        \begin{align*} j(f)_{v_n} =  \Vol(\Xi_U)|F^\times\bsl \BA_F^\times/\Xi|\int_{E^\times\bsl \BA_E^\times / \BA_F^\times}\int_{E^\times\bsl \BA_E^\times  }^*  \sum_{g\in \tilde U\bsl \BB^\times/\tilde U} f(g) j_{\overline  {v_n}}(Z(g)_{U,*} t_1,t_2) \Omega^{-1}(t_2)\Omega(t_1) dt_2dt_1 .       \end{align*} 
      The expression of $ j(f)_{v_n}$   is useful for the computation (see \ref{11.4.3}). Then            similar to \cite[7.2.2]{YZZ},    by Lemma \ref{van23} and 
      the CM theory  (see Corollary \ref{TSCM}), 
we have
   \begin{equation*}H(f)= \sum_{v\in |X|}   -( i(f)_v+j(f)_v))\log q_{v}.   \end{equation*}

\subsection{Supersingular case}\label{nonsplit}
    
    Let   $v\in |X|-\infty$ be split in $\BB$,    and let $B=B(v)$ be the $v$-nearby quaternion algebra of $\BB.$ Then $B_v$ is a division algebra.  
 Let $n$ be the level of  the  principal congruence subgroup $U_v\subset \BB_v^\times$.   To each $\cD$-elliptic sheaf $\BE$ over $\CO_{F_v}$ (resp.  $\overline  {k(v)}$) with   Drinfeld level structure of level-$nv$, one can associate to
  it a divisible $\CO_{F_v}$-module $ M$ over $\CO_{F_v}$  (resp.  $\overline  {k(v)}$) of  height 4 with $\cD_v$-action and  Drinfeld level structure of level $n$    (see \cite[6.3]{Boy}). 
Fix an isomorphism $\cD_v\cong \RM_2(\CO_{F_v})$. Let \begin{align}  M_1:=\begin{bmatrix}1&0 \\0 &0\end{bmatrix}M,
 \ 
 M_2:= \begin{bmatrix}0 &0 \\0&1\end{bmatrix}M
 \label{special fibers'} \end{align}
which are   divisible $\CO_{F_v}$-modules of height 2.
Then $M=M_1\oplus M_2$ and  $\begin{bmatrix}0&1 \\1 &0\end{bmatrix}$ gives an isomorphism $M_1\cong M_2$ of divisible $\CO_{F_v}$-modules.

Recall that for each positive integer $m$, there is a unique connected formal $\CO_{F_v}$-module of   height  $m$ over  $\overline  {k(v)}$.
If $M_1 $ is not connected, it is the direct sum of the connected  divisible $\CO_{F_v}$-module of   height 1 and the constant  divisible $\CO_{F_v}$-module  $F_v/\CO_{F_v}$.
  \begin{defn}(1) A $\cD$-elliptic sheaf $\BE$ over $\overline  {k(v)}$  is called supersingular if $M_1$ is connected.
  Otherwise    $\BE$ is called ordinary.  
    A point in $ \CM_{U ,\overline {k(v)}}$ is called supersingular if the underlying $\cD$-elliptic sheaf   is supersingular.
  
   (2) Let $\CM_{U ,\overline {k(v)}}^\sing\subset \CM_{U ,\overline {k(v)}}$    be the subset of  supersingular   points.  Let $\CM_{U,\overline {k(v)}}^\sing\subset \CM_{U,\overline {k(v)}}$  be the subset  of points
whose images in     $\CM_{U ,\overline {k(v)}}$ is contained in  $\CM_{U ,\overline {k(v)}}^\sing$.
  Points in  $\CM_{U,\overline {k(v)}}^\sing$ are called superinsgular points.
 Points
  in $  \CM_{U,\overline {k(v)}}$ outside   $\CM_{U,\overline {k(v)}}^\sing$ are called ordinary points.
  
  (3) 
   An irreducible component in the special fiber  of $\CN_{U,\hat\CO_{H_w}^\ur} $ is called a  supersingular component if its
   image in the special fiber of $\CM_{U,\hat\CO_{F_v} ^\ur}$ is    a supersingular point. Otherwise it is  called an  ordinary component. 

     \end{defn}   
  
By  \cite[(10.6)]{LLS} and \cite[Proposition 10.2.2]{Boy}, we have a bijection
     \begin{equation}\label{ss12'}\CM_{U,\overline {k(v)}}^\sing\cong B^\times\bsl    \BZ \times \BB^{v,\times}/ \tilde U^v.    
 \end{equation}
     
  The following lemma is easy to prove. 
         \begin{lem} \label{lieord}     
 
    If $v$ is split in $E$,  points in $CM_U$ have   reductions outside supersingular components on $\CN_{U,\CO_{H_w}}$. 
 Otherwise  they  have   reductions in supersingular components.
        \end{lem}

    
\subsubsection{Components} \label{special fibers}
 Let $F_U$ be the abelian extension of $F$ corresponding to  $F ^\times\bsl \BA_{F }^\times/\det (\tilde U)$ by the class field theory.  
 Let $\CM_{U,(v)}$ be the restriction of $\CM_U$ to $X_{(v)}$ (see \ref{mio}).  By  
\ref{mio} and the determinant  construction of $\cD$-elliptic sheaves (see \cite{LL97}), 
there is a natural projection $$\Det_U:\CM_{U,(v)}\to\Spec \CO_{F_U,(v)}  $$ of $X_{(v)}$-schemes.
Here $\Spec \CO_{F_U,(v)} $ is identified with a moduli space of rank 1 elliptic sheaves with   (Drinfeld) level structures associated to $\det (\tilde U)$.
 The base change of $\Det_U$ via the embedding $F\to \BC_\infty$ gives a morphism
$$\Det_{U,\BC_\infty}: M_{U,\BC_\infty}\to\Spec \BC_\infty\times  F ^\times\bsl \BA_{F }^\times/\det (\tilde U) . $$ 
After rigid analytification, $\Det_{U,\BC_\infty}$ coincides with the combination of the rigid analytic uniformization 
Proposition 6.6\ref{riguni} and  the determinant  morphism  (over $\BC_\infty$) on the Drinfeld's covering space in \cite[IV]{{Gen}}. The fibers of   the determinant  morphism  on Drinfeld's covering space are connected.
 In particular, the fibers of $\Det_{U,\BC_\infty}$  are the connected components of $M_{U,\BC_\infty}$. 
   
   Let $\CV^\ord$ be the set of ordinary components  in the special fiber  of $\CN_{U,\hat\CO_{H_w}^\ur} $.    
 Then $\Det_{U,\hat\CO_{H_w}^\ur}$ (combined with  
 $\CN_{U,\hat\CO_{H_w}^\ur}\to \CM_{U,\hat\CO_{H_w}^\ur}$) induces a map $\CV^\ord\to F ^\times\bsl \BA_{F }^\times/\det (\tilde U) $ whose fibers are connected components of 
$\CN_{U,\hat\CO_{H_w}^\ur} $  
 by Zariski's connectedness theorem.
      Recall that a Drinfeld level structure of level $n$ on the divisible $\CO_{F_v}$-modules $M_1$ is an  $\CO_{F_v}$-module morphism 
      $$((\varpi_v^{-n}\CO_{F_v})/\CO_{F_v})^2\to M_1[\varpi_v^n]$$ satisfying certain conditions  (see \cite[Secton 4, B)]{DriEll1}). 
         The  $ \CO_{F_v}$-submodules
          of $((\varpi_v^{-n}\CO_{F_v})/\CO_{F_v})^2$ of rank 1 are indexed by $ B_v\bsl \GL_2(F_v)/U_v,$
          where $B_v$ is the standard Borel subgroup.
By the   argument  in \cite[9.4]{Car} (see also \cite[10.4]{Boy}), we have the following results. On each irreducible component of $\CM_{U,\overline {k(v)}}$, the Drinfeld level structure   vanishes on a  corresponding rank-1 $ \CO_{F_v}$-submodule. In particular, we have a map from irreducible components of $\CM_{U,\overline {k(v)}}$ to $B_v\bsl \GL_2(F_v)/U_v$. 
Combined with  
 $\CN_{U,\overline {k(w)}}\to \CM_{U,\overline {k(v)}}$,
we have a map $\CV^\ord\to   B_v\bsl \GL_2(F_v)/U_v$. 

  \begin{prop}\label{123} 
  Irreducible  components of $\CM_{U,\overline {k(v)}}$ (with reduced scheme structures) are smooth. The ones in the same connected component 
   intersect   at  supersingular points.
Moreover,  we have a bijection  $$\CV^\ord\cong B_v\bsl \GL_2(F_v)/U_v  \times  F ^\times\bsl \BA_{F }^\times/\det (\tilde U)  .$$

  \end{prop}

  \subsubsection{Uniformizations}\label{Serre-Tate theory,   the multiplicity function}
Assume that $v$ is not split in $E$. 
     Let $\Art_{\hat \CO_{F_v}^\ur}$ be the category of complete artinian $\hat \CO_{F_v}^\ur$-algebras with residue fields isomorphic to  $\overline  {k(v)}\cong\hat \CO_{F_v} ^\ur/\varpi_v$. 
  Let $  \CS_{U_v}$ be the level $n$  deformation   space of the unique formal $\CO_{F_v} $-module $\BM$ of    height 2 over $\overline  {k(v)}$. 
  In other words,  an element in  $ \CS_{U_v}(R)$, where $R\in \Art_{\hat\CO_{F_v} ^\ur}$,   is a(n isomorphism class) of  formal $\CO_{F_v} $-module $M$ over $R$ with  Drinfeld level  structure of level $n$ 
  and a    quasi-isogeny  \begin{equation}M_{\overline  {k(v)}}\to\BM.\label{left}\end{equation}  Then there is a natural projection
 \begin{equation}\CS_{U_v}\to\BZ\label{CSBZ}\end{equation} which maps an element in  $ \CS_{U_v}(R)$ to the (valuation of the) degree of the quasi-isogeny.
Let   
  $g\in B_v^\times$ act  by $v(\det(g))$.  
  There is a natural left action of $$B_v^\times \cong (\End(\BM)\otimes _{\CO_{F_v}}  F_v)^\times$$ on  $\CS_{U_v} $  which respect the morphism \eqref{CSBZ} and 
  the actions of $ B_v^\times$ on both sides.          
 
  Let $\hat \CM_{U }$ be the formal completion of $\CM_{U, \hat \CO_{F_v}^\ur}$ along the    supersingular locus. 
  The Serre-Tate theory for $\cD$-elliptic sheaves 
  (see \cite[Proposition 5.4]{DriEll1}, \cite[Theorem 7.4.4]{Boy})  gives a formal uniformization of $\hat\CM_U$ (see \cite[Proposition 14.1]{Boy}):     \begin{equation}\label{ST}
    \hat \CM_{U }  \cong B^\times\bsl \CS_{U_v}  \times \BB^{v,\times}/\tilde U^v.
 \end{equation}    
 Under the isomorphisms   \eqref{ss12'} and \eqref{ST},  the reduction map $\hat \CM_{U } \to \CM_{U,\overline {k(v)}}^\sing$ is given by the map $\CS_{U_v}   \to \BZ.$   

Let $\hat \CN_U$ be the  formal completion of $\CN_{U,\hat\CO_{H_w} ^\ur}$ along the union of supersingular components.  Let $\tilde \CS_{U_v}$   be the minimal desingularization of the normalization of $\CS_{U_v  }$ in the ring of meromorphic functions of $ \CS_{U_v}\hat\otimes_{\hat\CO_{F_v}^\ur}\hat \CO_{H_w}^\ur$. Then $\tilde\CS_{U_v}  $ admits an action of $B_v^\times $   by the functoriality of  minimal desingularization.
Then \eqref{ST} induces an isomorphism    of formal schemes:
          \begin{equation} \label{CNU}\hat \CN_U \cong B^\times\bsl  \tilde \CS_{U_v}   \times \BB^{v,\times}/\tilde U^v.
 \end{equation}

 Define $\CH_{U_v}$ to be the set of (isomorphism
classes of) quasi-canonical liftings    of $\BM$   to $\CO_{\overline v}$  with   Drinfeld level structures of level $n$.
Here a lifting is called quasi-canonical if its endomorphism ring contains $E_v$.
Then  there is a canonical $ B_v^\times  $-action on $\CH_{U_v}$, and a  $ B_v^\times  $-equivariant map  
  \begin{equation}\CH_{U_v}\to \CS_{U_v}(\CO_{\overline v}) .\label{HUv'}
 \end{equation} 
        Consider   maps \begin{equation}CM_U\to \CM_U(\CO_{\overline v}) \to \hat \CM_U(\CO_{\overline v})\label{113}\end{equation}  where the first map sends a CM point to the base change to $\CO_{\overline v}$ of  its Zariski closure in $\CM_U$ and the second map is the  natural one. 
We   have a ``uniformization" of $CM_U$ by $\CH_{U_v}$ and a ``uniformization" of the composition of \eqref{113} by \eqref{HUv'} as follows.    

Let $t\in E_v^\times $ act on $B_v^\times$ via  right multiplication  by $t^{-1}$, and act on $\BB_v^\times$  via  left multiplication by $t$.
 By     \cite[5.5]{Zha01}, 
 there is a  $ B_v^\times  $-equivariant  bijection 
\begin{equation*}\CH_{U_v}\cong B_v^\times\times_{E_v^\times}\BB_v^\times/U_v   ,  \end{equation*} 
 where
  $ B_v^\times  $-action on the right hand side is by  left multiplication on $B_v^\times$.    This bijection depends on the choice of    the preimage of $(1,1)$ in $ \CH_{U_v}$. 
We take $(1,1)$ to be the  formal $\CO_{F_v}$-module of height 2 with Drinfeld level structure of level $n$ associated to $\cD$-elliptic sheaf represented by $P_0$ (see \ref{CM points}).  We also have an embedding $E_v^\times\incl \CH_{U_v}$  by $t\mapsto (t ,1)=(1,t)$.

The   bijection \eqref{111} induce   a bijection \begin{equation}CM_U\cong B^\times\bsl \left ( (B ^\times\times_{E ^\times}\BB_v^\times/U_v) \times \BB^{v,\times}/\tilde U^v\right),\label{CMUv}
 \end{equation} 
by sending $\beta\in E^\times\bsl \BB^\times/  \tilde U$ to $((1,\beta_v),\beta^v)$.  
Under the natural inclusion
\begin{equation}B ^\times\times_{E ^\times}\BB_v^\times/U_v\incl \CH_{U_v},\label{HUv}
 \end{equation} 
 we regard $\CH_{U_v}$  as an ``uniformizing space" of $CM_U$.
 By 
  \eqref{CMUv} and   \eqref{ST}, the composition of \eqref{HUv} and \eqref{HUv'} induces  
 a map \begin{align} CM_U \to  B^\times\bsl  \CS_{U_v} (\CO_{\overline v})   \times \BB^{v,\times}/\tilde U^v \cong \hat \CM_U(\CO_{\overline v}) .\label{118} \end{align}
Up to choices of  the data away from $v$ defining \eqref{ST}, we  have the following result. \begin{lem} \label{11.2.6}The map \eqref{118} is the composition of \eqref{113}.
 \end{lem}

     A point in $\CS_{U_v}(\CO_{\overline v})$ lifts to a point in $\tilde \CS_{U_v}(\CO_{\overline v})$ ($\tilde \CS_{U_v}$ is regarded as over $\hat\CO_{F_v}^\ur$) by strict  transform.
Thus the  map $\CH_{U_v}\to \CS_{U_v}(\CO_{\overline v}) $ lifts to a $ B_v^\times  $-equivariant map \begin{align}\CH_{U_v}\to \tilde\CS_{U_v}(\CO_{\overline v}).\label{119}\end{align}  
This map induces a  map \begin{align}\begin{split}&CM_U\cong B^\times\bsl \left ( (B ^\times\times_{E ^\times}\BB_v^\times/U_v) \times \BB^{v,\times}/\tilde U^v\right)\\  
 \to &\hat \CN_U(\CO_{\overline v})\cong B^\times\bsl  \tilde\CS_{U_v} (\CO_{\overline v}) \times \BB^{v,\times}/\tilde U^v .\label{120}\end{split}\end{align}  Lemma \ref {11.2.6} implies the following result.
  \begin{lem} \label{11211}The map \eqref{120} coincides with the composition of   maps
$$CM_U\to\CN_U  (\CO_{\overline v})\to  \hat \CN_U  (\CO_{\overline v}) $$ 
where the first map maps a CM point to the base change to $\CO_{\overline v}$ of  its Zariski closure in $\CM_U$ and the second map is the  natural one. 
\end{lem}
  \subsubsection{Multiplicity function}
   
 \begin{defn}\label{mf1} Define the multiplicity function $m_v$ on $ \CH_{U_v}-\{(1,1)\}$ as follows: for $(g_1 ,g_2 )\in \CH_{U_v}  $ 
 and $(g_1,g_2)\neq (1,1)$, let  $m_v(g_1 ,g_2 )$   
 to be the  intersection number  of the  images of the
  points $(g_1,g_2)$ and $(1,1)$ in $\tilde \CS(\CO_{\overline v}) $ under \eqref{119}.  
  
 \end{defn} We have the following properties of the multiplicty function  $m_v$.
 \begin{lem}\label{bequiv1}
(1) The following equation holds: $  m_v(g_1^{-1} ,g_2^{-1} )= m_v(g_1 ,g_2 ).  $

(2) The number
$m_v(g_1,g_2)$ is the  intersection number of the  images of the
  points $(gg_1,g_2)$ and $ (g,1)$  in $\tilde \CS(\CO_{\overline v}) $  for every $g\in B_v^\times .$  
 \end{lem}
 \begin{proof} (1)  follows from   definition.
 (2) follows from the $B_v^\times$-equivariance of the map \eqref{119}.
 \end{proof}

 Now we consider the nonvanishing of $m_v$.  
      Let $F_{U_v}$ be the totally ramified abelian extension of $\hat F_v^\ur$ corresponding to $\det(U_v)$ via local class field theory, so that $(\Spec \CO_{F_{U_v}})(\CO_{\overline v})\cong \CO_{F_v}^\times/\det(U_v)$. 
      There is   a natural  morphism $\CS_{U_v} ^0\to \Spec \CO_{F_{U_v}}$   constructed in  \cite{Str}. 
(In fact, this morphism comes from the    determinant  construction of  formal modules, see \cite{Hed}.)
Its composition with the natural morphism   $\tilde \CS_{U_v} ^0\to \CS_{U_v} ^0$ is denoted by
\begin{equation}\label{Detuv}\Det_{U_v} : \tilde \CS_{U_v}^0 \to \Spf \CO_{F_{U_v}}.\end{equation}
The composition of  $B_v^\times\times_{E_v^\times}\BB_v^\times/U_v\cong \CH_{U_v}   $ with
\begin{equation} \CH_{U_v}   \to \tilde \CS_{U_v}(\CO_{\overline v})\to\CS_{U_v}(\CO_{\overline v}) \to \BZ\label{CCBB}\end{equation}
 is given   $(g_1,g_2)\mapsto v(\det(g_1)\det(g_2)).$ Let $\CH_{U_v} ^0$ be the preimage of 0. Then 
the composition of 
$$\CH_{U_v} ^0\to \tilde \CS^0(\CO_{\overline v})  \to \tilde \CS_{U_v}^0 (\CO_{\overline v}) \xrightarrow{\Det_{U_v}}(\Spec \CO_{F_{U_v}})(\CO_{\overline v})\cong \CO_{F_v}^\times/\det(U_v) $$    is given by  $(g_1,g_2)\mapsto \det(g_1)\det(g_2).$
       Thus we have the following lemma.   \begin{lem}\label{vdetun1}The multiplicty function $m_v(g_1,g_2)\neq 0$ only if $ \det(g_1)\det(g_2) \in \det(U_v)$.
  \end{lem} 
   
   \subsubsection{Compute $i(f)_v$}
Let $i_{\overline v}(t_1g,t_2)$ be as in the definition of $i(f)_v$ 
 (see \ref{Decomposition of the height distribution}). 
  \begin{lem}\label{heightpullpack} 
Let $x,y\in   CM_U$
be   \textit{distinct} CM points. Assume  $x=t_1g,y=t_2$ for $t_1,t_2\in \BA_E^\times $ and $g\in \BB^\times$, then 
$$i_{\overline v}(x,y)=\sum_{\delta\in   B^\times} m_v(t_{1,v}^{-1}\delta t_{2,v},g_v ^{-1})  1_{\tilde U^v}(((t_1 g )^{-1}  \delta t_2 )^v).$$
 Moreover, this is a finite sum.
  
\end{lem} 
\begin{proof}
The decomposition is standard, see \cite[Lemma 8.2]{YZZ}. The particular   expression here follows from Lemma \ref{bequiv1}.  
Only need to show that the sum is a finite sum.
Since $B_v $ is a division algebra, the preimage of  an open compact  subset of $F_v^\times$ under the reduced norm $\det$ is compact. By Lemma  \ref{vdetun1},
  the function $$m_v(t_{1,v}^{-1}\delta t_{2,v},g_v ^{-1})  1_{\tilde U^v}(((t_1 g )^{-1}  \delta t_2 )^v)$$ for $\delta\in B^\times(\BA_F)/\varpi_\infty^\BZ$ is only nonvanishing   on a compact subset.
  Since the image of the inclusion $B^\times\incl B^\times(\BA_F)/\varpi_\infty^\BZ$ is discrete and closed,  the sum is a finite sum.
 \end{proof}

    To compute $i(f)_v$, we first deal with the regularized integral involved  in its definition (see \ref{Decomposition of the height distribution}).
 By   a direct computation and Fubini's theorem, we have the following lemma.
     
     \begin{lem} \label{innersum}  Let $V$ be an open  compact subgroup of  $\BA_{F,\mathrm{f}}^\times$ whose intersection with $F^\times$ is $\{1\}$, 
  $\delta \in B^\times_\reg$, and let  $\phi$ be a  function on $  \BA_E^\times\delta \BA_E^\times $ which is   $V$-invariant and $\Xi_\infty$-invariant.   
 Then  if either side of the   equation   \begin{align*}& \Vol(V) \int_{E^\times\bsl \BA_E^\times / \BA_F^\times}\int_{E^\times\bsl \BA_E^\times  / \BA_F^\times}  \sum_{t\in F^\times\bsl \BA_F^\times/\Xi_\infty V}\sum_{x\in     E^\times \delta E^\times} \phi(t_1^{-1}xtt_2)d{t_2}d{t_1} \\
  =& \int_{ \BA_E^\times / \BA_F^\times}\int_{ E^\times(\BA_{F,\mathrm{f}}) }\int_{ E^\times(F_\infty)/\Xi_\infty } \phi(t_1^{-1}\delta t_{2,\mathrm{f}} t_{2,\infty})d_{t_{2,\infty}}d{t_{2,\mathrm{f}}}d{t_1}      \end{align*}
      converges  absolutely, the other side also converges absolutely,
 and in this case the   equation holds.
 
 \end{lem}
 

  
Use Assumption \ref{fvan} (1) to get rid of contributions from singular orbits in $i(f)_v$ as follows. 
\begin{lem}\label{asmp4im}(1) If $t_1g=t_2$ as points in $  M_U$,  then $f(g)=0.$

(2)  Let ${g\in \BB^\times/\tilde U} $, we have      \begin{align*}  
 & f(g) \sum_{\delta\in   B^\times} m_v(t_{1,v}^{-1}\delta t_{2,v},g_v ^{-1})  1_{\tilde U^v}(((t_1 g )^{-1}  \delta t_2 )^v)\\
 =& f(g)\sum_{\delta\in   B_\reg^\times} m_v(t_{1,v}^{-1}\delta t_{2,v},g_v ^{-1})  1_{\tilde U^v}(((t_1 g )^{-1}  \delta t_2 )^v).\end{align*}

\end{lem}
\begin{proof}(1)  If $t_1g=t_2$, then   $g\in \BA_E^\times \tilde U$. Thus $f(g)=0$ by Assumption \ref{fvan} (1)  and the $\tilde U$-invariance of $f$.
(2) Let $\delta\in B^\times_\sing$. If $1_{\tilde U^v}(((t_1 g )^{-1}  \delta t_2 )^v)\neq 0$,  then
$g^v\in t_1  ^{-1}  \delta t_2 ^v  \tilde U ^v$.  So $f(g)=0$ by Assumption \ref{fvan} (1)  and the $\tilde U$-invariance of $f$.
The equation follows. \end{proof}

Now we compute $i(f)_v$.  Lemma \ref{asmp4im} validates  the third ``$=$"  in the	 equation
  \begin{align*}
   i(f)_v&=\int_{E^\times\bsl \BA_E^\times / \BA_F^\times}\int_{E^\times\bsl \BA_E^\times  }^*  i_{\overline v}( \tilde Z(f)_{U,*}t_1,t_2)\Omega^{-1}(t_2)\Omega(t_1) dt_2dt_1\\
   &=\Vol (\Xi_U) |F^\times\bsl \BA_F^\times/ \Xi|\int_{E^\times\bsl \BA_E^\times / \BA_F^\times}\int_{E^\times\bsl \BA_E^\times  }^*\sum_{g\in \BB^\times/\tilde U}        f(g)   i_{\overline v}( t_1g,t_2)\Omega^{-1}(t_2)\Omega(t_1) dt_2dt_1\\
   &=  \Vol (\Xi_U) |F^\times\bsl \BA_F^\times/ \Xi|\int_{E^\times\bsl \BA_E^\times / \BA_F^\times}\int_{E^\times\bsl \BA_E^\times  / \BA_F^\times} \frac{1}{|F^\times\bsl \BA_F^\times/ \Xi|}
    \\
   &\sum_{t\in F^\times\bsl \BA_F^\times/\Xi_\infty   \Xi_U}\sum_{g\in \BB^\times/\tilde U}        f(g)   \sum_{\delta\in   B_\reg^\times}m_v(t_{1,v}^{-1}\delta tt_{2,v},g_v ^{-1})  1_{\tilde U^v}(((t_1 g )^{-1}  \delta tt_2 )^v) \omega^{-1}(t)\Omega^{-1}(t_2)\Omega(t_1) dt_2dt_1. \end{align*}

Applying Lemma \ref{innersum} 
and Fubini's theorem, we have
 \begin{align} \begin{split}
   i(f)_v
   &  = \sum_{\delta\in E^\times\bsl B_\reg^\times/E^\times} \int_{ \BA_E^{v,\times} / \BA_F^{v,\times}}\int_{ \BA_E^{v,\times} /\Xi_\infty }  f^v( (t_{1 } ^{v})^{-1}  \delta t_{2}^v   ) \Omega^{-1}(t_2^v)\Omega(t_1^v)  d {t_{2}^v}d{t_1^v}\\ 
      &\ \  \ \ \   \sum _{g_v\in  \BB_v^\times/U_v}f_{v}(g_v)   \int_{ E^\times_v/F_v^\times}\int_{ E^\times_v}  m_v( t_{1,v} ^{-1}  \delta t_{2,v}  ,   g_v^{-1} )     \Omega_v^{-1}( t_{2,v})\Omega_v(t_{1,v})   dt_{2,v}d t_{1,v}\label{1118}
  \end{split}     \end{align}
 provided the right hand side is absolutely convergent. The absolutely convergence is as follows.
First,    the integral of $m_v$ is absolutely  convergent by Lemma \ref{vdetun1}. 
 Second, the integral over  $\BA_E^{v,\times} / \BA_F^{v,\times}\times   \BA_E^{v,\times} /\Xi_\infty $
  decomposes into a product of local  orbital integrals such that for almost all places the values are 1 (see Lemma \ref{sint=1}, \ref {sint=1'}). 
So  we only need to prove that   the right hand side is a finite sum.
Let the invariant map $\inv_{E^\times}':B^\times(\BA_F)\to\BA_F$  be the product of the local ones 
$$\inv_{E_v^\times}':B_v^\times\to F_v$$defined in    \eqref{inv'}.
Since  $\Nm (E_v^\times)\not\in \inv _{E^\times}(B_v^\times)$,  an open neighborhood of 1 is not  contained in $\inv _{E^\times}(B_v^\times)$.
So $ \inv _{E^\times}'(B_v^\times)$
  is contained in a compact subset of $F_v$. 
In particular, with $g_2$ fixed, the subset 
$\inv_{E^\times}'(\{g_1\in (B_v^\times)_\reg  :m_v(g_1   ,g_2  ) \neq0\}) $
of $ \inv _{E^\times}'(B_v^\times)$
  is contained in a compact subset of $F_v$. 
By Proposition \ref{Proposition 2.4, Proposition 3.3Jac862},   there is  a compact subset of $\BA_F$, such that the summand in the right hand side is nonzero  only 
   if $\inv_{E^\times}'(\delta)$ is in this compact subset. Since $\inv_{E^\times}'(B^\times)$ is a discrete closed  subset of $F$,   the sum on  the right hand side is a finite sum.

        \begin{defn}\label{iterm1}For $\delta\in B^\times_\reg$, define the   arithmetic orbital integral of  the multiplicty function $m_v$ weighted  by  $f_v$ to be 
        $$ i (\delta,f_v)  :=  \int_{ E^\times_v/F_v^\times}\int_{ E^\times_v} \sum _{g_v\in  \BB_v^\times/U_v}f_{v}(g_v)  m_v( t_{1,v} ^{-1}  \delta t_{2,v}  ,   g_v^{-1} ) 
       \Omega_v^{-1}( t_{2,v})\Omega_v(t_{1,v})   dt_{2,v}d t_{1,v}  .  $$   \end{defn}
     By \eqref{1118}, we have the following proposition.
     \begin{prop}\label{regwell}
     Under Assumption \ref{fvan},  we have
 \begin{equation*} i(f)_v= \sum_{\delta\in E^\times\bsl B_\reg  ^\times/E^\times} \CO_{\Xi_\infty} (\delta,f^v)  
 i(\delta,f_v). 
 \end{equation*}
    \end{prop}
    


\subsubsection{Compute the multiplicity function $m_v$}
 We will use some computations  in \cite{Zha01} \cite{YZZ}. They are   based on Gross'  results   \cite{Gro}, which works in arbitrary characteristics.
 
     We first assume that $U_v$ is maximal. 
 Then $\CM_{U,\CO_{F_v}} $ is a smooth  model. So  $\CN_{U,\CO_{H_w}} \cong\CM_{U,\CO_{F_v}} \otimes_{\CO_{F_v}} \CO_{H_w} $ and  is a smooth model.
     Then $m_v$ can be computed 
     using Gross's theory of quasi-canonical lifting. 
    Fix an isomorphism $\GL_2(F_v)\cong \BB_v^\times$ which maps  $\GL_2(\CO_{F_v})$ to $U_v$.

 \begin{lem}  \label{GL2decom} 
 Let $h_c=\begin{bmatrix}\varpi_v^c&0 \\0 &1\end{bmatrix}\in   \BB_v^\times$.
Under Assumption \ref{asmpe0},
 there is a decomposition 
 $$\BB_v^\times =\coprod_{c\in \BZ_{\geq 0}}   E_v^\times h_c U_v.$$
  
 \end{lem}  
\begin{lem} [{\cite[Lemma 5.5.2]{Zha01} \cite[Lemma 8.6]{YZZ}}]\label{unramified multiplicity function} 
The multiplicity function $m_v$
    on  $\CH_{U_v} $     is nonzero 
 only if $\det (g_1)\det(g_2)\in \CO_{F_v}^\times  $.
 In this case, assume $g_2\in E_v^\times h_c U_v $, then
 \begin{itemize}
 \item[(a)] if $c=0$ then $m_v(g_1,g_2)=\frac{1}{2}(v(\inv_{E_v^\times}' (g_1))+1)$;
 
 \item[(b)] if $c>0$ and  $E_v/F_v$ is unramified, then $m_v(g_1,g_2)=q_v^{1-c}(q_v+1)^{-1};$

 \item[(c)] if $c>0$ and  $E_v/F_v$ is  ramified, then $m_v(g_1,g_2)=\frac{1}{2}q_v^{ -c}.$
  \end{itemize}

   \end{lem}
   
Now suppose $U_v  $  is   principal of level $n>0$.  Assume $g_2=1$. Then
  by Lemma \ref{vdetun1}, $m(g_1,1)$ is supported on $B_{v,0}:=\{g_1\in B_v^\times: \det (g_1)\in \det U_v\}$ (and  not defined on   $   E_v^\times$).

   \begin{lem}\label{U'}There is an open compact subgroup $U'\subset B_{v,0}$ such that
   \begin{itemize}
  \item[(1)]  as subgroups of $E_v^\times$, $U'\cap E_v^\times=U_v\cap  E_v^\times$:
  \item[(2)]  the function $$m (g_1,1     ) - \frac{v (\inv_{E_v^\times}'(g_1)) }{2} 1_{U'}(g_1) $$
  on  $ B_{v,0} -   E_v^\times$  
  can be extended to a   locally constant function      on $B_{v,0} $.
   \end{itemize}
   \end{lem} 
   \begin{proof}   
 For $ (g_1,g_2)\in \CH_{U_v}$, let $\red(g_1,g_2)$ be the reduction of its image in $\tilde \CS_{U_v}$. 
Let $U'$ be the maximal open compact subgroup  through which the restriction of the  $B_{v,0}$-action  on $\red(1,1)$ factors. Then  the inclusion    $U'\cap E_v^\times\supset U_v \cap E_v^\times$   and 
(2)  follows from the argument in \cite[Lemma 5.5.3, Lemma 5.5.4]{Zha01}. 
Now we prove   (1). For  $t\in B_{v,0} \cap E_v^\times-U_v\cap E_v^\times $, the reductions of the  images of $(t^{-1},1)$ and $(1,1)$ in 
  $ \CS_{U_v}\hat\otimes_{\hat\CO_{F_v}^\ur}\hat \CO_{H_w}^\ur$  are the same. Thus the blowing-up process separates  the  images of $(t^{-1},1)$ and $(1,1)$ in  $\tilde \CS_{U_v}$, i.e. $\red(t^{-1},1)\neq \red(1,1)$. Thus $U'\cap E_v^\times\subset U_v \cap E_v^\times$  and (1) follows.
      \end{proof}
   
           \subsubsection{Compute $j(f)_v$}    
Now we compute the   $j(f)_v$.  
  
\begin{prop}\label{nonsplitj} There exists $\overline {f_v}\in C_c^\infty(B_v^\times)$ such that 
  $$j(f)_v= \sum_{\delta\in E^\times\bsl B_\reg^\times/E^\times}\CO _{\Xi_\infty}(\delta,f^v)   
\CO (\delta ,\overline {f_v})    .$$  
      \end{prop}   
 We will prove Proposition \ref{nonsplitj} in  \ref{cadmissible extensions} after some preparations.  

   Let $\CV^\sing$ be the set of  supersingular components  $\CN_{U,\hat\CO_{H_w}^\ur}  \otimes_{\hat \CO_{H_w}^\ur}\overline{k(w)}$. Let $  \CV $ be the set of   exceptional  (irreducible reduced) curves of $\tilde \CS_{U_v}$ (from the desingulariation), contained in $\tilde \CS_{U_v}  \otimes_{\hat \CO_{H_w}^\ur}\overline{k(w)}$, then there is a natural action of $B_v^\times$ on $\CV$.      From   \eqref{CNU}, we have  a bijection
     \begin{equation} \CV^\sing \cong B^\times\bsl   \CV \times \BB^{v,\times}/\tilde U^v.\label{vsing}\end{equation}
For $C\in \CV$, $g\in\BB^{v,\times}$, let $[C,g]$ be the corresponding element in $\CV^\sing.$

      \begin{defn} \label{66}Let $C\in   \CV$,  define a function $l_{ C}$  on $ \CH_{U_v}$ 
  as follows. For $(g_1,g_2)\in   \CH_{U_v}$, let $l_{ C}(g_1,g_2)$ be the  intersection number  of $C$ and the  image  of the 
  point $(g_1,g_2)$   in $\tilde \CS(\CO_{\overline v})$.  
  \end{defn}
  Since points in $  \CH_{U_v}$ are in fact defined over $\hat \CO_{H_w}^\ur$,   the image of the first map 
   in \eqref{CCBB} is in $ \tilde \CS(\hat \CO_{H_w}^\ur) $. Since $ \tilde\CS $ is regular, the reductions of points in  $ \tilde\CS(\hat \CO_{H_w}^\ur) $  the are in the smooth locus of $ \tilde\CS \otimes_{\hat \CO_{H_w}^\ur}\overline{k(w)}$. Thus we have  a map  $  \CH_{U_v}\to   \CV$.
    The  maps  in \eqref{CCBB} also induce    a   map  $ \CV\to \BZ .$  

  \begin{lem} \label{9123}The function $l_C$ satisfies the following properties:
  \begin{itemize}
 
  \item[(1)] for $h \in \CH _{U_v}$, $l_{C}(h)\neq 0$ 
  only if  the image of
  $  C$ in $\BZ$  under the map $       \CV\to \BZ$ and 
  the image of $h$  under  $\CH_{U_v}\to \BZ$ (the composition of   \eqref{CCBB}) are the same;
  
 \item[(2)]  for $b\in B_v^\times $, $h \in \CH_{U_v} $, $l_{bC}(bh)=l_C(h)$; 
  
 \item[(3)]  $l_C$    is locally constant.
\end{itemize}
     \end{lem}
     \begin{proof} (1)   follows from Definition \ref{66}. (2) follows from the fact that the action by $b$ is  an isomorphism.
      (3) follows from (2) and the fact that  the $B_v^\times$-action on the special fiber of $\tilde\CS$ factors through the open compact subgroup $U'$ as in Lemma \ref{U'}. 
     \end{proof}
   Similar to Lemma \ref{heightpullpack}, we have the following lemma. 
   \begin{lem}\label{heightpullpack'}
Let $g \in( \BB ^{v})^{\times}$ and $y\in CM_U   \cong  B^\times\bsl \left ( (B ^\times\times_{E ^\times}\BB_v^\times/U_v) \times \BB^{v,\times}/\tilde U^v\right)$, the intersection number of $(C,g)\in \CV^\sing$ and the Zariski closure of $y $ in $\CN_U \otimes _{\CO_{H_w}}\CO_{\overline v}$   is given by 
  $$\sum_{\delta\in   B^\times} l_C ( \delta^{-1} ,y_v)1_{\tilde U^v}(g^{-1}\delta^{-1} y^v).$$ 
   Moreover, this is a finite sum.

  \end{lem}
   The finiteness of the sum is implied by Lemma \ref{9123} (1).
      
         \subsubsection{Compute $\CZ$-admissible extensions}
    For $t \in  \BA_E^\times  $,  regarded as   a point $ CM_U$,   let $\overline  t $ be the 
       Zariski closure  of $t$ in $\CN_{U,\hat\CO_{H_w}^\ur}$, and $
    \hat t $     the  $\CZ$-admissible extension $\hat t$ of $t $ in $\CN_U
    $. There exists $ C_i\in     \CV  $ 
      such that $$\hat 1=\bar 1+\sum_{i=1}^n a_i[ C_i,1] -(C_{\ord,1},1).$$ 
Here  $C_{\ord,1} $ is a linear  combination of elements in $B_v\bsl \GL_2(F_v)/U_v$ and $(C_{\ord,1},1)$
represents a  linear  combination ordinary components via the canonical bijection  in  Proposition \ref{123}:
 $$\CV^\ord\cong B_v\bsl \GL_2(F_v)/U_v  \times  F ^\times\bsl \BA_{F }^\times/\det (\tilde U)  .$$
 Let  $Z_{t_v} $ be  the image of $ (1,t_v)\in \CH$ in $\tilde \CS_{U_v} $, then  the image of $\bar t$ in $\hat \CN_{U} $
  under the isomorphism
 \begin{equation}  \hat \CN_U \cong B^\times\bsl \tilde \CS_{U_v}   \times \BB^{v,\times}/\tilde U^v\label{wawa}
 \end{equation}       is  $[Z_{t_v}  ,t^v]$.
     Let  $D\in \CV$, $g^v\in  \BB^{v,\times}$.  Since   $[D,g^v]$ is an exceptional divisor,  $  \CZ \cdot [D,g^v]=0$ by   the definition of $\CZ$ (see Definition \ref{ZNU}).
Then 
      Definition \ref{admext} (1) for  $ \hat 1$ and $[D,g^v]$  implies   \begin{align} \begin{split}\sum_{\gamma  \in \tilde U^v(g^v)^{-1}\cap B^\times} ( Z_{1} +\sum_{i=1}^n  a_i  C_i  )\cdot \gamma D = (C_{\ord,1},1)\cdot [D,g^v],
   \label{ext1} \end{split}\end{align} 
   where the intersections on the left hand side happen on $\tilde \CS_{U_v}  $, and the intersection  on the right hand side happens on $\CN_{U,   \hat\CO_{H_w}^\ur}$.  
 
Let $\BA_E^{s,v,\times}$ be the  group of ideles away $|X|_s$ and $v$. Let $K_v\subset (U_E)_v $ be  the stabilizer of $C_i$'s. 
Let $\hat \CO_{E,s}$ be the   product of  complete local ring   of $\CO_E$ at split places.
           \begin{lem}\label{admj0}For   $t \in \BA_E^{s,v,\times}K_v\hat \CO_{E,s} \BA_F^\times  $, the  sum of the supersingular     components in  $\hat t $ is
            $$\sum_{i=1}^n  a_i[ C_i  ,t^v] .$$  
   \end{lem}
   \begin{proof} 
      We prove that     \begin{equation*}\hat t=\overline  t+\sum_{i=1}^n  a_i[ C_i  ,t^v] - (C_{\ord,1},\Nm (t^v)). \end{equation*}   
    We verify   Definition \ref{admext} (1) for  $ \hat t$ and supersingular components.  For ordinary components, the verification of (1)  is similar.  The verification of  Definition \ref{admext} (2) is similar and easier.      
 
 By the assumption  that $t_v\in (U_E)_v$, we have   $Z_{t_v}=Z_1$.
Thus we only need to  prove that for every   $[A,h^v]\in \CV^\sing$,   the equation    \begin{equation} \sum_{\gamma  \in t^v\tilde U^v(h^v)^{-1}\cap B^\times} ( Z_1+\sum_{i=1}^n  a_i C_i  )\cdot \gamma A =   (C_{\ord,1},\Nm (t^v))\cdot [A,h^v]  \label{ext2}  \end{equation} 
holds.   
     Note that each $U_x$, $x\in |X|-\{\infty\}$, has the form \eqref{orbint1K}, and $U_\infty$ is generated by $\varpi_\infty$ and a subgroup of the form \eqref{orbint1K} or is $\BB_\infty^\times$.
    By   Assumption  \ref{asmpe0} and a direct computation,   we have $t^v\tilde U^v =\tilde  U^vt^v$. 
      Thus \eqref{ext2} is equivalent to 
              \begin{equation} \sum_{\gamma  \in  \tilde U^v (h^v(t^v)^{-1})^{-1}\cap B^\times}  ( Z_1 +\sum_{i=1}^n  a_i C_i  )\cdot \gamma A =   (C_{\ord,1},\Nm (t^v))\cdot [A,h^v]  . \label{ext3}  \end{equation} 
                             Claim: $$(C_{\ord,1},\Nm (t^v)))\cdot [A,h^v ] =(C_{\ord,1},1))\cdot [A,h^vt^{v,-1} ].$$Then  \eqref{ext3}   is implied by   \eqref{ext1} by choosing $D=A$ and $g^v=h^vt^{v,-1}$. The lemma follows.  
           To prove the claim,   we use the following description of the restriction of $\CV^\ord$ to $   \hat \CN_U $. 
         The morphism \eqref   {Detuv}  induces a morphism $$\tilde \CS_{U_v}^0 \to \Spec \CO_{F_{U_v}} \hat\otimes_{\hat\CO_{F_v}^\ur}\hat \CO_{H_w}^\ur\cong \Spec  \hat \CO_{H_w}^\ur\times  \CO_{F_v}^\times/\det(U_v) .$$
         Here the last isomorphism is due to that $   F_{U_v}\subset \hat H_w^\ur$    which    comes from the  class field theory.
  Let  $\tilde \CS_{U_v}^{00}$ be the preimage  of  $\Spec  \hat \CO_{H_w}^\ur\times\{1\}$.
  By \cite[Appendice 8, Proposition]{Car}, the non-exceptional irreducible components of the special fiber of
    $\tilde \CS_{U_v}^{00}$    are indexed by $  B_v\bsl \GL_2(F_v)/U_v$ (this fact is similar to the second paragraph of \ref{special fibers}). Thus
  there is a vertical divisor $V$ of   $\tilde \CS_{U_v}^{00}$ which does not contain any exceptional curve such that 
 for every   $b\in  \BB^{v,\times}$, the restriction $ (C_{\ord,1},\Nm (b))$ to   $\hat \CN_{U_v} $ is $[V, b] $  under the isomorphism       
 \eqref{wawa}. Then the claim follows from a direct computation (again we use that  $t^v\tilde U^v =\tilde  U^vt^v$). 
   \end{proof}
   Lemma \ref{admj0} is improved as follows.
    \begin{lem}\label{admj}For   $t \in \BA_E^{ \times}  $, the  sum of the supersingular     components in  $\hat t $ is
            $$\sum_{i=1}^n  a_i[ t_vC_i  ,t^v] .$$  
   \end{lem}
   \begin{proof}Choose a subgroup $I $ of  the groups of ideles at split places, such that $ \Gal(H/E) $ is  the direct sum of image of  $\BA_E^{s,v,\times}K_v\hat \CO_{E,s} \BA_F^\times $  and the image of $I$ in $ \Gal(H/E) $ via the reciprocity map. 
Let $H'$ be the fixed subfield of the image of $I$. Define $\CN_{U}'$ in the same way as $\CN_{U}$, but with $H$ replaced by $H'$. Then 
$\CN_{U,\hat\CO_{H_w}^\ur}=\CN'_{U,\hat\CO_{H_w}^\ur}$. For $t$ as in  Lemma \ref{admj0}, define $\hat t'$ on $\CN_{U}'$ to be the admissible extension of $t$. Then $\hat t$ and $\hat t'$ have the same base change to 
$\CN_{U,\hat\CO_{H_w}^\ur}=\CN'_{U,\hat\CO_{H_w}^\ur}$, since $\CZ$ is defined over $\CM_U$.  
Note that Galois action keeps admissible $\CZ$-extensions, again since $\CZ$ is defined over $\CM_U$.  
Applying the   action of $I$ on $\CN_{U}'$ to    $\hat t' $  with $t$ as in  Lemma \ref{admj0}, the lemma follows.
   \end{proof}
   \begin{rmk} In \cite[8.5.1]{YZZ}, the computation of admissible extension is missed.
   \end{rmk}
  \subsubsection{Proof of Proposition \ref{nonsplitj}}\label{cadmissible extensions}  Let  $l=\sum  _{i=1}^na_i l_{C_i}$. 
By Lemma \ref{heightpullpack'}, we have
\begin   {align*}  j_v(t_1g,t_2)&=\sum_{\delta\in   B^\times} \sum_{i=1}^na_il_{t_{2,v}C_i}( \delta  , t_{1,v}  g_v )  1_{\tilde U^v}((t_2^{v}) ^{-1}  \delta t_1g^v )\\
&=\sum_{\delta\in   B^\times} l( t_{2,v}^{-1}\delta^{-1} t_{1,v}  ,g_v )  1_{\tilde U^v}((t_2^{v}) ^{-1}  \delta^{-1} t_1g^v ).
\end{align*}
       
Similar to  Proposition \ref{regwell}, we have the following expression of $j(f)_v$.
Let   \begin{equation}  j(\delta,f_v):=\int_{ E^\times_v/F_v^\times}\int_{ E^\times_v} \sum _{g_v\in  \BB_v^\times/U_v}f_{v}(g_v)   l(t_{2,v}^{-1} \delta^{-1} t_{1,v}  ,g_v ) 
       \Omega_v^{-1}( t_{2,v})\Omega_v(t_{1,v})   dt_{2,v}d t_{1,v}\label{1129}\end{equation}   
which is well-defined by  Lemma \ref{9123} (1). 
\begin{prop} \label{regwell1}
     Under Assumption \ref{fvan} and Assumption \ref{asmpe0},  we have
 \begin{equation*} j(f)_v=  \sum_{\delta\in E^\times\bsl B_\reg  ^\times/E^\times} \CO _{\Xi_\infty}(\delta,f^v)      j(\delta,f_v).  \end{equation*}

    \end{prop}
    
Define a function $\overline {f_v}$ on $B_v^\times$ by \begin{equation}\overline {f_v}(h):=
    \sum _{g_v\in  \BB_v^\times/U_v}f_{v}(g_v)   l(  h^{-1}   ,g_v ) .\label{barf}\end{equation}
   By  Lemma \ref{9123} (1), $\overline {f_v}\in C_c^\infty(B_v^\times)$.
By   \eqref{1129}, $j(\delta ,f_v)=\CO (\delta ,\overline {f_v})  .$
              This finishes the proof of Proposition \ref{nonsplitj}. 
\subsection{ Superspecial case}\label{Superspecial case} Let $v\in\Ram-\{\infty\}$. Then $\BB_v$ is a division algebra.  
 In particular,   $v$ is not split in $E$. Let $B=B(v)$ be the $v$-nearby quaternion algebra of $\BB.$ Then $B_v\cong \RM_{2,F_v}$.   Let $n$ be the level of  the  principal congruence subgroup $U_v\subset \BB_v^\times$.
    \subsubsection{Formal models of   $\CN_{U ,\CO_{H_w}}$ and  the multiplicity function} 
    Let $\Omega_v$  be  rigid analytic Drinfeld's upper half plane over $ {F_v}$,  $\hat \Omega_v$  be Deligne's formal model of $\Omega_v$   over $\CO_{F_v}$. Then  $\hat \Omega_v \hat \otimes \hat \CO_{F_v}^\ur $ is  the  deformation  space of special height 4 formal $\CO_{F_v}$-modules  (see \cite{Dridomain}).
  Let $\Sigma_n$ be the $n$-th covering of  $  \Omega_v \hat \otimes_{F_v}   \hat F_v ^\ur $. 
Then   $\Sigma_n$ admits a natural $B_v^\times\times\BB_v^\times$-action (see \cite{Dridomain}). 
   \begin{prop} \label{CDunif'}    There is an isomorphism of rigid analytic spaces over $F_v$:
 $$   M_{U}    ^\an \cong B ^\times\bsl  \Sigma_n  \times \BB^{v,\times}/\tilde U ^v.$$ \end{prop}
 \begin{proof}When $U_\infty=\BB_\infty^\times$, this is proved in \cite[Theorem 8.3]{Hau}.
 The general case can be obtained by applying   \cite[Proposition 4.28]{Spi}  to Proposition  \ref{riguni}.
\end{proof}

 Let $\hat \Sigma_n$ be minimal desingularization the normalization of $ \hat \Omega_v \hat \otimes \hat \CO_{F_v}^\ur $ in the  
  rigid analytic space $\Sigma_n \hat \otimes_{   F_v } H_w $.   
Let      $\hat \CN_{U  }$ be the formal completion   of $  \CN_{U,\CO_{H_w} }$  along its special fiber.

   \begin{cor}  \label{CDunif''}    
   There is an isomorphism of formal schemes over $\CO_{H_w}$:
 $$   \hat\CN_{U}     \cong B ^\times\bsl  \hat \Sigma_n  \times \BB^{v,\times}/\tilde U ^v.$$ \end{cor}  

 Now we consider  $\CO_{\overline v}$-points.
 Define
   \begin{equation*}\CH_{U_v}:=   B_v^\times\times_{E_v^\times}\BB_v^\times/U_v   \end{equation*} 
and let  $B_v^\times$ acts on $\CH_{U_v}$ by left multiplication. 
 We define a  $B_v^\times$-equivariant map  $\CH_{U_v}\incl  \hat \Sigma_n(\CO_{\overline v}) $ 
  as follows. Here  $\hat \Sigma_n$ is regarded as a formal scheme over $\CO_{H_w}$.
    Let  $\hat P_0 \in  \hat\CN_{U}(\CO_{\overline v})$   be the image of the Zariski closure of $P_0$ (see \ref{CM points}).
Then up to the choice of the data defining  the isomorphism in Corollary \ref{CDunif''},  
there exists $  \hat z_0 \in  \hat \Sigma_n(\CO_{\overline v})   $, fixed by the image of the diagonal embedding 
   $E_v^\times\incl B_v^\times\times \BB_v^\times$,
 such that under this isomorphism
 $\hat P_0=[\hat z_0,1]  .$
 Then define $\CH_{U_v}\incl  \hat \Sigma_n(\CO_{\overline v}) $  
 by $(g_1,g_2)\mapsto (g_1,g_2)\cdot \hat z_0. $
  
   \begin{defn} \label{mf2}Define the multiplicity function $m_v$ on $ \CH_{U_v}-\{(1,1)\}$ as follows: for $(g_1 ,g_2 )\in \CH_{U_v}  $ 
 and $(g_1,g_2)\neq (1,1)$, let  $m_v(g_1 ,g_2 )$   
 to be the  intersection number of the  images of the
  points $(g_1,g_2)$ and $(1,1)$ in $ \hat \Sigma_n (\CO_{\overline v}) $.  
  
 \end{defn}
 Similar to Lemma \ref{vdetun1}, we  have the following nonvanishing condition on $m_v$  by the determinant  construction in \cite[IV]{Gen}.   \begin{lem}\label{vdetun"}The multiplicty function $m_v(g_1,g_2)\neq 0$ only if $ \det(g_1)\det(g_2) \in \det(U_v)$.
  \end{lem}

Let $\hat z_0'\in \hat  \Omega_v (\CO_{\overline v})$  be  the image of $\hat z_0$ by the  composition of  $\hat\Sigma_n  \to  \hat \Omega_v \hat \otimes \hat \CO_{F_v}^\ur  \to  \hat  \Omega_v  .$
It is the base change of a point $ \hat z_0'\in \hat  \Omega_v (\hat \CO_{E_v}^\ur).$ 
  Then the composition of
    \begin{align}\CH_{U_v}\incl\hat \Sigma_n (\CO_{\overline v})\to\hat \Omega_v (\hat \CO_{E_v}^\ur) \label{929}
    \end{align}  is   given by 
 $(g_1,g_2)\mapsto g_1\hat z_0'. $
  
      \subsubsection{Special fibers of $\hat \Omega_v$   } \label{stabilizer}
      
         The irreducible components of the special fiber of $\hat\Omega_v  \hat \otimes \hat \CO_{F_v}^\ur$, now regarded as a formal scheme over $ \hat \CO_{F_v}^\ur$, are $\BP^1_{\overline {k(v)}}$'s
         parametrized by the set of dilation classes  of $\CO_{F_v}$-lattices  of $F_v^2$. Moreover, $B_v^\times$ acts on the special fibers in the same way as
it acts on the lattices via an isomorphism     $B_v^\times\cong \GL_2(F_v)$.    Assume  that 
    $E_v/F_v$ is unramified. 
    Then 
the reduction of 
  of $\hat z_0'$  is a smooth point of the special fiber of $\hat\Omega_v  \hat \otimes \hat \CO_{F_v}^\ur$, so is only in one irreducible component. Let $K_v$ be the  maximal compact subgroup of $B_v^\times$ such that $F_v^\times K_v$ is the stabilizer of this irreducible component.

Each two  $\BP^1_{\overline {k(v)}}$'s in  the special fiber of $\hat\Omega_v  \hat \otimes \hat \CO_{F_v}^\ur$ 
 intersect at an ordinary double point.  The double points in  the special fiber of $\hat\Omega_v  \hat \otimes \hat \CO_{F_v}^\ur$  one-to-one correspond to \textit{non-ordered} 
pairs  of dilation classes of adjacent lattices.  
 Assume that  $E_v/F_v$ is ramified. Then    $$\hat z_0'\in  \hat \Omega_v( \hat \CO_{E_v}^\ur)-\hat \Omega_v( \hat \CO_{F_v}^\ur).$$
  In fact,  the generic fiber of $\hat z_0'$ in $ \Omega_v(  \hat E_v^\ur)$ is a fixed point of $E_v^\times\incl B_v^\times$.
  A direct computation on Drinfeld's upper half plane shows that the generic fiber of $\hat z_0'$ is not defined over  $\hat F_v^\ur$.
  Thus   the reduction of 
  $\hat z_0' $ is a double point.  Let  $S_v$ be the stabilizer of this double point.
Let  $s\in B_v^\times $ which 
  switches the two dilation classes of adjacent lattices (such $s$ exists and is unique up to $F_v^\times$). Let $K_v\subset B_v^\times$  be the a maximal compact subgroup   such that $F_v^\times K_v$ is the stabilizer   of $[L_0]$ or $[L_0']$.  Then the group $S_v$ is generated by $F_v^\times,s$ and  $K_v$. 

    \begin{lem}\label{vdetun'}
 We have the following necessary conditions for  $m(g_1,g_2)\neq 0$:  
      \begin{itemize} 
  \item[(a)]    when $E_v/F_v$ is unramified, then   $g_1\in F_v^\times K_v$;
     \item[(b)] when $E_v/F_v$ is ramified, then   $g_1\in S_v.$
  \end{itemize}

  \end{lem}
  \begin{proof}
Assume $n=0$ and  $m(g_1,g_2)\neq 0$.
Then the   images of 
$(g_1,g_2)$ and $(1,1)$ under  
$\CH_{U_v}\incl\hat \Sigma_0 (\CO_{\overline v})  
$
have the same reduction.
Recall that  $\hat z_0', g_1 \hat z_0'$ are the images of  $(1,1)$ and $(g_1,g_2)$ in $ \hat \Omega_v( \hat \CO_{F_v}^\ur)$
under the composition of \eqref{929}.    Thus  the  reductions  of  
 $\hat z_0', g_1 \hat z_0'$   are the same   point. Assume that $E_v/F_v$ is unramified, then 
   the  reductions  of $\hat z_0', g_1 \hat z_0'$
are in the same irreducible component of the special fiber of $ \hat\Omega_v  \hat \otimes \hat \CO_{F_v}^\ur $. This gives condition (a). 
Now let $E_v/F_v$ be ramified. 
Then  the reductions  of  
 $\hat z_0', g_1 \hat z_0'$   are the same double point.
 This gives condition (2).  \end{proof}

   \subsubsection{Compute $i(f)_v$}\label{9.2.3}
By Corollary \ref{CDunif''},
    we can express $i_{\overline v}$ by $m_v$ as in Lemma \ref{heightpullpack}.   
  \begin{lem} Let $x,y\in  CM_U $
be \textit{distinct} CM points. Assume  $x=t_1g,y=t_2$ for $t_1,t_2\in E^\times(\BA_{F })$ and $g\in \BB ^\times$, then 
$$i_{\overline v}(x,y)=\sum_{\delta\in   B^\times} m_v(t_{1,v}^{-1}\delta t_{2,v},g_v ^{-1})  1_{\tilde U^v}(((t_1 g )^{-1}  \delta t_2 )^v).$$ Moreover, this is a finite sum.  \end{lem}  
The proof is the same as the one of  Lemma \ref{heightpullpack}, except that here we use Lemma  \ref{vdetun"}  and \ref{vdetun'}   to  show that   the support of $m_v(\cdot,g_2)$ for fixed $g_2$ is contained in a compact subset of $B_v^\times$.

     \begin{prop}\label{regwell'}  
  Under Assumption \ref{fvan},  we have
 \begin{equation*} i(f)_v= \sum_{\delta\in E^\times\bsl B_\reg  ^\times/E^\times} \CO_{\Xi_\infty} (\delta,f^v)  i(\delta,f_v). 
 \end{equation*}
    \end{prop}
    Here $ i(\delta,f_v)$ is defined as in {iterm2}.
     The proof is the same  as the one of   Proposition \ref{regwell}, except that here we use the condition $n>0$,
     Lemma  \ref{vdetun"}  and \ref{vdetun'}
     to show  that  $$\inv_{E^\times}'(\{g_1\in (B_v^\times)_\reg :m_v(g_1   ,g_2  )\neq0\}) $$
is contained in a   compact subset of $F_v$   for a fixed $g_2$.
      \subsubsection{Compute the multiplicity function $m_v$} \label{9.2.4}

 By  Lemma \ref{vdetun"} , $m(g_1,1)$ is supported on the open subgroup $ B_{v,0}:=\{g_1\in B_v^\times: \det (g_1)\in \det U_v\}$ of $B_v^\times $ (and not defined on $  E_v^\times$). 
   Similar to Lemma \ref{U'}, we have the following result.
   \begin{lem}\label{U''}There is an open compact subgroup $U'\subset B_{v,0}$ such that
   \begin{itemize}
  \item[(1)]  as subgroups of $E_v^\times$, $U'\cap E_v^\times=U_v\cap  E_v^\times$;
  \item[(2)]    the function $$m (g_1,1     ) - \frac{v (\inv_{E_v^\times}'(g_1)) }{2} 1_{U'}(g_1) $$
  on  $ B_{v,0} -   E_v^\times$  
  can be extended to a   locally constant function      on $B_{v,0} $.   \end{itemize}
   \end{lem} 
  
                  \subsubsection{Compute $j(f)_v$}   \label{9.2.5} 
                  
                 The following proposition  is  the same as Proposition \ref{nonsplitj},
                   \begin{prop}\label{nonsplitj'} There exists $\overline {f_v}\in C_c^\infty(B_v^\times)$ such that 
  $$j(f)_v= \sum_{\delta\in E^\times\bsl B_\reg^\times/E^\times}\CO _{\Xi_\infty}(\delta,f^v)   
\CO (\delta ,\overline {f_v})    .$$        \end{prop}

  Define  a set  $\CV$  of irreducible components of  the special fiber of $\hat \Sigma_n\hat\otimes_{\CO_{F_v}}\hat \CO_{F_v}^\ur$  as follows. Consider the morphism $\hat \Sigma_n\to \hat\Omega_v\hat\otimes\hat \CO_{F_v}^\ur $ of formal schemes over
  $\CO_{F_v}$. Its base change to  $\hat \CO_{F_v}^\ur$ is 
 \begin{equation}\hat \Sigma_n\hat\otimes_{\CO_{F_v}}\hat \CO_{F_v}^\ur\to \hat\Omega_v\hat\otimes\hat \CO_{F_v}^\ur\hat\otimes_{\CO_{F_v}}\hat \CO_{F_v}^\ur\cong
  \hat\Omega_v\hat\otimes\hat \CO_{F_v}^\ur\times\hat \BZ. \label{1134}\end{equation}
  Here $\hat \BZ$ is the profinite completion of $\BZ$, and we use the canonical isomorphism $\Gal (\hat  F_v^\ur/F_v)\cong \hat \BZ$ which maps the Frobenius map to $1$.  Then the Galois action of $\Gal (\hat  F_v^\ur/F_v)$  is given by the addition on $\hat \BZ$.

  If $E_v/F_v$ is unramified and $n=0$, then let $\CV$ be the set of  all irreducible  components whose images   via \eqref{1134}  are not points.
 Otherwise, let $\CV$ be the set of all irreducible  components whose images   are double points.
For $C\in  \CV$, let $S_C\subset B_v^\times$ be the   stabilizer of the corresponding irreducible  component or double point.
Let  $n_C$ be the image of $C$  in the factor $\hat \BZ$ in
   the last term of \eqref{1134}.
Let         $$ \CV^\ssp =B^\times\bsl   \CV \times \BB^{v,\times}/\tilde U^v,$$
which is a subset of the set of irreducible components of $\CN_{\overline {k(w)}}$.
For $C\in \CV$, $g\in\BB^{v,\times}$, let $[C,g]$ be the corresponding element in $\CV^\ssp.$

 \begin{defn} For $C\in  \CV$,  define a function $l_{ C}$  on $ \CH_{U_v}$ 
  as follows. For $(g_1,g_2)\in   \CH_{U_v}$, let $l_{ C}(g_1,g_2)$ be the  intersection number  of $C$ and the  image  of the 
  point $(g_1,g_2)$   in $\hat\Sigma_n(\CO_{\overline v})$.  
  
  \end{defn}
 Let $\hat z_0$ be as above Definition \ref{mf2}. Let $n_{\hat z_0} $ be the image of ${\hat z_0} $ in the factor $\hat \BZ$ in
   the last term of \eqref{1134}.
 Similar to  Lemma \ref{9123}, we have the following lemma.
  \begin{lem}\label{lcq} The function $l_C$ satisfies the following properties:
  \begin{itemize}
 
  \item[(1)]    $l_{C}(g_1,g_2)\neq 0$ only if    
 $ v(\det (g_1)\det (g_2))=n_C-n_{\hat z_0} $ and $g_1  \in S_C$;
     
 \item[(2)]  for $b\in B_v^\times $, $h \in \CH $, $l_{bC}(bh)=l_C(h)$; 
  
 \item[(3)]  $l_C$    is locally constant.
\end{itemize}

     \end{lem}
   \begin{proof} By \cite[\S 2 Theorem]{Dridomain}, the action of  $B_v^\times\times \BB_v^\times$ on the factor $\hat \BZ$ in
   the last term of \eqref{1134} is  the addition by $ v(\det (g_1)\det (g_2))$. Then
the   first part of (1) follows. The proof of the second part of (1) is  similar to the proof of  Lemma \ref{vdetun'}. The  proof of (2)(3) are similar to  the proof of Lemma \ref{9123} (2)(3).    \end{proof}
     Similar to Lemma \ref{heightpullpack'}, we have the following lemma.
   \begin{lem}\label{11.3.10}
Let $g \in( \BB ^{v})^{\times}$ and $y\in CM_U   \cong  B^\times\bsl B^\times\times_{E^\times}\BB_v^\times/U_v\times \BB^{v,\times}/\tilde U^v$, the intersection of $(C,g) $ and the Zariski closure of $y $ in $\CN_U \otimes _{\CO_{H_w}}\CO_{\overline v}$    is given by 
  $$\sum_{\delta\in   B^\times} l_C ( \delta^{-1} ,y_v)1_{\tilde U^v}(g^{-1}\delta^{-1} y^v).$$ 
     Moreover, this is a finite sum.

  \end{lem}

        Let  $\sum_{i=1}^n a_i[C_i,1] ,$ where $ C_i\in     \CV  $, be the sum of the  vertical components  of the $\CZ $-admissible extension of $1\in CM_U$ which are  contained in $\CV^\ssp$.    
     \begin{lem}\label{admj'} For   $t \in   E^\times\bsl  \BA_E^\times /\tilde U_E $, $\sum_{i=1}^n  a_i[t_vC_i  ,t^v] $ is the    vertical part  of the $\CZ $-admissible extension of $t\in CM_U$ which are  contained in $\CV^\ssp$.    \end{lem} 
   The proof of Lemma \ref{admj'} is similar to the one of Lemma \ref{admj}. 
    Another more conceptual proof is to use the $\BB^\times$-action (this proof is not available for Lemma \ref{admj} since   there the restriction of $\CZ$ on the local integral model is not compatible with the  $\BB^\times$-action). 
     
Similar to  Proposition \ref{regwell1}, we have the following expression of $j(f)_v$.
\begin{prop}  
    We have
 \begin{equation*} j(f)_v= \sum_{\delta\in E^\times\bsl B_\reg  ^\times/E^\times} \CO _{\Xi_\infty}(\delta,f^v)      j(\delta,f_v),  \end{equation*}
  where $   j(\delta,f_v)$ is defined by the same formula as \eqref {1129}
    \end{prop} 
    
    Define a function $\overline {f_v}$ on $(B_v^\times)$ by the same formula as \eqref{barf}. By  Lemma \ref{lcq} (1), $\overline {f_v}\in C_c^\infty(B_v^\times)$. 
    Then Proposition \ref{nonsplitj'} follows. 
    \subsection{Ordinary case}   \label{Ordinary case}
     Let $v\in |X|_s$  be split in $E$.  
Let $i(f)_{v_n} $   and $j(f)_{v_n}$, $n=1,2$, be as in \ref{Decomposition of the height distribution}.
\begin{prop}\label{splitI}
    Under Assumption \ref{fvan},    $i(f)_{v_n}  =0$.
          \end{prop}
    \begin{prop}\label{splitj}
    Under Assumption \ref{fvan},    $j(f)_{v_n}  =0$.
          \end{prop}
        Let $w_n$  be the restrictions of $  \overline  {v_n}$ to  $H$.  
The special fiber of  $\CN_{U,\CO_{H_{w_n}}}$
is  described  in \ref{special fibers}.  Since  $v$ is split in $E$, the reductions of CM points in $\CN_{U,\CO_{H_{w_n}}}$ are ordinary points.   Let $Y_v=N\bsl \BB_v^{\times}/ U_v$ 
where $N$ is a unipotent subgroup  of  $\BB_v^{\times}$ and $Y^v= \BB^{v,\times}/\tilde U^v$.
  Then  ordinary points  in $\CN_{U, \overline {  k(v)}} $  are parametrized   by $E^\times \bsl Y_v\times Y^v$.
  Indeed, this is a special case of   the discussion in \cite[10.3, 10.4]{LLS}.
   The reduction map from $CM_U\cong E^\times\bsl \BB^\times/ \tilde U$ to the set of   ordinary points  (see Lemma \ref{lieord}) is induced by the natural map \begin{equation} \BB_v^{\times}/ U_v\to Y_v \label{BYv}.\end{equation}
 \subsubsection{Compute $i(f)_v$}
 Let $\hat\CN_{U }$ be the formal completion of  $\CN_{U,\hat\CO_{H_{w_n}}^\ur}$ along the special fiber.
For   $y\in Y_v$, let   $\cD_y$ be the formal completion of $\CN_{U,\hat\CO_{H_{w_n}}^\ur}$ at $[y,1]\in E^\times \bsl Y_v\times Y^v$. 
For  $g_v\in \BB_v^{\times}/ U_v$,  let $D_{g_v} $ be the image of 
$g_v1^v\in CM_U$ in $\hat\CN_{U }(\CO_{\overline  {v_n}}).$
If $g_v$ has image $y$ in $Y_v$ via \eqref{BYv}, then  $D_{g_v}\in \cD_y(\CO_{\bar v})$.
Thus we have a map \begin{equation*}\BB_v/U_v\to  \coprod _{y\in Y_c}\cD_y(\CO_{\overline  {v_n}}). \end{equation*}
  
 \begin{defn} Define the multiplicity function $m_{  {v_n}}$ on $\BB_v/U_v\times \BB_v/U_v$ as follows: for $(g_1 ,g_2 )\in \BB_v/U_v\times \BB_v/U_v  $,  let  $m_{  {v_n}}(g_1 ,g_2 )$   
  be the  intersection number  of the  image of the
  points $g_1$ and $g_2$ in $  \coprod _{y\in Y_c}\cD_y(\CO_{\overline  {v_n}})$.  
  
 \end{defn}

    Similar to Lemma  \ref{heightpullpack}, we have the following lemma.
    \begin{lem}\label{heightpullpacksplit}  Let $x,y\in  CM_U:=E^\times\bsl \BB^\times/\tilde U$
represent  two distinct CM points, then  
  $$i_{\overline  {v_n}}(x,y)= \sum_{\delta\in   E^\times} m_{  {v_n}}(x_v,\delta y_v)  1_{\tilde U^v}( (x^v)^{-1}\delta y^v) .$$
   
 Moreover, this is a finite sum. 
 
\end{lem}

    Similar to Lemma \ref{innersum}, we have the following lemma.
    
    \begin{lem} \label{innersum'} Let $V$ be an open  compact subgroup of  $\BA_{F,\mathrm{f}}^\times$, 
     $\phi$ be a  function on $  \BA_E^\times  $ which is  $V$-invariant and $\Xi_\infty$-invariant. Then  if either side of the   equation   \begin{align*} & \Vol(\tilde U/\Xi_\infty)  \int_{E^\times\bsl \BA_E^\times  / \BA_F^\times}  \sum_{t\in F^\times\bsl \BA_F^\times/\Xi_\infty U}\sum_{x\in   E^\times } \phi( xtt_2)d{t_2}  \\
  = &\int_{ E^\times(\BA_{F,\mathrm{f}}) }\int_{ E^\times(F_\infty)/\Xi_\infty } \phi(   t_{2,f} t_{2,\infty})d{t_{2,\infty}}d{t_{2,f}} 
  \end{align*}
 converges  absolutely, the other side also converges absolutely,
 and in this case the   equation holds.

 \end{lem}

 Similarly to \eqref{1118},
  we have   \begin{align*}i(f)_{v_n}      =  &\int_{ E^\times\bsl \BA_E^\times / \BA_F^\times}\int_{ \BA_E^\times /\Xi_\infty }  f( (t_{1} ^{v})^{-1}   t_{2}^v  ) \\
& \sum _{g_v\in  \BB_v^\times/U_v}f(g_v)m_{  {v_n}}(t_{2,v}  ,    t_{1,v}  g_v  ) d{t_{2,\infty}}\Omega^{-1}(t_2)\Omega(t_1)d{t_{2}} d{t_1},
     \end{align*}
    and the right hand side is absolutely convergent. Since $v\not \in S_{s,\reg,i}$ for $i=1$ or 2, by Assumption \ref{fvan} (1),   $f( (t_{1} ^{v})^{-1}   t_{2}^v  )=0$.
Thus Proposition \ref{splitI} follows.          
          
   \subsubsection{Compute $j(f)_v$}\label{11.4.3}
 
By Proposition  \ref{123}, we have     \begin{align}\CV ^\ord\cong F ^\times\bsl \BA_F^\times/\det (\tilde U)\times B_v\bsl \GL_2(F_v)/U_v .\label{1234}\end{align}   
                             Let $\BB^{v,\times}$ act on $ \CV ^\ord$  via  $\det$ and multiplication on the first component.

                   \begin{defn} For $C\in  \CV^\ord$,  define a function $l_{ C,n}$  on $CM_U$ 
  as follows. For $g\in   CM_U $,  
  let $l_{ C,n}(g)$ be the  intersection number   of $C$ and the Zariski closure of $g$ in $  \CN_{U,{\CO_{\overline  {v_n}}}}$.  
  
  \end{defn}            
                            \begin{lem}\label{ord1}   
  (1) Let   $v'\neq v $, $g\in \BB_{v'}^\times$ such that $\det g\in\det U_{v'}$.  
   Then $l_{ C,n}=l_{g\cdot C,n}$.

(2)  Let  $v'\in |X|-\Ram-\{v\}$ and $g\in \BB_{v'}^\times$. 
Then $$l_{C,n}(Z(g)_{U,*}t)=|\tilde U g \tilde U/ \tilde U| l_{g^{-1} C ,n} (t).$$

  \end{lem}

  \begin{proof}(1) follows from definition.
Now we prove  (2).    
  Let $V=U\cap  gUg^{-1}$, $V'=  g^{-1}Ug\cap U$. 
  Let $L$ be a finite extension of $H$ over which all geometrically irreducible components of $M_{V}$ (and $M_{V'}$)  are defined.
  Let $u_n$ be the restriction of $\overline{v_n}$ to $L$.
  Define $  \CN_{V,{\CO_{L_{u_n}}}}$ in the same way that $  \CN_{U,{\CO_{H_{w_n}}}}$  is defined.
Then $  \CN_{V,{\CO_{L_{u_n}}}}$ is the minimal desingularization of the normalization of $  \CN_{U,{\CO_{H_{w_n}}}}$, also
$\CM_{V,{\CO_{F_v}}}$, in the function field of $M_{V,L_n}$. 
Let  $ \pi_g$ be 
the natural morphism from  $\CN_{V,{\CO_{H_{w_n}}}}$ to $ \CN_{U,{\CO_{H_{w_n}}}}$.
By  \ref{mio},  $\pi_g$ finite \etale away from supersingular components and cusps.
Similar conclusions hold for  $  \CN_{V',{\CO_{L_{u_n}}}}$ and the natural morphism $ \pi_1$ from $ \CN_{V',{\CO_{H_{w_n}}}}$ to $\CN_{U,{\CO_{H_{w_n}}}}.$      Let $T_g:\CN_{V,{\CO_{H_{w_n}}}}\to\CN_{V',{\CO_{H_{w_n}}}}$ be the natural extension of  the  isomorphism $T_g$ between the generic fibers.  
Then the restriction of $ \pi_{1,*}T_{g,*} \pi_g^* $ to the generic fiber is    $cZ(g)_{U,*}$  for a constant 
$c$.
Note that  the Zariski closure  $\overline t$ of $t$ in $  \CN_{U,{\CO_{H_{w_n}}}}$ has reduction   outside the supersingular components.  
 Thus 
   \begin{equation*}l_{C,n}(Z(g)_{U,*}t)=c\pi_{1,*}T_{g,*} \pi_g^*\overline  { t} \cdot C =c\overline  { t} \cdot \pi_{g,*}T_{g}^*\pi_1^* C .\end{equation*}
By the formation of \eqref{1234} (see the discussion above Proposition \ref{123}), it is easy to check that the support of $ \pi_{g,*}T_{g}^*\pi_1^* C $ is $g^{-1} C.$  
 Since $\pi_1$ and  $\pi_g$ are finite \'{e}tale and the correspondence $Z(g^{-1})$ of the generic fibers has degree $|\tilde U g \tilde U/ \tilde U|$, we have 
 $$  \pi_{g,*}T_{g}^*\pi_1^* C =|\tilde U g \tilde U/ \tilde U| g^{-1} C.$$  
Thus (2) follows.  \end{proof}

                     \begin{prop}\label{splitj??}
                    Let  $v'\in S_{s,\ave}$ such that $v'\neq v$.   
                    Then $$ \sum_{g \in U_{v'}\bsl  \BB_{v'}^{\times} /U_{v'}}f_{v'}(g )l_{C,n}(Z(g )_{U,*}t)=0.$$
          \end{prop}   
          \begin{proof}  
          Let $h_{-m}$ be a fixed element in $ \BB_{v'}^{-m}$.
         By  Lemma \ref{ord1}, for every $m\in \BZ$, we have  \begin{align*} \sum_{g \in U_{v'}\bsl  \BB_{v'}^{m} /U_{v'}}f_{v'}(g )l_{C,n}(Z(g )_{U,*}t)=
      \left( \sum_{g \in  \BB_{v'}^m /U_{v'}} f_{v'}(g )\right)  l_{h_{-m}C    ,n}(t)
                 .\end{align*}       
Then inner sum of the right hand side   is 0 by Assumption \ref{fvan} (2).
The proposition follows.  \end{proof}
 From the definition of $j(f)_{v_n}$,  Proposition \ref{splitj} is implied by Proposition \ref{splitj??}.
      \subsection{The $\infty$ place}\label{theinf}
    
          \subsubsection{Compute hyperbolic distances}
         Extend  $|\cdot|_\infty$ to $\BC_\infty$ and denote this extension by $|\cdot|$. 
  For $z\in \Omega_\infty(\BC_\infty)$, let $|z|_i:=\inf_{a\in F} |z-a|$ be the ``imaginary part" of $z$.  
   The ``hyperbolic"  distance between $z_1,z_2\in \Omega$  is defined to be $$d(z_1,z_2):=\frac{|z_1-z_2|^2}{|z_1|_i|z_2|_i}.$$

 The following lemma is easy to check.
 \begin{lem}\label{152}
 (1) For $\delta =\begin{bmatrix}a&b\\ c&d\end{bmatrix} \in \GL_2(F_\infty)$, we have
 $|\delta z|_i= |\det \delta||z|_i/|cz+d|^2$
 
 (2) The ``hyperbolic"  distance is $ \GL_2(F_\infty)$-invariant, i.e. for every $g\in \GL_2(F_\infty)$ and $z_1,z_2\in \Omega$,   we have $d(z_1,z_2)=d(gz_1,gz_2).$

\end{lem}

Let $z$ be a fixed point  of  an embedding $E_\infty^\times\subset \GL_2(F_\infty)$.   
\begin{lem}\label{minf}
Assume that $p>2$, or assume that  $p=2$ and $E_\infty/F_\infty$ is unramified, then for   $\delta \in \GL_2(F_\infty)$, the
following equation  holds:  \begin{align}d(z,\delta z)=|\inv_{E_\infty^\times} ' (\delta)|.\label{1111}\end{align}

If $p\neq 2$ and $E_\infty/F_\infty$ is a ramified extension,  let $\varpi_\infty'$ 
be a uniformizer of $E_\infty$,  then for   $\delta \in \GL_2(F_\infty)$, the
following equation  holds:   \begin{align}d(z,\delta z)=|\varpi_\infty'|^2 |\inv_{E_\infty^\times} ' (\delta)|.\label{1112}\end{align}
  
    \end{lem}\begin{proof}
    Easy to see that  the truth of this lemma does not depend on the embedding. 
  By    Lemma \ref{152} (2) 
   we only need to prove
  the lemma for  one representative in the  $E_\infty^\times\times E_\infty^\times$-orbit of $\delta$.    Thus we   choose a representative $\delta=a+bj$ (the notation is as in \ref  {matchorb}) 
  with $a=1$ or $b=1$ to simplify the computation.      Applying  Lemma \ref{152}  (1), a direct computation gives the lemma.
\end{proof}
    
    \begin{defn} 
    For $\delta \in D^\times-E^\times$, let  
   $$ m_\infty(\delta)=-\frac{\log_{q_\infty} d(z,\delta z)}{2}.$$
    \end{defn}
  \begin{prop}\label{heightpullpackinf}     Suppose  $U_\infty=\BB_\infty^\times. $ Let $x,y\in  CM_U\cong E^\times\bsl \BB_{\mathrm{f}}^\times/U$
representing two distinct CM points, then  
  $$i_{\bar \infty }(x,y) = \sum_{\delta\in   D^\times -E^\times,\ d(z_0,\delta z_0)<1} m_\infty (\delta)1_U(x^{-1}\delta y ).$$
   Moreover,  this is a finite sum.
 
\end{prop}
 \begin{proof} Use \cite[Proposition 4]{Tipp}
 and the argument in \cite[8.1.1]{YZZ}.
 Note that $x,y$ are defined over the maximal unramified extension of  $E_v$ by the condition $U_\infty=\BB_\infty^\times$.
 Thus the   ramification index involved in \cite[Proposition 4]{Tipp} is cancelled by the normalization in the definition of    $i_{\bar \infty }$.
\end{proof}

  \subsubsection{Compute $i(f)_\infty$ and $j(f)_\infty$}\label{ijinfty}   

 We compute  $i(f)_\infty$ first. The results below can be proved via mild modifications of the computations in \ref{9.2.3},  \ref{9.2.4}.   Let $f_\infty=1_{U_\infty}$ for the simplicity of notations.
Regard $U_\infty\cap  E_\infty^\times$ as a subgroup of $D_\infty^\times$ via the embedding $E_\infty^\times\incl D_\infty^\times$.
  \begin{defn}\label{iterm3}For $\delta\in D^\times_\reg$ and a function $m_\infty$ on $D_\infty^\times- U_\infty\cap  E_\infty^\times$ invariant by $\Xi_\infty$, define the   arithmetic orbital integral of  the   function $m_\infty$ weighted  by  $f_\infty$ to be 
        $$ i (\delta,f_\infty)  :=  \int_{ E^\times_\infty/F_\infty^\times}\int_{ E^\times_\infty/\Xi_\infty}   m_\infty( t_{1,\infty} ^{-1}  \delta t_{2,\infty}   ) 
       \Omega_\infty^{-1}( t_{2,\infty})\Omega_\infty(t_{1,\infty})   dt_{2,\infty}d t_{1,\infty}  .  $$   \end{defn}
Let  $D_{\infty,0}:=\{g_1\in D_\infty^\times: \det (g_1)\in \det U_\infty\}$.
   \begin{prop}\label{regwellinf}Under Assumption \ref{fvan},  we have
    $$i  (f)_\infty =
         \sum_{\delta\in E^\times\bsl D^\times_{\reg  }/E^\times}  \CO( \delta,f ^\infty) i (\delta,f_\infty),$$           
         where  $i (\delta,f_\infty)$ is  the   arithmetic orbital integral weighted  by  $f_\infty$ of   a  function $m_\infty$ which satisfies the following properties:
       
   \begin{itemize}
   \item[(1)] the support of $m_\infty$ is contained a compact modulo center subgroup of  $D_{\infty,0}$;   
  \item[(2)]   there is an open   subgroup $U'\subset D_{\infty,0}$ which is compact modulo center such that   $$U'\cap E_\infty^\times=U_\infty\cap  E_\infty^\times$$ as subgroups of $E_\infty^\times$;
  \item[(3)]         the function $$m_\infty (g  ) - \frac{v_{E_\infty}(\inv_{E_\infty^\times}'(g)) }{2} 1_{U'}(g)$$  
   on $ D_{\infty,0} -    E_\infty^\times$
  can be extended to a   locally constant function      on  $D_{\infty,0}$.
   \end{itemize}
 In particular,  $i (\delta,f_\infty)$ is a convergent integral, and $i  (f)_\infty$ is a finite sum.

             \end{prop}

For  $j(f)_\infty$, we have the following result which can be obtained via a mild modification of the  proof  of Proposition \ref{nonsplitj'}.
\begin{prop} \label{155}
    Under Assumption \ref{fvan}  and Assumption \ref{asmpe0},  there exists $\overline {f_\infty}\in C_c^\infty(D_\infty^\times/\Xi_\infty)$ such that   
 \begin{equation*} j(f)_\infty= \sum_{\delta\in E^\times\bsl D_\reg  ^\times/E^\times}  \CO( \delta,f ^\infty)   \CO_{\Xi_\infty} (\delta  ,\overline {f_\infty})  .  \end{equation*} 
      \end{prop}

 \section{Arithmetic smooth matching and  arithmetic fundamental lemma }
 \label{amafl}
 Let $F$ be a local field.  
  For  $\Phi\in C^\infty_c(\CS)$ and $x\in F^\times-\{1\}$,
 we have defined the local orbital integral $\CO(s, x,\Phi)$   
  and its derivative $\CO'(0, x,\Phi)$  in \ref{local orbital integrals0}. 
     We want to compare these derivatives with   orbital integrals of the local intersection multiplicity  functions on $G_{\ep'}$    weighted by functions on $G_\ep$ coming from  Definition \ref{iterm1} and \ref{iterm3}. 
In    the application to the global setting, $G_\ep$  will be   $\BB_v^\times$ and  $G_{\ep'}$  will
 be  $B(v)_v^\times$ at  a place $v\in |X| $ nonsplit in $E$.

   \subsection{Compute $\CO'(0, x,\Phi)$ for arithmetic fundamental lemma}
In this and next subsection, we assume that $E/F$, $\psi$  and $\Omega$ (so $\omega$) are 
  unramified,  and   $x\in \varpi \Nm(E^\times)= \inv_{T_\varpi}(G_{\varpi,\reg})$.
 Then $v(x)$ is odd.
     Let $$h_c=\begin{bmatrix}\varpi^c&0\\0&1\end{bmatrix},$$ which is regarded as an element in    $ G_1 \cong  \GL_{2,F}$
or $G\cong \GL_{2,E}$.

\begin{eg}\label{derlocalint} Let $\Phi=1_{K\cap\CS } $ where $K$ is the standard maximal compact subgroup  of $G$.  
 
  If $v(x)>0$,   then $\CO'(0,x,\Phi)=\frac{v(x)+1}{2}(-\log q^2).$
 
If $v(x)<0$,  then $\CO'(0,x,\Phi)=0.$
  
 \end{eg}
 \begin{eg}
 Let $n>0
 $ and be even.   Let $\phi_n$ be the characteristic function of $Kh_n K\cap \CS$.  
 

 If $v(x)<0$, then $\CO'(0,x,\phi_n)=0.$  
 
 If $v(x)>0$, then $\CO'(0,x,\phi_n)= \xi^{n/2}(n+v(x))(-\log q^2).$
 
  
 \end{eg}
 \begin{eg}\label{dergeneralhecke'}
  Let $\Phi$ be the characteristic function of the set of matrices $g\in \CS$ with integral entries such that $v (\det g)=n$. 
  Then
  \begin{align*}\CO'(0,x,\Phi) 
 &= \sum _{0\leq c<n/2}\CO'(0,x,\phi_{n-2c})  \eta\omega^{-1} (\varpi^c)+\CO'(0,x,1_K)\eta\omega^{-1} (\varpi^{n/2}) .
 \end{align*}
 
 If $v(x)<0$, then $\CO'(0,x,\Phi )=0.$
 
   If $v(x)>0$, then $\CO'(0,x,\Phi)= \xi^{n/2}\frac{v(x)+n+1}{2}(-\log q^2).$

 \end{eg}

 \subsection{Arithmetic fundamental lemma}  \label{arithmetic fundamental lemmas}

   Fix  an isomorphism $ G_{ 1}\cong  \GL_{2,F}$      which maps $K_1$ to
$\GL_2(\CO_F)$.
   Then 
 $$  G_1=\coprod_{c\in \BZ_{\geq 0}}   T _1 h_c K_1  .$$
 Let $T_1^\circ\subset T_1\cong E^\times$ be $\CO_E^\times$, then
 \begin{equation}T_1^\circ  h_c K_1 =K _1 h_c K _1. \label{THK}
 \end{equation}
We define   the unramified  multiplicity function   as follows (see Lemma \ref{unramified multiplicity function}).
\begin{defn}   \label{unramified multiplicity function1}

   The unramified multiplicity function $m(\delta,g)$
    on  $G_{\varpi} \times_{T_{\varpi}  \cong T_1 }G _1 $     is nonzero 
 only if $\det (\delta)\det(g)\in \CO_F^\times  $.
 In this case, let $g\in T_1 h_cK_1   $,  then
 \begin{itemize}
 \item[(a)] if $c=0$ then $m(\delta,g)=\frac{v(\inv_{T_{\varpi}}'(\delta))+1}{2}$;
 
 \item[(b)] if $c>0$ then $m(\delta,g)=q^{1-c}(q+1)^{-1}.$

\end{itemize}
\end{defn}  

 The unramified  arithmetic orbital integrals are define as follows (see Definition \ref{iterm1}).
 \begin{defn}
 Let $f\in C_c^\infty(G_1)$.
 For $\delta\in G_\varpi$, define
         $$ i (\delta,f )  := \int_{  T_\varpi/Z_\varpi}\int_{ T_\varpi} \sum _{g \in  G_1 /K_1}f(g)  m ( t_{1 } ^{-1}  \delta t_{2}  ,   g ^{-1} ) \Omega(t_1)\Omega^{-1}(t_2)
        dt_{2 }d t_{1 }  .  $$   \end{defn}
        We have the following  arithmetic fundamental lemma for the full
 spherical Hecke algebra (compare with  \cite{Zha12}, which is only stated  for  the unit   in the spherical Hecke algebra).  
  \begin{prop}[Arithmetic fundamental lemma] \label{AFLgeneral} 
  Let  $f\in C_c^\infty(G)$ be    bi-$K$-invariant, then for all $x\in F^\times-\{1\}$ with $v(x)$ odd, we have $$  i\left(  \delta(x),\bc\left(\frac{\Vol(K_{H_0}) \Vol(K_1)}{\Vol(K)}f\right)\right) \cdot (-2 \log q)=O'(0,x,\Phi_{ f}).$$

    \end{prop}
     By the method in \cite[Section 3]{JLR}, it is enough to prove the following proposition.

     \begin{prop}  \label{AFLgeneral1} Let $n\geq 0$ and be even.
   Let $f$ be the characteristic function of the set of matrices $g\in G_1$ with integral entries such that $v(\det g)=n$, 
  $\Phi$  be the characteristic function of the set of matrices $g\in \CS$ with integral entries such that $v(\det g)=n$.
 Let $x\in F^\times-\{1\}$ with $v(x)$ odd,  then we have $$2  i(  \delta(x), f) \cdot (- \log q)=O'(0,x,\Phi).$$
    \end{prop}

\begin{proof}    

The case $n=0$, i.e. $f=1_{K_1}$, $\Phi=1_{K_1\cap \CS}$.
 By definition $$i( \delta(x),1_{K_1}) =\int_{  T_\varpi/Z_\varpi}\int_{ T_\varpi}
      m ( t_{1} ^{-1}  \delta(x) t_{2}  ,   1)   dt_{2} dt_{1} .$$
 If $v(x)<0$, then  $v(\det \delta(x) )=v(x)$ is odd.  Since $\det t=t\bar t$ is even for every $t\in E^\times$ as $E/F$ is unramified, $\det  (t_{1} ^{-1}  \delta(x) t_{2} ) \not\in \CO_F^\times$. By Definition \ref{unramified multiplicity function1},  $m ( t_{1} ^{-1}  \delta(x) t_{2},1)=0$. Thus $I( 1_{K_1},\delta(x))=0.$
      If $v(x)>0$,  then  $v(\det \delta(x) )=1$. For each $ t_{1}$, we have $$\Vol(\{t_{2}\in T_\varpi :\det ( t_{1} ^{-1}  \delta(x) t_{2} )\in\CO_F^\times\}) =\Vol( \CO_E^\times).$$
 So $$i( \delta(x),1_{K_1}) = \frac{v(x)+1}{2}\Vol(E^\times/F^\times)\Vol( \CO_E^\times) =\frac{v(x)+1}{2} .$$
By Example \ref{derlocalint}, we proved the proposition in this case.

The case $n>0$.
 Let $f'=  1_{\varpi^{(n-c)/2} K_1 h_c K_1}$.  By \eqref{THK} and Definition \ref {unramified multiplicity function1}, we have
 $$i( \delta(x),f')  = [ \varpi^{(n-c)/2} T_1^\circ h_c K_1:K_1]\int_{  T_\varpi/Z_\varpi}\int_{ T_\varpi}
      m ( t_{1} ^{-1}  \delta(x) t_{2}  ,  (\varpi^{(n-c)/2}  h_c)^{-1})\Omega(t_1)\Omega^{-1}(t_2)
        dt_{2 }d t_{1 } .$$
    If $v(x)<0$, $v(\det  t_{1} ^{-1}  \delta(x) t_{2})$ is odd.  Since $n$ is even,  by Definition \ref{unramified multiplicity function1},   $m ( t_{1} ^{-1}  \delta(x) t_{2}  ,  (\varpi^{(n-c)/2}  h_c)^{-1})  =0$.
    Let $v(x)>0$, then for every $t_1,$  we have $$\Vol\{t_2:v(\det  t_{1} ^{-1}  \delta(x) t_{2})=n\}=\Vol( \CO_E^\times).$$ So
if $c=0$, then
    $$i( \delta(x),f')  =\Vol(E^\times/F^\times)\Vol( \CO_E^\times)  \frac{v(x)+1}{2} \xi^{n/2}= \frac{v(x)+1}{2} \xi^{n/2};$$
    if $c>0$,  then
   $$i( \delta(x),f')  =q^{c-1}(1+q) \Vol(E^\times/F^\times)\Vol( \CO_E^\times) q^{1-c}(q+1)^{-1}\xi^{n/2}=\xi^{n/2}.$$  
   Thus   $$i( \delta(x),f)=\sum_{0\leq c\leq n/2} i( 1_{\varpi^{(n-c)/2} K_1 h_c K_1},\delta(x))=  \frac{v(x)+1+n}{2}\xi^{n/2}.$$ 
   By Example \ref{dergeneralhecke'},  we proved the proposition in this case.
  \end{proof}     
    
  \subsection{Compute $\CO'(0, x,\Phi)$  for arithmetic smooth matching}
 We use Lemma \ref{intwell} to compute $\CO'(0,x,\Phi)$.
    Let
$ K_{l,\xi,n}, K_{l,\xi,n}'$ be as in \eqref{orbintK1st} and \eqref{orbintK'}.

   \begin{eg} \label{fepm1'}Let $l, n$ be  large enough such that $\Omega(1+\fp_E^n)=1$ and $\eta(-\tr(\xi)+\tr(\fp_E^l))=\eta(-\tr(\xi))$.
      Then    for  $x\in F^\times-\{1\}$, $\CO'(0,x,1_{ K_{l,\xi,n}\cap \CS})=0$ and  $\CO'(0, x,1_{ K_{ l,\xi,n}'\cap \CS})=0$ unless $v(x)\geq n++v_E(\tr(\xi))$.
      In this case,   we have  $$\CO'(x, 1_{ K_{l,\xi,n}\cap \CS})=\eta(-x\tr(\xi))\Vol(1+\fp_E^{n })\Vol(-\tr(\xi)+\tr(\fp_E^l))( v_E(x)-v_E(\tr(\xi))\log q_E$$    
      and $$  \CO'(0, x,1_{ K_{ l,\xi,n}'\cap \CS})=
      \eta(-\tr(\xi)) \Vol^{\times}(1+\fp_E^{n })\Vol^{\times}(-\tr(\xi)+\tr(\fp_E^l))\cdot  \Omega(-1)(-v_E(\tr(\xi))\log q_E.$$
\end {eg}

\begin{lem}\label{expmat'}For   $f\in C_c^\infty(G)$ given  in  Proposition  \ref{expmat}, we can further require that   for  $x\in \ep'\Nm E^\times-\{1\}$, if $v_E(x)\geq 2 m+v_F(\ep)$, then 
  $$\CO'(0,x, \Phi_f)=\frac{1}{2}\Vol(1+\fp_E^{m })\Vol(E^\times/F^\times) ( - v_E(x)\log q_E),$$
  otherwise $\CO'(0,x, \Phi_f)=0$.
\end{lem}
\begin{proof} Apply 
 the  example above to the explicit constructions in Lemma \ref{822}, \ref{823}.
    \end{proof} 
 \subsection{Arithmetic smooth matching }\label{multiplicity functions and  arithmetic smooth matching}
 Let $\ep,\ep'$ be two the representatives of
 $F^\times/\Nm(E^\times)$ fixed in Assumption \ref{asmpep}. We summarize the  properties of the   multiplicity functions $m(g)=m(g, 1)$ obtained in Lemma \ref{vdetun1},    \ref{U'}, \ref{vdetun"}, \ref{vdetun'},  and \ref{U''}.
               Let  $U $ be an open compact subgroup of $G_\ep$, and  let $G_{\ep',0}:=\{g \in G_{\ep'}: \det (g )\in \det U\}.$
   \begin{defn}  \label{local multiplicity function}

   A special   multiplicity function of level $U$  is a function $m$ on $G_{\ep'}- T_{\ep'}$  supported on 
   $G_{\ep',0}-  T_{\ep'}$  
 satisfying the following conditions:  
     \begin{itemize} 
 \item[(a)] there exists  an open compact subgroup  $U' \subset G_{\ep',0}$  
 such that $U'\cap T_{\ep'}=U\cap T_\ep$ as subgroups of $E^\times\cong T_{\ep'}\cong T_{\ep}$,
and  the function   $$m (  g    ) -\frac{v(\inv_{T_{\ep'}}' (g)) }{2} 1_{U'}(g) $$  on $G_{\ep',0}-  T_\ep$
   can be extended  to a  locally constant function    on $G_{\ep',0}  $;
    \item[(b)] 
     if $\ep'\neq1$,  then $m(g)\neq 0$ only if   $ \det (g) \in \det (U )$;  
       \item[(c)]   if $\ep'=1$, 
          then $m(g)\neq 0$  only if   
       $g\in K_{\ep'} $.  \end{itemize}

   \end{defn}

     We fix   a special multiplicity function $m$ of level $U$. 
     The  arithmetic  orbital integrals of $m$ weighted by $1_U$  are defined as follows (following Definition \ref{iterm1}  and \ref{iterm3}).

     \begin{defn}\label{I(f,delta)} Let $\delta\in G_{\ep,\reg}$. Define
 $$i(  \delta,1_U) :=  \int_{  T_{\ep'}/Z_{\ep'}}\int_{T_{\ep'} }
         m ( t_{1} ^{-1}  \delta t_{2}   )  \Omega^{-1}(t_2) \Omega(t_1)dt_{2}t_{1}. $$
        
         \end{defn}

 \begin{prop}[Arithmetic smooth matching for  $i$-part]\label{smoothmatching}
 
  Let $m$ be a large enough positive integer. Let
 $U=K_{\ep,m}$  and  $ f$ be as in Proposition \ref{expmat} 
 which purely matches $1_U$.
  There exists $  \overline  f \in C_c^\infty(G_{\ep'})$ such that
for $x\in F^\times-\ep\Nm(E^\times)-\{1\}$, the following equation holds:
    \begin{equation*}  i(  \delta(x), 1_U) \cdot (-2\log q_F ) =   \CO'(0, x ,\Phi_f )+\CO (  x, \bar f ).
\end{equation*}

 \end{prop}

  \begin{proof} 
  By  Proposition \ref {Proposition 2.4, Proposition 3.3Jac862},  it is enough to prove the following statements:
 \begin{itemize}
 \item[(1)]   $i(  \delta(x),1_U)   $ is $\Omega (a)$ times  a constant  for $x=\ep a \bar a$ near  $\infty$;
  \item[(2)]   $i(  \delta(x),1_U)=0  $   for  $x$ near 1;

   \item[(3)] $ \CO'(0,\gamma(x),f)  =0$   for $x$ near $\infty$  and   $x$ near 1; 
  
     \item[(4)] $  \CO'(0,\gamma(x),f)-  i(  \delta(x), 1_U) \cdot (-2\log q_F )$     is a  constant  for  $x$ near $0$.
 \end{itemize}
 

  (1) follows from Definition \ref{local multiplicity function} (b) (c) and  Proposition \ref{Proposition 2.4, Proposition 3.3Jac862} (4).
         Here we use the fact that $G_{\ep'}/Z_{\ep'}$ is compact if $\ep'\neq 1$.
 
   (2)   is already indicated in the proof of Proposition \ref{regwell} and \ref{regwell'}. 
 If $\ep=1$, then $1\not \in \ep'  \Nm (E^\times).$ So only need to consider the case $\ep'=1$. By Definition \ref{local multiplicity function} (c), only need to prove that $t_1^{-1}\delta(x)t_2\not \in  K_{\ep'}\cap G_{\ep',0} $ for $x$ near 1. This follows from the facts that
  $\inv_{T_{\ep'}}' (t_1^{-1}\delta(x)t_2)\to \infty$ for $x\to 1$ and $\inv_{T_{\ep'}}' ( K_{\ep'} ) $ is a compact subset of $F$.

       (3) follows from the explicit computations in  Lemma \ref{expmat'}.

  (4) We first compute    $ i(  \delta(x),1_U)$      for  $x$ near $0$. 
   Fix $N$ large enough such that $K_{\ep',N}\subset U'$. Let $v(x)$ large enough      such that $ \delta(x) \in K_{\ep',N}$.     
   Then $t^{-1} \delta(x) t \in K_{\ep',N}$  for every $t\in T_{\ep'} $.
   Since $t^{-1}\delta(x) s=t^{-1} \delta(x) t t^{-1}s$, it is in $U'$ if and only if $t^{-1}s\in U'$, i.e. $s\in t(U'\cap E^\times)$.
   So  $$\Vol\{s\in E^\times:t\delta(x) s\in  U'  \}=\Vol(U'\cap T_{\ep'} )=\Vol(U\cap T_{\ep'} ) =\Vol(1+\fp_E^m).$$
   Also note that for $t^{-1}s\in U'\cap T_{\ep'} $, $\Omega(t^{-1}s)=1$. 
 Since  $v(\inv_{T_{\ep'}}' (\delta(x))) =v(x)$  for  $x$ near $0$, 
 by Definition \ref{local multiplicity function} (a) and  Proposition \ref{Proposition 2.4, Proposition 3.3Jac862} (3), we have   $$i(  \delta(x),1_U)=\Vol (E^\times/F^\times) \Vol(1+\fp_E^m)  \frac{v(x)}{2}  +C$$
 for  $x$ near 0, where $C$ is a constant. Compared with Lemma  \ref{expmat'}, (4) follows.
   \end{proof}
  
  When in the situation of \ref{theinf}, we have a modification of Proposition \ref{smoothmatching1}.
  Suppose $\ep\neq 1$.
   Let   $U= K_{\ep,n}\varpi^\BZ$, $\Xi=U\cap Z_\ep$.  Define the   multiplicity function $m(g)$   by the properties   given in Proposition \ref{regwellinf}.   Define the  arithmetic  orbital integral $i(  \delta,1_U)$ of $m$ weighted by $1_U$ as  in Definition \ref{iterm3}.

 \begin{prop}[Arithmetic smooth matching  for  $i$-part]\label{smoothmatching1}
  Let $m$ be a large enough positive integer. Let
 $U=K_{\ep,m}\varpi^\BZ$. Let $ f$ be as in Proposition \ref{expmat1}. 
  Then $f$ purely matches $1_U$ and 
  there exists $  \overline  f \in C_c^\infty(G_{\ep'}/\Xi)$ such that
for $x\in F^\times-\ep\Nm(E^\times)-\{1\}$, the following equation holds:
    \begin{equation*} i(  \delta(x), 1_U \cdot (- 2\log q _F) =   \CO'(0, x ,\Phi_f )+\CO _\Xi(  x, \bar f ).
\end{equation*}

 \end{prop}

\begin{proof} The proof is similar to the proof of Proposition \ref{smoothmatching}.
  \end{proof}
 
      \section{Proof of Theorem \ref{GZ}}\label{Proof of Theorem}
In Section \ref{Height distributions, Rational Representations, and Abelian varieties}, we have reduced Theorem \ref{GZ} to Theorem
     \ref{GZdis'}.  
  Now we choose $f_v$'s in  Theorem \ref{GZdis'}  as follows.  Our notations are as in the situation of  Theorem \ref{GZ}  and 
     \ref{GZdis'}.  
   
    Let   $S\subset |X| $ be a finite set which contains 
    \begin{itemize}   \item[(1)] all  ramified places  of $\BB$,  $E/F$, $\psi $, or $\pi$  and all 
places below ramified places of $\Omega$;
 \item[(2)]
a set $S_{s,\reg}  \subset |X|_s$ of cardinality  $\geq  2$;
 \item[(3)]
 a set $S_{s,\ave} \subset |X|_s-S_{s,\reg}  $ of cardinality $\geq  2$  over which $\pi $  is  unramified.
\end{itemize}
Here  $|X|_s\subset |X|$ is the subset of places split in $E$. 
 The sets  $S_{s,\reg}$ and  $S_{s,\ave}$ are not necessary disjoint from the set of places in (1).

 
        \begin{lem}\label{splitplacesave} For   $v\in S_{s,\ave}$,  there exists $f_v$ satisfying Assumption \ref{fvan} (2) such that 
$$\alpha^\sharp_{\pi_v}(f_v)\neq 0.$$
   \end{lem} 
   \begin{proof}Let $U_v  =\GL_2(\CO_{F_v})$ and $g=\begin{bmatrix}\varpi_v^2&0 \\ 0&1\end{bmatrix}$.  Let $$f_v=(q_v^2+q_v+1)1_{\varpi _vU_v}-1_{U_vg
   U_v} . $$
Suppose the Satake parameters of $\pi$ are $q_v^{-s_1}$, $q_v^{-s_2}$  where $s_1,s_2$ are purely imaginary,
 and  let $W_0 \in W(\pi_v,\psi_v)$ be $U_v$-invariant.   Then 
   $$\alpha^\sharp_{\pi_v}(f_v)= ((q_v^2+q_v+1)q_v^{-s_1-s_2}-(q_v^{-s_1-s_2}+q_v^{1-2s_1}+q_v^{1-2s_2}))\Vol({U_v})\frac{\lambda_{\pi_v}(W_0)\overline {\lambda_{\pi_v} ( W_0)}}{\pair{W_0,W_0}}.$$   
   Since $q_v^2+q_v> 2q_v$, $\alpha^\sharp_{\pi_v}(f_v)\neq 0$.
     \end{proof} 

For   $v\in S_{s,\ave}$, choose $f_v$  as in Lemma \ref{splitplacesave}.
For    $v\in S_{s,\reg}$,
choose $ f_v$  as in  Lemma \ref{nontrisplit}. Then  $\alpha^\sharp_{\pi_v}(f_v)\neq 0$ and $f$ satisfies  Assumption \ref{fvan}  (1).
   For    $v\in |X|_s\cap S-S_{s,\reg}-S_{s,\ave}$, choose an arbitrary $f_v$  such that $\alpha^\sharp_{\pi_v}(f_v)\neq 0$.
 For   all  $v\in |X|_s\cap S$, choose $f'_v $  as in  Lemma \ref{afneq00} (1).    Then $f'_v$ purely matches $f_v$.
  By Lemma \ref{afneq00} (2),
 $f'$ satisfies Assumption  \ref{freg}.  
 For  $v\in S -|X|_s$, 
 let $f_v=1_{K_{\ep,m}}$ (take $U_v=K_{\ep,m}$) and  $f_v'$    its matching function  in  Proposition \ref{expmat} and \ref{expmat1}.  
  Let $m$ be large enough, such that $\det (K_{\ep,m})\subset \Nm(E_v^\times)$
and 
   $\alpha^\sharp_{\pi_v}(f_v)\neq 0$
 (see Lemma \ref{pipieta}).   
     
\begin{defn}
Functions $f^S\in C_c^\infty((\BB^\times)^S)$ and $f'^S\in  C_c^\infty (G(\BA_E^S))$ are called  matching spherical functions if
 \begin{itemize}
   \item[(1)] for   $v\in |X|-S-|X|_s $,  $ f'_v $ is spherical, and $f_v$  is  the matching function of $f_v'$ on $B_v^\times $ given by   Proposition \ref{FLgeneral};

  \item[(2)] for   $v\in |X|_s-S $,   $ f'_v=(f_{1,v},f_{2,v}) $ is spherical,
     and $f_v$ is the matching function of $f_v'$ on $B_v ^\times$ as in \eqref{splitplaces}.   \end{itemize}
   
\end{defn}
By   the computations in Section \ref{Local intersection multiplicity} (summarized in  Theorem \ref{summary}),   Lemma \ref{decpure},     Proposition \ref {AFLgeneral}, Proposition \ref{smoothmatching} and  \ref{smoothmatching1}, we have the following theorem.

    \begin{thm} [Arithmetic relative trace formula  identity]  \label{jacrtf'} Let $f_S$, $f'_S$ be as above.    There exists  $\overline {f_v}\in C_c^\infty(B(v)_v^\times)$ for $v\in S-|X|_s$ and $\overline {f_\infty}\in C_c^\infty(D_\infty^\times/\Xi_\infty)$,     such that for every pair of  matching spherical functions $f^S$ and $f'^S$,  the following equation holds:
     $$2 H(f) =\CO' (0,f' )+\sum_{v\in S-|X|_s  } \CO_{\Xi_\infty}(f^v\overline {f_v}).$$
          \end{thm}
         Here   $\CO_{\Xi_\infty}(f^v\overline {f_v})$ is as in  Definition \ref{COXi} for $B=B(v)$.
 Now we prove Theorem \ref{GZdis'}.
\begin{proof}[Proof of Theorem \ref{GZdis'}] 
    Let   $\sigma$ be   the base-change  of 
  $\pi$ to $G=\GL_{2,E}$.
 
 Suppose that  the Jacquet-Langlands correspondence of  $\pi$  to $\GL_{2,F}$ is not of the form $\pi_\xi$   as in the end of  \ref{Local-global decomposition of  periods 2}.
 Then $\pi\not \cong  \pi\otimes\eta$,  and  $\sigma$ is cuspidal. 
 By the   condition  $\Ram=\Sigma(\pi,\Omega)$ and Theorem \ref{TSlocal}, 
 the toric period associated to the Jacquet-Langlands correspondence of $\pi$ to $B(v)^\times$ is 0 for every $v\in S-|X|_s $.
Thus by  \eqref{COSigma}, \eqref{COSigma'},       and Corollary \ref{Eisterm'}, we have the following spectral decomposition of the arithmetic relative trace formula   identity
 $$\CO' _\sigma (0,  f'_S  f'^S )=2(H_\pi(f_Sf^S)+H_{\pi\otimes \eta}(f_Sf^S)).$$
 If   $f $ is supported on $\{g\in  \BB^\times :\det (g)\in \Nm (\BA_E^\times )\}$, then $H_\pi(f_Sf^S)=H_{\pi\otimes \eta}(f_Sf^S)$. So
     \begin{align}\CO'_\sigma (0,  f'_Sf'^S)= 4H_\pi(f_Sf^S) .\label{COHpi}\end{align}
     
Suppose that the  Jacquet-Langlands correspondence of $\pi$   to $\GL_{2,F}$  is the  representation $\pi_\xi$ as in the end of \ref{Local-global decomposition of  periods 2}.
Then   $\pi \cong  \pi\otimes\eta$, and    $\sigma =\sigma_\xi $ which is defined in   \ref{Local-global decomposition of  periods 2}.
 By the  condition  $\Ram=\Sigma(\pi,\Omega)$ and Theorem \ref{TSlocal}, 
 the toric period associated to the Jacquet-Langlands correspondence of $\pi$ to $B(v)^\times$ is 0 for every $v\in S-|X|_s $.
Thus by  \eqref{COSigma}, \eqref{COSigma'},        and Corollary \ref{Eisterm'}, we have the following spectral decomposition of the arithmetic relative trace formula  identity \begin{align} 2H_\pi(f)= \CO_{\sigma_\xi}'(0,f')+ \CO_{\sigma_{{\xi}^{-1}\omega_E^{-1}}}'(0,f')= 2  \CO_\sigma'(0,f')   . \label{COHpid}\end{align}

Now we only need to deduce the equation \eqref{GZdiseq}  in Theorem \ref{GZdis'} from \eqref{COHpi}   and  \eqref{COHpid} in these two cases.
Consider  the first case. Let $$I^\sharp_{ \sigma_v}(s, \cdot):=\frac{L (1,\pi_v,\ad)}{L (1/2+s,\pi_v,\Omega_v)L (2,1_{F_v})}I _{ \sigma_v}(s,\cdot).$$
Choose $$f'^S=\frac{1_{K^S}}{\Vol(K_{H_0}^S)\Vol(K'^S)},\ f^S= \frac{ 1_{K_1^S}}{ \Vol(K'^S)}= \frac{1_{K_1^S}}{ \Vol(K_1^S)} .$$
By Proposition \ref{CO1} and  \eqref{COHpi}, we have 
  \begin{align*}  4H_\pi(f_Sf^S)= \CO_\sigma'(0,f'_Sf'^S)= \frac{d}{d s}|_{s=0}\left(\left(\prod_{v\in S}I^\sharp_{ \sigma_v} ( s, f'_v)\right) L_S(1,\eta)   \frac{L (1/2+s,\pi,\Omega)L (2,1_F)}{L (1,\pi,\ad)}\right)\end{align*} 
   By  the  condition  $\Ram=\Sigma(\pi,\Omega)$ and that $\Ram$ is of odd cardinality, we have $\vep(1/2,\pi,\Omega)=-1$. Thus $L(1/2,\pi,\Omega)=0$, and  
  \begin{align*}  4H_\pi(f) =\left(\prod_{v\in S}I^\sharp_{ \sigma_v} ( 0, f'_v)\right)  L_S(1,\eta)   \frac{L' (1/2 ,\pi,\Omega)L (2,1_F)}{L (1,\pi,\ad)}.\end{align*} 
   By   Proposition \ref{localRTF2} and  \ref{expmat},    the computation of $\vep$ (or $\gamma$)-factor  in
 \cite[2.5]{Tat}, and the fact that the product of the local root numbers of $\eta$  is 1, we have \begin{align} 4 H_\pi(f)=   \frac{L' (1/2 ,\pi,\Omega)L (2,1_F)}{L (1,\pi,\ad)}  \prod_{v\in |X|}    \alpha_{\pi_v}^\sharp (f_v) . \label{4H}\end{align}  

Recall  that
  \begin{equation*}H_\pi (f) =  \frac{  [F^\times\bsl \BA_F^\times/\Xi]}{\Vol(\tilde U/\Xi )\Vol(M_U)}   H_\pi^\sharp (f)\end{equation*}
 (see \eqref{specdecomht0}), where $\Vol(M_U)=\deg L_U. $
Similar to \cite[(4.5.1)]{YZZ},  we have the following computation of the coefficient, which was promised below Corollary \ref{the constant}.
\begin{lem}\label{the constant'} We have the following equation:
$$\frac{\Vol (\tilde U/\Xi )\Vol(M_U)}{ |F^\times\bsl \BA_F^\times/  \Xi|} =4.$$ 
\end{lem} 
Thus
\eqref{GZdiseq} follows from   \eqref{4H}, and Theorem \ref{GZdis'} is proved in this case.
For the second case, only need to replace Proposition \ref{CO1} in the above reasoning by Proposition \ref{CO1d}.   \end{proof}

\begin{proof}[Proof of Lemma \ref{the constant'}]We  follow the proof of \cite[(4.5.1)]{YZZ}.
Since the  number on the left hand side of the equation is independent of the choice of $\tilde U$, we may assume that $U_\infty =\BB_\infty^\times$ and $U\cap D^\times$ is small enough so that it acts on $\Omega$ freely. 
Then $$\frac{\Vol (\tilde U/\Xi )\Vol(M_U)}{ |F^\times\bsl \BA_F^\times/  \Xi|} =\Vol(\BB_\infty^\times/F_\infty^\times)\frac{\Vol (U)\Vol(M_U)}{\Vol(\Xi_U)|F^\times\bsl \BA_{F,\mathrm{f}}^\times/  \Xi_U|} .$$

Suppose $D$ is a division algebra. Let $ D^1 $ (resp. $U^1$) be the subgroup of  $D$  (resp. $U$) of elements of norm 1. 
Then by the formula of Serre \cite{Ser},
the degree $\kappa$  of the canonical bundle of $(U^1\cap D^1)\bsl\Omega_\infty$ is
$$2\frac {q_\infty-1}{\Vol(\SL_2(\CO_{F_\infty}))}\Vol((U^1\cap D^1)\bsl   D_\infty^1).$$ 
By the simply-connectedness of $D^1$ 
 and the strong approximation theorem for   $D^1$, we have $$\Vol((U^1\cap D^1)\bsl   D_\infty^1)\Vol(U^1)=\Vol(D^1\bsl D_\infty^1U^1)=1.$$
Thus $$\kappa=2\frac {q_\infty-1}{\Vol(\SL_2(\CO_{F_\infty}))\Vol(U^1)}.$$
Note that $$\Vol(M_U)=\deg L_U=\kappa |F^\times\bsl \BA_{F,\mathrm{f}}^\times/  \det (U)|,$$ 
 where $|F^\times\bsl \BA_{F,\mathrm{f}}^\times/  \det (U)|$ is the number of geometrically connected component of $M_U$.
Combining the last two equations and the easy  fact that $$\frac{\Vol (U) }{\Vol(\Xi_U)|F^\times\bsl \BA_{F,\mathrm{f}}^\times/  \Xi_U|}=\frac{\Vol (U^1) }{ |F^\times\bsl \BA_{F,\mathrm{f}}^\times/  \det (U)|},$$
we have
$$\frac{\Vol (U)\Vol(M_U)}{\Vol(\Xi_U)|F^\times\bsl \BA_{F,\mathrm{f}}^\times/  \Xi_U|} =\frac{2(q_\infty-1)}{\Vol(\SL_2(\CO_{F_\infty}))}=\frac{2(q_\infty-1)\Vol(\CO_{F_\infty}^\times)}{\Vol(\GL_2(\CO_{F_\infty}))}.
$$ 
Let $\CO_{ \BB_\infty}$ be the maximal order of $\BB_\infty$, then we have
\begin{equation*}\Vol(\BB_\infty^\times/F_\infty^\times)=2\frac{\Vol (\CO_{ \BB_\infty}^\times)}{ \Vol(\CO_{F_\infty}^\times)}.
\end{equation*}
Now the lemma follows from  a direct computation  which says that \begin{equation*}\Vol( \CO_{\BB_\infty}^\times)=
\frac{\Vol(\GL_2(\CO_{F_\infty}))}{q_\infty  -1}. \end{equation*}

If $D$ is the matrix algebra, apply the explicit formula   in \cite[VII, Theorem 5.11]{Gek}.   
\end{proof}


\part{Application}
\section{The Birch and Swinnerton-Dyer conjecture}

We apply  Theorem \ref{GZ}  to  prove the Birch and Swinnerton-Dyer conjecture in the analytic rank 1 case.
Indeed, we allow abelian varieties of $\GL_2$-type and  twists   by characters, see  Theorem \ref{BSDTE} and  Corollary  \ref{BSDT0}.
\subsection{Abelian varieties  of strict $\GL_2$-type}\label{GLT}

         
       Let $A$ be an  abelian variety     over $F$  of strict $\GL_2$-type {\cite[3.2.1]{YZZ}}, i.e.  $K:=\End(A)_\BQ$ is a finite field extension of $\BQ$ of degree $\dim A$ .           Then  $H^1(A_{F^\sep}, \BQ_l)$, $l\neq p$, is an irreducible representation of  $G_F=\Gal(F^\sep/F)$ over  $K\otimes_\BQ  \BQ_l$ of  rank 2. 
       For a continuous character $\chi$ of $G_F\to K^{\chi,\times}$, where $K^\chi$ is a finite extension  of $K$, let $L(s,A, \chi)$ be the twisted $L$-function  valued in        $K^\chi\otimes_\BQ\BC$ defined as in \eqref{ldef0} \eqref{ldef}.
       Then for  all embeddings    $K\incl \BC$, the  corresponding components of $L(s,A, \chi)$ have the same order at $s=1$, which is defined as $\ord _{s=1}L(s,A,\chi)$.
 In other words, $L(s,A, \chi)$ is defined as $P(q^{-s})$ for a rational function $P$,  and $\ord _{s=1}L(s,A,\chi)$ is the   multiplicity of $q^{-1}$ as a   root of $P$. 
  
  We consider the following twisted Birch and Swinnerton-Dyer conjecture.
          \begin{conj}\label{BSD} 
        
  We have      $$\rank_{K^\chi} (A(F^\sep)   \otimes_{\BZ} \chi)^{G_F}=\ord _{s=1}L(s,A,\chi).$$
 
        \end{conj}    
       
         Results of  Tate \cite{Tat1}, Milne \cite{Mil} and Schneider \cite{Sch} imply the following partial result on Conjecture \ref{BSD}.     \begin{thm}\label{BSDL} 
        
  We have      $$\rank_{\BQ} (A(F)   \otimes_{\BZ} \BQ) \leq \ord _{s=1}L(s,A).$$
 
        \end{thm}   
            \begin{rmk}   When $\chi$ is the trivial character,  by the work of Tate \cite{Tat1}, Milne \cite{Mil}, Schneider \cite{Sch}, Kato and  Trihan \cite{KT},        the full BSD conjecture for $A$ (which gives the 
leading term of the $L$-functions of $A$ at s = 1,  see \cite{Tat1} or \cite{KT} for the explicit formulation) follows from Conjecture \ref{BSD}.

 \end{rmk} 
    
     \subsection{Twisted abelian varieties}\label{TAV}
 We assume  that $$\End(A)=\CO_K$$
and $$K^\chi \text{ is generated by the values of }\chi$$ without loss of generality.
(Indeed, it is obvious that Conjecture  \ref{BSD} holds for $A$ and $K^\chi$ if and only if it holds for  an abelian varieties isogenous to 
$A$  over $F$ and a finite extension of $K^\chi$.)

Under these two assumptions, the twist $A^\chi$  of $A$ by $\chi$ is defined as in  \cite{MRS}, which is an abelian variety over $F$ of dimension $[K^\chi:K]\cdot \dim A$ with  a natural action by  $\CO_{K^\chi}$ (see \cite[Corollary 1.7]{MRS}). Moreover,  by \cite[Theorem 2.2]{MRS}, we have 
     \begin{equation} L(s,A^\chi)=L(s,A,\chi).\label{Ltwist}\end{equation}   and 
      \begin{equation}\label{RP}\rank_{K^\chi} \left (A(F^\sep)   \otimes_{\BZ} \chi\right)^{G_F} =\rank_{K^\chi} \left( A^\chi (F^\sep)  \right)^{G_F} .  \end{equation}
 
 
      \subsubsection{Determinant character}  
    Let $\mu_l:G_F \to \BZ_l$ be the $l$-adic cyclotomic character.
       By the class field theory and the Weil-pairing on $A$, we have the following lemma. (The analog over number fields is stated in \cite{Rib})
              \begin{lem}  \label{rder}
              There is a character $\omega:G_F\to K^\times$ of finite  order 
              such that  
              $$\det_{K\otimes_\BQ  \BQ_l}H^1(A_{F^\sep}, \BQ_l)= \omega  \mu_l^{-1}.$$

       \end{lem}  
Thus, by \cite[Theorem 2.2]{MRS}, we have  \begin{equation} \det_{K^\chi \otimes_\BQ  \BQ_l}H^1(A^\chi_{F^\sep}, \BQ_l)= \chi ^2\omega  \mu_l^{-1}.\label{Ctwist}\end{equation}

   
\subsubsection{A direct application of Theorem \ref{GZ}}
 Let $E/F$ be a separable quadratic extension and  $\Omega$  a continuous character of $\Gal(F^\sep/E)$ valued in  a finite extension $K'$ of $K$. 
 
    \begin{thm}\label{BSDTE}
 Assume  that 
    $\Omega|_{\BA_F^\times}=\omega  $ where  $\Omega$ and $\omega$ are regarded as Hecke characters  via   the   reciprocity maps.    
If $\ord _{s=1}L(s,A_E,\Omega)=1$, then the counterpart of Conjecture \ref{BSD}  for  $A_E$ and $\Omega$ holds, i.e.,
  $$\rank_{K'} (A(F^\sep)   \otimes_{\BZ} \Omega)^{\Gal(F^\sep/E)}=1.$$

 \end{thm} 
 \begin{proof}   Since $\ep(1,A_E,\Omega)=-1$, the set  $S$ of places $v$ where the local root numbers are not equal to $\Omega_v(-1)$ has odd cardinality. 
 Let $\BB$ be the incoherent quaternion algebra over $\BF$ which ramification set $S$. 
 Let $\infty\in S$ and let $\varpi_\infty$ be a uniformizer of $F_\infty$. 
 Choose $\chi$ such that $\chi^2\omega(\varpi_\infty)=1$ (see Lemma \ref{chi}). 
 Enlarge $K'$ such that it contains the field $K^\chi$ generated by the values of $\chi$. (The enlargement  of $K'$ does not affect the truth of Theorem \ref{BSDTE}.)
 Instead of $A$ and $\Omega$, we consider $A^\chi$ and $ \chi^{-1}_E \Omega$. 
 Apply  Theorem \ref{GZ}, and  choose test vectors $\phi$ and $\varphi$ such that the right hand side of \eqref{GZeq} in  Theorem \ref{GZ} is nonzero (by  the assumption that     $\Omega|_{\BA_F^\times}=\omega  $).
 Then  Theorem \ref{BSDTE}
 holds  
 by (suitable adaptations of)  Theorem \ref{BSDL}, \eqref{Ltwist},   \eqref{RP} and \eqref{Ctwist}.
  \end{proof}

  \subsection{Theorem without base change} 
 Assume that  there is a place of $F$ at which   $A$ does not have potential good reduction.    
  \begin{thm}\label{BSDT}

  If     
  $\omega=1$ 
 and  $\ord _{s=1}L(s,A)=1$, then $\rank_{K} (A(E)   \otimes_{\BZ} \BQ) =1$, i.e.,   Conjecture \ref{BSD} holds for $\chi=1$.

 \end{thm}
  In particular,  the (untwisted)  Birch and Swinnerton-Dyer conjecture holds for  every elliptic curve  with analytic rank 1.
More precisely, Theorem \ref{BSDT} applies if  the elliptic curve  is not isotrivial (see \cite[11.4.1]{Ulm}).   
 For isotrivial elliptic curves, the  Birch and Swinnerton-Dyer conjecture holds by  a result of Tate \cite{Tat2} (see   \cite[Theorem 12.2]{Ulm2}). 

     By \eqref{Ltwist}, \eqref{RP} and \eqref{Ctwist}, we have the following corollary of Theorem \ref{BSDT}.

         \begin{cor}\label{BSDT0}
 If    $\chi^2=\omega^{-1}$ and  $\ord _{s=1}L(s,A,\chi)=1$, then Conjecture \ref{BSD} holds.
 \end{cor}

 \subsection{Proof of Theorem \ref{BSDT}}
 Let   $\omega=1$ 
 and  $\ord _{s=1}L(s,A)=1$.
 \begin{lem}  The curve $X$ is geometrically connected, i.e., the composition  $\overline{\BF_q} F$ is a field. 
 \end{lem}
 \begin{proof} If the lemma if not true, the morphism $X\to \BF_q$ factors through a   morphism $X\to \BF_{q^n}$ 
 for $n>1$. Then $n$ divides $\ord _{s=1}L(s,A )=1$, which is a contradiction.
   \end{proof}

If we have the nonvanishing of  the $L$-function of a certain  quadratic twist  of $A$ at 1, we may  apply Theorem \ref{BSDTE} directly. 
(Over number fields, the existence of such a quadratic twist is known, see \cite{Wal2}.)
 We only need the following weaker result, which will be proved in the next subsection. 

    \begin{lem} \label{tlem} 
 There exists  a positive integer $n$ and a separable quadratic field extension $E$ of $\BF_{q^n}F$   such that  $$\ord _{s=1}L\left (s,A_{E} \right)=1.$$  
 \end{lem} 

Now we prove  Theorem \ref{BSDT}. Choose $n$ and $E$ as in Lemma \ref{tlem}.
  By \cite[Theorem 4.5]{MRS}, the Weil restriction of $A_E$ to $F$ is isogenous to the direct sum of $A$
  and the twists $A_\sigma$ for all  nontrivial irreducible rational representations $\sigma$   of $\Gal(E/F)$.  
   Since  $\ord _{s=1}L(s,A)=1$,    
   $\ord _{s=1}L(s,A_\sigma)=0$ for every nontrivial $\sigma$. 
   By Lemma \ref{tlem} and Theorem \ref{BSDTE}, $\rank_{\BQ} (A(E)   \otimes_{\BZ} \BQ) =1.$
Then by  Theorem \ref{BSDL}, $\rank_{\BQ} (A( F)   \otimes_{\BZ} \BQ) =1.$

\subsection{Proof of Lemma \ref{tlem}}

Let  $\iota: K^\chi\incl \bar \BQ_l$ be  an embedding, and  let 
 \begin{equation}\rho:=H^1(A _{F^\sep}, \BQ_l)\otimes _{K \otimes \BQ_l,\iota} \bar \BQ_l.\label{riot}\end{equation}
 We   look for
 a positive integer $n$ and a separable quadratic extension $E$ of $\BF_{q^n}F$   such that  $$\ord _{s=1}L\left (s,\rho_{\BF_{q^n}F} \right)=1,\  L\left(1,\rho_{\BF_{q^n}F},\eta  \right)\neq 0,$$  
 where $\eta$ is the quadratic Hecke character of $( \BF_{q^n}F)^\times$ associated to the quadratic extension  $E/\BF_{q^n}F$.

\begin{lem} \label{verygood} Given a finite set $S$ of places and  a  quadratic  character  $\eta_v$ of $F_v^\times$ for each  $v\in S$, there exists  a quadratic Hecke character $\eta_0$ of $ F^\times$ such that 
$\ep(1,\rho,\eta_0)=1 $
and  $\eta_{0,v}=\eta_v$ for $v\in S$.
\end{lem} 
\begin{proof}By the Langlands correspondence, we consider the automorphic side. 
If $p\neq 2$, the lemma follows from (the proof of) \cite[Lemma 41]{Wal2}.  We modify the proof of  \cite[Lemma 41]{Wal2} in the case $p=2$ as follows.
We only need to modify the proof of  \cite[Proposition 16, b)]{Wal2}, which states that for a discrete series representation  $\pi$ of $\PGL_2(C)$, where $C$ is a nonarchimedean local  field, there exists a quadratic character $\chi$  of $C^\times$ such that $\ep(1/2, \pi\otimes \chi)=-\chi(-1)\ep(1/2, \pi) $.
 The proof of \cite[Proposition 16, b)]{Wal2} involves taking sum over all quadratic characters $\chi$  of $C^\times$, which is not valid in characteristic 2.
In  characteristic 2, choose  an open compact subgroup $U$ of $C^\times$ such that the set  $Q$ of all quadratic characters of $C^\times/U$ has cardinality  $|Q|>2$. 
Replacing  the sum over  all quadratic characters $\chi$ in \cite[Proof of Proposition 16,  (1)]{Wal2} by  the 
sum of the characters in $Q$, we get   the characteristic function of $UC^{\times,2}$ multiplied by $|Q|$.
In
the rest of the proof of \cite[Proposition 16, b)]{Wal2}, replace the kernel of the reduced norm map  on the unit group of the  division quaternion 
algebra over $C$ by the preimage of $U$ under the reduced norm map. Then   \cite[Proposition 16, b)]{Wal2} follows.
\end{proof}

\begin{rmk}\label{notgod} In the proof of Lemma \ref{verygood}, we used the    place, say $v$, of $F$ at which     $A$ does not have potential good reduction. Indeed, 
the local Langlands correspondence of $\rho_v$ is   a discrete series representation.
Moreover, this condition is necessary, see \cite[Lemma 41]{Wal2}.
\end{rmk}

  By the argument in \cite[11.3.1]{Ulm1}, there exists a positive integer $N$ such that for every $n$ coprime to $N$,
$\ord _{s=1}L\left(s,\rho_{\BF_{q^n}F} \right)=1$. 
 When $p$ is odd and  certain tame conditions are satisfied,   the lemma follows from the results of Katz on twisted $L$-functions \cite{Kat1} as follows.
Let $V$ be the moduli variety  of  functions on $X$ which satisfy  the local conditions   in \cite[5.0, 6.0, 6.1]{Kat1}. Points in $V(\BF_{q^n})$ give quadratic   Hecke characters  of $ (\BF_{q^n}F)^\times$ corresponding to 
separable quadratic extensions of $\BF_{q^n}F$  defined by the functions via the Kummer extension.   And there is a lisse $\overline{\BQ_l}$-sheaf $\CG$ on $V$ such that the characteristic polynomials of the stalks   give the $L$-functions of quadratic twists of  $\rho_{\BF_{q^n}F}$.
By Lemma \ref{verygood}, there exists a point in  $V(\BF_q)$ corresponding to a quadratic   Hecke character $\eta_0$ of $F^\times$ such that  $\ep(1,\rho,\eta_0)=1 $. Then the monodromy of $\CG$ can only be the first two possibilities in \cite[8.3.2]{Kat1} by    \cite[Theorem 8.3.8]{Kat1} about the emptiness of $X_{{\sign}+}$ in the case of the third possibility. Then  Lemma \ref{tlem} follows from an application of Deligne's equidistribution theorem
\cite[Theorem 8.3.2, Theorem 8.3.6]{Kat1}. 
Without the tame conditions, the monodromy of $\CG$ was computed in  \cite[Theorem 8.2.1]{Kat2} using higher moments. 
 When $p=2$, in   \cite[Chapter 8]{Kat2}, replacing the Kummer sheaf by the Artin-Schreier sheaf, the analog of   \cite[Theorem 8.2.1]{Kat2} holds by the same proof.  Moreover, Lemma \ref{verygood} and  \cite[Theorem 8.3.8]{Kat1}  do not depend on the tameness of $\rho$ or the characteristic  $p$. Thus the same proof carries on to conclude Lemma \ref{tlem}.
 \section*{Conflict of interest}
 On behalf of all authors, the corresponding author states that there is no conflict of interest.
 \section*{Data Availability Statement}
       Data sharing not applicable to this article as no datasets were generated or analysed during the current study.


\begin{thebibliography}{99}
                
                          \bibitem{BP}  Beuzart-Plessis, Rapha\"el. Comparison of local spherical characters and the Ichino-Ikeda conjecture for unitary groups. arXiv preprint arXiv:1602.06538 (2016).
                      \bibitem{BP1}    Beuzart-Plessis, Rapha\"el. Plancherel formula for $\mathrm {GL} _n (F)\backslash\mathrm {GL} _n (E) $ and applications to the Ichino-Ikeda and formal degree conjectures for unitary groups. arXiv preprint arXiv:1812.00047 (2018).
   \bibitem{BS}   Blum, A., and U. Stuhler. Drinfeld modules and elliptic sheaves. Lecture Notes in Mathematics 1649(1997):110-188.
                         \bibitem{Boy}   Boyer, P. Mauvaise r\'{e}duction des vari\'{e}t\'{e}s de Drinfeld et correspondance de Langlands locale. Inventiones Mathematicae 138.3(1999):573-629.
                                          \bibitem{CST}Cai, Li, Jie Shu, and Ye Tian. Explicit Gross-Zagier and Waldspurger formulae. Algebra $\&$ Number Theory 8.10 (2014): 2523-2572.
              
                         \bibitem{Car}Carayol, Henri. Sur la mauvaise r\'eduction des courbes de Shimura. Compositio Mathematica 59.2 (1986): 151-230.      
                          \bibitem{Cas} Casselman, William. On some results of Atkin and Lehner. Mathematische Annalen 201.4 (1973): 301-314.
                        \bibitem{HC} Chandra, Harish, and Gerrit van Dijk. Harmonic analysis on reductive p-adic groups. Springer, 1970.
   
                      \bibitem{CW}Chuang, Chih-Yun, and Fu-Tsun Wei. Waldspurger formula over function fields. Transactions of the American Mathematical Society 371.1 (2019): 173-198.       
                       \bibitem{Del}   Deligne, Pierre. La conjecture de Weil. II. Publications Math\'ematiques de l'Institut des Hautes Etudes Scientifiques 52.1 (1980): 137-252.
                                      \bibitem{DriEll1}   Drinfel'd, Vladimir G. Elliptic modules. Mathematics of the USSR-Sbornik 23.4 (1974): 561.
                       
\bibitem{DriEll2}Drinfel'd, Vladimir G. Elliptic modules. II. Mathematics of the USSR-Sbornik 31.2 (1977): 159.
\bibitem{Dridomain}Drinfel'd, Vladimir G. Coverings of p-adic symmetric domains. Functional Analysis and its Applications 10.2 (1976): 29-40.
\bibitem{DriCar}Drinfel'd, Vladimir G.: Letter to H. Carayol (January 12th, 1980).
 \bibitem{Dri3} Drinfel'd, Vladimir G. Cohomology of compactified moduli varieties of $F$-sheaves of rank 2. Zapiski Nauchnykh Seminarov POMI 162 (1987): 107-158.

     \bibitem{Fli} Flicker, Yuval Z. Transfer of orbital integrals and division algebras.
J. Ramanujan Math. Soc. 5 (1990), no. 2, 107-121.

\bibitem{GGP}Gan, Wee Teck, B. H. Gross, and D. Prasad. Symplectic local root numbers, central critical L-values, and restriction problems in the representation theory of classical groups. Asterisque 13.346(2009):171-312.

\bibitem{Gek}Gekeler, Ernst Ulrich. Drinfeld modular curves.   Springer-Verlag, 1986.
\bibitem{Gen}Genestier, Alain. Espaces sym\'{e}triques de Drinfeld. Ast\'{e}risque (1996).
\bibitem{Gro} Gross, Benedict H. On canonical and quasi-canonical liftings. Inventiones mathematicae 84.2 (1986): 321-326.

\bibitem{GZ}Gross, Benedict H., and Don B. Zagier. Heegner points and derivatives of $L$-sequence. Inventiones mathematicae 84.2 (1986): 225-320.
 \bibitem{Guo} Guo, Jiandong. On the positivity of the central critical values of automorphic $L$-functions for GL (2). Duke Mathematical Journal 83.1 (1996): 157-190.



\bibitem{Hau}Hausberger, Thomas. Uniformisation des vari\'et\'es de Laumon-Rapoport-Stuhler et conjecture de Drinfeld-Carayol."Annales de l'institut Fourier. Vol. 55. No. 4. 2005.
\bibitem{Hed}Hedayatzadeh, S. Exterior powers of Lubin-Tate groups. Journal de th\'eorie des nombres de Bordeaux 27.1 (2015): 77-148.
\bibitem{Jac86} Jacquet, Herv\'e. Sur un r\'esultat de Waldspurger. Annales scientifiques de l'\'Ecole Normale Sup\'erieure. Vol. 19. No. 2. 1986.
 \bibitem{Jac87}Jacquet, Herv\'{e}. Sur un r\'esultat de Waldspurger. II. Compositio Mathematica 63.3 (1987): 315-389.


  
 \bibitem{JLR} Jacquet, Herv\'{e}, King F. Lai, and Stephen Rallis. A trace formula for symmetric spaces. Duke Mathematical Journal 70.2 (1993): 305-372.
 \bibitem{JL}Jacquet, Herv\'{e}, and Robert P. Langlands. Automorphic forms on GL (2). Vol. 114. Springer, 2006.

\bibitem{JN} Jacquet, Herv\'{e}, and Chen Nan. Positivity of quadratic base change $ L $-functions. Bulletin de la Soci\'e\'e math\'ematique de France 129.1 (2001): 33-90.
\bibitem{KT} Kato, Kazuya, and Fabien Trihan. On the conjectures of Birch and Swinnerton-Dyer in characteristic $p> 0$." Inventiones mathematicae 153.3 (2003): 537-592.
\bibitem{Kat1}Katz, Nicholas M. Twisted L-Functions and Monodromy.(AM-150). Vol. 150. Princeton University Press, 2009.
\bibitem{Kat2}Katz, Nicholas M. Moments, monodromy, and perversity: a Diophantine perspective. No. 159. Princeton University Press, 2005.
\bibitem{LL97}Lafforgue, Laurent. Chtoucas de Drinfeld et conjecture de Ramanujan-Petersson. Ast\'{e}risque (1997).
\bibitem{LL} Lafforgue, Laurent. Chtoucas de Drinfeld et correspondance de Langlands. Inventiones mathematicae 147.1 (2002): 1-241.
\bibitem{Lan} Langlands, Robert P. Base change for $\GL (2)$. No. 96. Princeton University Press, 1980.
         \bibitem{LLS}     Laumon, Gerard, Michael Rapoport, and Ulrich Stuhler. $\cD$-elliptic sheaves and the Langlands correspondence. Inventiones mathematicae 113.1 (1993): 217-338.
         \bibitem{MW}Martin, Kimball, and D. Whitehouse. Central $L$-values and toric periods for $\GL(2)$. International Mathematics Research Notices 2009.1(2009):141.
\bibitem{MRS} Mazur, 
Barry, Karl Rubin, and Alice Silverberg. Twisting commutative algebraic groups. Journal of Algebra 314.1 (2007): 419-438.
\bibitem{Mil} Milne, James S. On a conjecture of Artin and Tate. Annals of Mathematics (1975): 517-533.

\bibitem{MWa}Moeglin, Colette, and Jean-Loup Waldspurger. Spectral Decomposition and Eisenstein Series: A Paraphrase of the Scriptures. Vol. 113. Cambridge University Press, 1995.

\bibitem{MFK}Mumford, David, John Fogarty, and Frances Kirwan. Geometric invariant theory. Vol. 34. Springer Science $\&$ Business Media, 1994.
  \bibitem{RSZ15}Rapoport, Michael, B. Smithling, and W. Zhang. On the arithmetic transfer conjecture for exotic smooth formal moduli spaces. Duke Mathematical Journal 166(2015):p\'{a}gs. 2183-2336.
\bibitem{RSZ16}Rapoport, Michael, B. Smithling, and W. Zhang. Regular formal moduli spaces and arithmetic transfer conjectures. Mathematische Annalen 3(2016):1-97.
\bibitem{RSZ17}Rapoport, Michael, Brian Smithling, and Wei Zhang. Arithmetic diagonal cycles on unitary Shimura varieties. arXiv preprint arXiv:1710.06962 (2017).


\bibitem{Rei}Reiner, Irving. Maximal orders. New York-London (1975).
\bibitem{Rib}Ribet, Kenneth A. Abelian varieties over $\BQ$ and modular forms. Modular curves and abelian varieties. Birkh\"auser, Basel, 2004. 241-261.
\bibitem{RT}R\"{u}ck, Hans-Georg, and Ulrich Tipp. Heegner points and $L$-sequence of automorphic cusp forms of Drinfeld type. Documenta Mathematica 5 (2000): 365-444.

 \bibitem{Sai}Saito, Hiroshi. On Tunnell's formula for characters of GL (2). Compositio Mathematica 85.1 (1993): 99-108. 
 \bibitem{Sch}Schmidt, Ralf. Some remarks on local newforms for GL (2). Journal of the Ramanujan Mathematical Society 17.2 (2002): 115-147.
\bibitem{Schne} Schneider, Peter. Zur Vermutung von Birch und Swinnerton-Dyer \"uber globalen Funktionenk\"orpern. Mathematische Annalen 260.4 (1982): 495-510.
\bibitem{Ser}Serre, Jean-Pierre. Cohomologie des groupes discrets. S\'{e}minaire Bourbaki vol. 1970/71 Expos\'es 382-399. Springer, Berlin, Heidelberg, 1971. 337-350.
     \bibitem{Spi}      Spiess, Michael. Twists of Drinfeld-Stuhler modular varieties. Documenta Mathematics (2010): 595-654.

\bibitem{Str}Strauch, Matthias. Geometrically connected components of Lubin-Tate deformation spaces with level structures.  Pure and Applied Mathematics Quarterly 4 (2008), no. 4, part 1, 1215-1232
        \bibitem{Tae2}   Taelman, Lenny. $ D $-Elliptic Sheaves and Uniformisation. arXiv preprint math/0512170 (2005).
        
             \bibitem{Tat1}  Tate, John. On the conjectures of Birch and Swinnerton-Dyer and a geometric analog. S\'eminaire Bourbaki 9.306 (1965): 415-440.
                         \bibitem{Tat2}  Tate, John. Endomorphisms of abelian varieties over finite fields. Inventiones mathematicae 2.2 (1966): 134-144.
                            \bibitem{Tat}Tate, J. T. Fourier analysis in number fields, and Hecke's zeta-functions, in 1967 Algebraic Number Theory (Proc. Instructional Conf., Brighton, 1965). Washington, DC.
                           
\bibitem{Tipp} Tipp, Ulrich. Green's functions for Drinfeld modular curves. Journal of Number Theory 77.2 (1999): 171-199.
      \bibitem{Tun}       Tunnell, Jerrold B. Local $\ep$-factors and characters of $\GL (2)$. American Journal of Mathematics 105.6 (1983): 1277-1307.  


\bibitem{Ulm}  Ulmer, Douglas. Elliptic curves and analogies between number fields and function fields. Heegner points and Rankin $L$-series 49 (2004): 285-315.
\bibitem{Ulm1} Ulmer, Douglas. Geometric non-vanishing. Inventiones mathematicae 159.1 (2005): 133-186.
\bibitem{Ulm2}Ulmer, Douglas. Elliptic curves over function fields. Arithmetic of L-functions 18 (2011): 211-280.

\bibitem{Wal}Waldspurger, Jean-Loup. Sur les valeurs de certaines fonctions L automorphes en leur centre de sym\'etrie. Compositio Mathematica 54.2 (1985): 173-242.
\bibitem{Wal2}
Waldspurger, Jean-Loup. Correspondances de Shimura et quaternions. Forum Mathematicum. Vol. 3. No. 3. Walter de Gruyter, Berlin/New York, 1991.
\bibitem{WFT}Wei, Fu-Tsun. On the Siegel-Weil formula over function fields. Asian Journal of Mathematics 19.3 (2015): 487-526.
\bibitem{YZZ}Yuan, Xinyi, Shouwu Zhang, and Wei Zhang. The Gross-Zagier formula on Shimura curves. No. 184. Princeton University Press, 2013.

\bibitem{YZ}Yun, Zhiwei, and Wei Zhang. Shtukas and the Taylor expansion of L-functions. Annals of Mathematics (2017): 767-911.
\bibitem{YZ2}Yun, Zhiwei, and Wei Zhang. Shtukas and the Taylor expansion of L-functions (II). Annals of Mathematics 189.2 (2019): 393-52

\bibitem{Zha1} Zhang, Shou-Wu. Heights of Heegner points on Shimura curves. Annals of mathematics 153.1 (2001): 27-147.
\bibitem{Zha01}Zhang, Shou-Wu. Gross-Zagier Formula for $GL_2$. Asian Journal of Mathematics 5 (2001): 183-290.
 \bibitem{Zha10} Zhang, Shou-Wu. Linear forms, algebraic cycles, and derivatives of $L$-series. Science China Mathematics 62.11 (2019): 2401-2408.
\bibitem{Zha12}Zhang, Wei. On arithmetic fundamental lemmas. Inventiones mathematicae 188.1 (2012): 197-252. 
\bibitem{Zha14}Zhang, Wei. Fourier transform and the global Gan-Gross-Prasad conjecture for unitary groups. Annals of Mathematics 180.3(2014):971-1049.
\bibitem{Zha14b}Zhang, Wei. Automorphic period and the central value of Rankin-Selberg $L$-function. Journal of the American Mathematical Society 27.2(2014):541-612.
   \bibitem{Zha19} Zhang, Wei. Weil representation and arithmetic fundamental lemma. arXiv preprint arXiv:1909.02697 (2019). 
           
\end{thebibliography}
 \end{document}